\documentclass[11pt]{article}

\usepackage[a4paper, total={16cm, 24cm}]{geometry}

\usepackage[english]{babel}
\usepackage{amsmath, amsthm, amssymb }
\usepackage[utf8]{inputenc}
\usepackage{hyperref}
\usepackage[T1]{fontenc} %Pour les guillemets 
\usepackage{indentfirst}
\usepackage[sort&compress,numbers,square]{natbib}
\usepackage{graphicx, color}
\usepackage[dvipsnames]{xcolor}
 \usepackage{float} 
\theoremstyle{plain}
\theoremstyle{definition}

\newtheorem{theorem}{Theorem}[section]
\newtheorem{corollary}{Corollary}[theorem]
\newtheorem{lemma}[theorem]{Lemma}
\newtheorem{proposition}[theorem]{Proposition}

\newcommand{\tr}{{}^\top\!}

\newcommand{\bu}{\bullet}

\newcommand{\R}{\mathbb{R}}
\newcommand{\N}{\mathbb{N}}

\newcommand{\E}{\mathbf{E}}
\title{Functional Laplace Transform of a Multivariate Hawkes Process, Subsequent Characteristics, and Numerical Approximations}

\author{ Bartholomé Vieille\footnote{INRAE, BioSP, 84914 Avignon, France}  \footnote{\url{bartholv@gmail.com}} \and Rachid Senoussi\footnotemark[1] \footnote{\url{rachid.senoussi@inrae.fr}} \and Samuel Soubeyrand\footnotemark[1] \footnote{\url{samuel.soubeyrand@inrae.fr}} \footnote{\url{https://orcid.org/0000-0003-2447-3067}}}

\date{\today}

\begin{document}
%%%%%%%%%%%%%%%%%%%%%%%%%%%%%%%%%%%%%%%%%%%%%%%%%%%%%%%

\maketitle

\paragraph{Abstract.}
Numerous studies grounded on Hawkes processes have been carried out in many fields including finance, biology and social network. Hawkes processes form a class of self-exciting simple point processes. In this article, we consider a general class of multivariate Hawkes processes envisioned to model dynamics of spatio-temporal epidemics. For this class, the igniting baseline intensity is time dependent and the exciting matrix function is a general one, making the model non-Markovian in most of the cases. 
In this article, we first provide the closed-form expression of the multivariate multi-temporal characteristic function of these Hawkes processes, extending in a natural way the classical single-time formula found in the Hawkes literature. Then, we use the infinitely divisible property of the Hawkes process to derive the equation system related to the probability distribution of counts at each single time, adapted to the general formulation of the Hawkes model considered in this article. Next, we provide closed-form formulas for the temporal structure of the two first moments of the process, which allows us to deduce an original expression of the multivariate covariance function at two distinct times, thereby extending existing results established for more restricted classes of Hawkes processes. Based on this expression, we analytically decompose the covariance at two distinct times into  singular and continuous parts.
We finish with brief numerical elements: We present a simple scheme for numerical approximations of the Laplace transform and the first two moments, and give examples of solutions of the different related integral equations. We also provides illustrative simulations of the multivariate Hawkes process for different model specifications.


\paragraph{MSC2020 subject classifications.} 60G55; 60E10.

\newpage
	
\section{Introduction}
	%\label{sec:intro}

A univariate  simple counting process $N$ is defined by:

\begin{equation*}
    N(t)=\sum_{l=1}^{+\infty} \mathbf 1_{\{T_l\le t\}},
\end{equation*}

\noindent where $T_1<\ldots<T_l<\ldots$ is a sequence of positive increasing random variables, and $\mathbf 1_E$ equals 1 if event $E$ is true and 0 otherwise. 

As usual, to describe the stochastic structure and the dynamics of $N$, we consider a complete probability space $(\Omega,\mathcal F,\mathbf P)$ with the  complete right-continuous natural filtration 
$\mathcal F=(\mathcal F_t)_{t \geq 0}$, with $\mathcal F_t=\sigma(N(s), s\leq t)$ . 

$(N(t))_{t\geq 0}$ is said a \textit{Hawkes} process driven by $(\lambda^0,\phi)$ if its $ (\mathcal{F}-\mathbf P)$ previsible compensator admits an intensity $\tilde{\lambda}^0(t)$   of the following form:

\begin{equation}
\label{1D-intensity}
\tilde{\lambda}^0(t) =\lambda^0(t) + \sum_{l=1}^{+\infty} \phi(t-T_l)\; \mathbf 1_{\{T_l\le t\}}.
\end{equation}

\noindent where $\lambda^0$, called the baseline intensity, and the function $\phi$, called the excitation function, are deterministic non-negative real functions defined on $\mathbb R_+$. This process is said self-exciting since each event $T_l$ induces a new point process with baseline intensity function $\phi$.

Hawkes processes were introduced by Alan G. Hawkes \cite{hawkes1971spectra}. They are a generalization of Poisson processes, wherein multiple Poisson processes are added over time. Hawkes processes have been used in various fields such as finance \cite{bacry2015hawkes,hawkes2018hawkes}, genomics \cite{reynaud2010adaptive}, epidemiology \cite{rizoiu2018sir}, insurance \cite{SWISHCHUK2021107}, neuroscience \cite{Gerencser,lambert2018reconstructing}, and seismology \cite{BENALLAL_RM2020}.

From the intensity expression \eqref{1D-intensity}, it is evident that Hawkes processes are usually non-Markovian and therefore challenging to investigate. Consequently, a significant portion of research on Hawkes processes focuses on specific models. For instance, extensive studies have been conducted on the case where the baseline intensity $\lambda^0$ is constant and the exciting function $\phi$ is exponential, in which case the analysis simplifies into the analysis of a two-dimensional Markovian process $(\tilde{\lambda}^0, N)$; see \cite{gao2018functional,hawkes1971spectra,oakes1975markovian,seol2019limit}.

The integral formula for the probability generating function restricted to a single time for the one-dimensional case was provided by \cite{hawkes1974cluster}. This result was further extended to Hawkes processes with random jump sizes, as established by \cite{gao2018transform} who derived the expression for the Laplace Transform for the one-dimensional case. Moreover, \cite{errais2010affine} investigated the Laplace and Fourier transforms and the distribution of a Hawkes process in a specific case where the exciting function is exponential and with a particular baseline intensity. Furthermore, using martingale methods, \citep{jaisson2014} provided second-, third- and fourth-order moments at multiple times for one-dimensional Hawkes processes with constant baseline intensity, these results being used in the study of weak convergence of Hawkes processes toward Cox–Ingersoll–Ross models \citep{jaisson2015}; See also \citep{hillairet2023}, who employed the Poisson embedding representation, Malliavin calculus, and the pseudo-chaotic expansion of the Hawkes process to derive expressions for second-order moments in the one-dimensional case with constant baseline intensity, as above.

A multivariate version of a Hawkes process was proposed by \cite{daley2003introduction}, and referred to as a marked temporal process $N$ with a finite mark set denoted by $\{1,...,d\}$. Subsequently, \cite{el2019characteristic} obtained an integral formula for the characteristic function of the multivariate case, limited to a single time $t$; see also \cite{karim2021}. The extension of these formulas to multiple times is crucial for statistical inference if data include multiple observation times and for identifying process limits. Precisely, using a population representation, \citep{jovanovic2015} obtained simple expressions of second- and third-order moments at multiple times for multivariate Hawkes processes that are stationary over $\R$, i.e., grounded on a constant baseline intensity and a short-time dependence assumption on the excitation function. \citep{privault2021} provided implicit forms of functional moments generalizing the results of \citep{jovanovic2015} to stationary spatial Hawkes processes. These forms are derived from the expression of the probability generating functional (Proposition 3.1 in \citep{privault2021}), which is analogous to our Theorem \ref{thrm:2} obtained in the non-stationary case. 

In this study, we consider a more general class of non-stationary, multidimensional Hawkes processes, characterized by a flexible and relatively broad excitation function. We characterize the multivariate multi-temporal Laplace transform and moments using integral equations. To achieve this, we use the immigrant-birth representation of a linear Hawkes Process, as described in \cite{hawkes1974cluster} and \cite{karabash2015limit} for marked Hawkes processes. This result not only characterizes the process but also provides intrinsic information about its behavior over time.

%\citet{jaisson2015} studied weak convergence limits of Hawkes processes toward Cox–Ingersoll–Ross models. They obtained these results by deriving, in the uni-dimensional case, second- third- and fourth-order moments by using martingale methods \citep{jaisson:hal-00935038} (in our case, we used the classical approach consisting of using partial derivatives of the Laplace transform). 

%\citet{jovanovic2015} obtained simple expressions of second- and third-order moments in the multidimensional and multitime case, in the particular case of Hawkes processes that are stationary over $\R$, i.e., grounded on a constant baseline intensity function and a short-time dependence assumption on the excitation function. They obtained these results via a population representation.

%\citet{privault2021} generalizes the case of \citep{jovanovic2015} by providing implicit forms of functional moments of any order, in the case of stationary spatial Hawkes processes. His approach is grounded on his Proposition 3.1, which is analogue to our Theorem 2.4 (in the non-stationary case for us). 

%Hillairet 2023 moment 1 et 2 en unidimensionnel. Alternative approach based on methodology relies on the Poisson imbedding representation and on recent findings on Malliavin calculus and pseudo-chaotic representation for counting processes.

In Section \ref{matmet}, we present a Volterra-like system of equations expressing the multi-temporal characteristic function of our multivariate Hawkes process. A similar expression for the functional Laplace Transform of the process is also given, extending classical single-time expressions. 
In Section \ref{section_inf_div}, we detail the precise form of the infinitely divisible property for our Hawkes process and the integral equations related to the probability distribution  of the Hawkes process at any time. 
Section \ref{section_int_equation_first_moments} is dedicated to expressing and calculating the first two moments of the process as well as the form of certain statistical functionals related to the general theory of spatial statistics. We obtain an original expression for the covariance function at two distinct times. Then, we propose in Section \ref{section-approximation}  some approaches to tackle the numerical aspects for the resolution of the system of Volterra-like equations mentioned above. We illustrate these numerical approaches for a specific Hawkes model, focusing on the mean and covariance functions. Section \ref{sec:simul} gives a few simulated examples of the multi-dimensional Hawkes process to illustrate the effects of the excitation function and the interaction between the different dimensions (i.e., regions in a spatial setting). To conclude, Section \ref{conclusion} outlines the general framework of this work and mentions some avenues of  potential applications for the results that we obtained, especially with regard to the development of statistical inference methods.

\subsubsection*{Notations.}	
First, we convene that $j\in [r:k]$ holds for the  enumeration $j=r,\cdots,k$.
In accordance with matrix calculus, a $d$-vector $w$ is always considered as a column vector while its transpose is the row vector $\tr w=(w_1,\ldots,w_d)$,   and for a $d\times n$-matrix $m=(m_i^j: i= 1:d, j = 1: n)$,  its $j^{th}$ column is  denoted $m_\bu^j$ and its $i^{th}$ row as $m_i^\bu $. We also set $\overline{w}=\sum_{i=1}^d w_i$ and  adopt the useful notation of restriction $u_{[r:k]}=(u_r,\ldots,u_k)$ for an $n$-row vector $u$ as well as  $m_\bu^{[r:k]}= (m_i^j:  1\leq i\leq d, r\leq j\leq k)$ for matrix $m$.  We also convene that a $k$-vector  having identical components, typically a scalar $u$, is written $u_{[k]}$, e.g. $0_{[d]}$ denotes the vector of zeros with dimension $d$.  The left limit of a function $f$  is defined by $f^-(t)= \lim_{s\uparrow t} f(s)$ if it exists (as this is the case for non decreasing functions), and its $\tau-$translate is defined by $f_{\tau_+}(s)=f(\tau+ s)$.   Multiple summation over $q=(q_1,\ldots,q_d) \in \mathbb{N}^d$ is simply written $\sum_q$.  The unit vectors of $\R^d$ are denoted $e_j$, $j=1:d$.

\bigskip
\section{Multivariate Hawkes Process}  \label{matmet}

A $d$-dimensional Hawkes process driven by $(\lambda^0, \phi)$ is a simple point process with $d$ marks
$$ N(t) =\tr(N_1,\ldots,N_d) (t)$$
described by the collection of the occurrence times $\{T^j_l: j=1:d, l \geq 1\}$
(indicating that an $l^{th}$ event of type $j$ occurs at $T^j_l$) and 
such that considering its natural filtration, $N$ has a predictable compensator 
$\tilde{\Lambda}^0(t)= \tr(\tilde{\Lambda}^0_1,\ldots, \tilde{\Lambda}^0_d)(t)$ admitting a conditional intensity function $\tilde{\lambda}^0(t)= \tr(\tilde{\lambda}^0_1,\ldots, \tilde{\lambda}^0_d)(t)$  of the following form:  
\begin{equation}\label{intensity2-1} 
	\tilde{\lambda}^0_i(t)= \lambda_i^0(t) +
	\sum_{j=1}^d \sum_{l: T^j_l <t} \phi^j_i (t-T^j_l), 
\end{equation}
where $\lambda^0 : \mathbb R_+ \rightarrow (\mathbb R_+)^d$ and $\phi=(\phi^j_i) : \mathbb R_+ \rightarrow \mathcal M_d^+$, being  the set of $d\times d$ matrices with non-negative coefficients. 

This can be rewritten in integral-matrix forms as: 
\begin{equation}\label{intensity2-2} 
\tilde{\lambda}^0(t)= \lambda^0(t) +
    \int_0^{t-} \phi(t-u)\,\mathrm d   N(u).
\end{equation}

$N$ can represent various counting contexts involving interacting entities, such as populations (cities, countries, animal species, cells), finance and economic activities, or physical fundamental particles, types, or locations of earthquakes, etc. For the sake of clarity, in epidemiology, we assume that it describes the evolution of an infection in a spatial domain composed of $d$ distinct regions. Therefore, $N_i(t)$ represents the number of infected individuals in region $i$ at time $t$.

\

Equation \eqref{intensity2-1} describes the dynamics of the process, indicating the presence of an initial unknown ancestor (labeled as $j=0$) that triggers a $d$-dimensional Poisson process $N^0=\tr(N^0_1,\ldots,N^0_d)$ with independent components and conditional intensity $\lambda^0=\tr(\lambda^0_1,\ldots,\lambda^0_d) $. This process is described by occurrence times $T^{0,j}_l,$ where $ j=1:d$ and $l\geq 1$,  representing individuals of the first generation.
Each individual of the first generation of type $j$, occurring at time $T^{0,j}_l$, in turn, triggers a new $d$-dimensional Poisson process with conditional intensity $\phi^j_\bu=\tr(\phi^j_1,\cdots,\phi^j_d)$, resulting in a second generation of individuals ($l'\geq 1)$ of type $j'$ at times 
$(T^{1,j,j'}_{l,l'}: j'=1:d, l'\geq 1)$. This process continues recursively, generating individuals of the third generation and so on. 
The occurrence times $\{ T^{m,j,j'}_{l,l'}: \  j,j'\in 1:d; m\geq 0; l,l'\geq 1\}$ when ordered increasingly in a unique sequence $\{T^j_n: j=1,\ldots, d; n \geq 1\}$, actually form the Hawkes process $N$. 
\bigskip

\paragraph{Remark and Notation.} It is well known for multivariate point processes, see \cite{jacod1} for example, that when  considering the natural filtration, the knowledge of the conditional intensity, parameterized here by $(\lambda^0, \phi)$, characterizes entirely the probability distribution of the process itself. Accordingly, when clarification is needed, we enhance this by using subscripts, eg  $\mathbf{P}_{(\lambda^0,\phi)}, \mathbf{E}_{(\lambda^0,\phi)}$.  Also, by replacing  $\lambda^0$  with $\phi^{j'}_\bu,  j'=1:d$, we can define $d$-new Hawkes processes $\tilde{N}^{j'}$, each  directed by $(\phi^{j'}_\bu, \phi)$. These   processes will play an important role in the sequel. Therefore convening that  $\phi^0_\bu=\lambda^0$ enables us to  identify $N$ with $\tilde{N}^0$ and then to group together many formulas.

\subsection{Population Representation of Multivariate Hawkes Processes}

A useful representation of $N$ for analytic investigation is 
the so called {\it population approach} adopted in previous works like \cite{hawkes1974cluster}  and \cite{el2019characteristic}. We mainly use here  their  notations  to  describe a Hawkes process $\tilde{N}^0 (t)$, directed by $(\lambda^0,  \phi)$,  as the sum of an  ancestor Poisson process  $N^0$ and of all Hawkes processes ignited by each first generation individual. That is to say, at each event  $T^{0,j}_l $, a new Hawkes process  
$\tilde{N}^{j,(l)} $ directed by $(\phi^j_\bu , \phi)$ starts (independently of index l), such that we have the representation
\begin{equation}\label{rep}
	\tilde{N}^0(t)= N^0(t) + \sum_{j=1}^d \sum_{l=1}^{N^0_j(t)} \tilde{N}^{j,(l)}(t-T^{0,j}_l).
\end{equation}

\subsubsection*{Important remarks for the sequel}

\begin{enumerate}
\item Actually from a point of view of filtration, the representation \eqref{rep} is much richer than the proper filtration of $\tilde{N}^0$, since we cannot recover actually the ascendant of individuals (i.e., retrieve the $T^{m,j}_l$ from the knowledge of the $T^j_n$ ).

\item Conditionally to the knowledge of the sequence $T^{m,j}_l$, the Hawkes processes are independent of each others. Moreover for each index $j$, the $\tilde{N}^{j,(l)} $ are  identically distributed, say as a reference Hawkes point process $\tilde{N}^{j}$.

\item Note that the sum of independent   Hawkes processes $\tilde{N}^{0,(l)}, l=1,\ldots,L$, each directed by $(\lambda^{0,(l)}, \phi)$, is also a Hawkes process directed  by $(\sum_{l=1}^L\lambda^{0,(l)}, \phi)$. Consequently, taking $\lambda^{0,(l)}=\lambda^0/L$,  proves that Hawkes processes have infinitely divisible  probability distributions. 

Similarly, any subdivision of $\mathbb{R}_+= \cup_{l=1}^L [t_{l-1},t_l[$ shows that $\tilde{N}^0$ directed by $(\lambda^0, \phi)$ is the sum of $L$ independent copies of Hawkes processes directed by $(\lambda^0\chi_{[t_{l-1},t_l[}, \phi ),   l=1,\ldots,L$, where $\chi_{ [t_{l-1}, t_m [ }$ is the indicator function that gives 1 if its argument is in $[t_{l-1}, t_l [$ and 0 otherwise, that is $\chi_{ [t_{l-1}, t_l[ }(t)=\mathbf 1_{t\in [t_{l-1}, t_l [}$.
\end{enumerate}

\subsubsection*{Probabilities of Zero and One Events}

As a consequence of the population representation \eqref{rep}, we have the following  result:
\begin{lemma}
	For any $t \ge 0$ and $j'=0:d$, we have for $0_{[d]} = (0,\ldots,0) \in \mathbb R^d$
\begin{eqnarray}
	\tilde{p}^{j'}_0(t)=P(\tilde{N}^{j'}(t)=0_{[d]} )&=&\exp \left(-\sum_{j=1}^d \int_0^t \phi^{j'}_j(s) \mathrm ds\right)\\
\tilde{p}^{j'}_{e_i}(t)=P(\tilde{N}^{j'}(t)= e_i)&=&\left(\int_0^t \phi^{j'}_i(s) \mathrm ds\right)\exp \left(-\sum_{j=1}^d \int_0^t \phi^{j'}_j(s) \mathrm ds\right).	
\end{eqnarray}
\end{lemma}
\begin{proof}
For $j'=0$,	 Equation \eqref{rep} tells us that the   events  $``\tilde{N}^0_i(t)=0"$ (resp.  $``\tilde{N}^0_i(t)=1"$ ) correspond exactly to  the independent Poissonian events $``N^0_i(t)=0"$ (resp. $``N^0_i(t)=0"$) directed by $\phi^0_i=\lambda^0_i$, for all $i=1:d$. 
	Clearly, this also holds true for $j'= 1:d$.	
\end{proof}

\subsection{Multi-Time Functional Laplace Transform of a Multivariate Hawkes Process}
\label{section1}

Let  $a= \left(a^j_i\right)$ be a $d\times n$  matrix with non negative values and   $ t=(t_1,\ldots,t_n)$ a set of ordered times  $0\le t_1\le\ldots\le t_n$.

The  $n-$multi-temporal  Laplace transform of the $d-$multivariate point process $\tilde{N}^{j'}$, directed by $(\phi^{j'}_\bu, \phi), \ j'=0:d,$  is defined as:
 
\begin{equation*}
\tilde{L}^{j',(n)}(a, t):=\mathbf E_{(\phi^{j'}_\bu, \phi)}\left[\exp\left(-\sum_{r=1}^n \tr a^r_\bu\;  \tilde{N}^{j'}(t_r) \right) \right].
\end{equation*}

\subsubsection{Classical Laplace Transforms of Hawkes Processes}

  Let $\tilde{N}^{j'}, \ j'=0:d,$ be Hawkes processes respectively driven by functional parameters $(\phi^{j'}_\bu, \phi )$  with respective Laplace transforms $\tilde{L}^{j',(k)}, \ j'=0:d , k= 1:n$. Then   we have the following set of integral equations relating their  multi-temporal Laplace transforms:

\begin{theorem}
\label{thm:1}
With previous notations $t_{[k,m]}$, \ $a^{[k,m]}_\bu, \
\overline{a^{[k,m]}_j}= \sum_{r=k}^m a^r_j$  and putting $t_0=0$, we have
for every $j' =0:d$ and $m =1:n$, 
\begin{equation}\label{Eq-tildeL}
	\tilde{L}^{j',(m)} \left(a,  t\right)= \tilde{p}^{j'}_0(t_m)
	\prod_{k = 1}^m 
	\exp\left(\sum_{j=1}^d
	\int_{t_{k- 1}}^{t_k} 
	 e^{- \overline{a^{[k,m]}_j}} \tilde{L}^{j,(m+1-k)}(a^{[k,m]}_\bu,  (t-u)_{[k,m]}) \ 
	\phi^{j'}_{j}(u)\mathrm du
	\right).
\end{equation}
\end{theorem}

\paragraph{Remark.} This theorem mainly tells us  that the  Laplace transform related to any $\lambda^0$ can be derived through the solution of the system of $d$ basic Laplace transforms 	$\tilde{L}^{j',(k)}, \ j'=1:d$. The system of non-linear integral equations \eqref{Eq-tildeL} will yield Volterra equations for the first and second order moments of the process (see Section \ref{section_int_equation_first_moments}). Hence, we call this system a Volterra-like system of equations.
%\end{remark}

\begin{proof}
 We thoroughly use the representation \eqref{rep} and the associated remarks, where $N^0$ is a multivariate Poisson process with independent components.
 For sake of simplicity we denote for $j=1:d$,  the set  of occurrence times 
 $\mathcal{T}^j_k=\{T^{0,j}_l:   t_{k-1}\leq T^{0,j}_l<t_k \}$ and  the corresponding set of indices  $E^j_k=\{ l :  T^{0,j}_l \in \mathcal{T}^j_k \}$ both of random size $Q^j_k$. 
 
For each $m=1,\ldots,n$,  the   expansion      
 $$
 \tilde{N}^0(t_m)=\sum_{k=1}^m \left( \left[N^0(t_k)-N^0(t_{k-1})\right] + \sum_{j,l \in E^j_k}  \tilde{N}^{j,(l)} (t_m- T^{0,j}_l) \right),
 $$
 shows that $\tilde{N}^0(t_m)$ is the sum of the contributions of $m$ independent Hawkes processes $\tilde{N}^{0,(k)}(t_m)$, each directed by  $(\lambda^0\chi_{ [t_{k-1}, t_k [ }, \phi)$.

Therefore, inverting summation order and  regrouping  similar terms yields:
 
 $$\sum_{m=1}^n \tr a^m_\bu\;  \tilde{N}^0 (t_m)= \sum_{k=1}^n \left( 
  \tr \overline{  a^{[k,n]}_\bu} [N^0(t_k) -N^0(t_{k-1})] + 
 \sum_{m=k}^n \sum_{j,l \in E^j_k}   \tr a^m_\bu \tilde{N}^{j,(l)} (t_m- T^{0,j}_l)    \right),
 $$
which reads as the sum of the contributions of  $n$ independent  Hawkes processes $\tilde{N}^{0,(k)}$.

 Next, for the Poisson process $N^0$ and $q=\tr (q_1,\ldots,q_d) \in \mathbb{N}^d$, we recall   that  
 \begin{align*}p_k^0(q)&=\mathbf P( N^0(t_k) -N^0(t_{k-1} ) =q)\\
 &=\prod_{j=1}^d \exp{ (\Lambda_j^0(t_{k-1}) -\Lambda_j^0(t_k) )}   \left( \Lambda_j^0(t_k) -\Lambda_j^0(t_{k-1})  \right)^{q_j} /q_j ! .
 \end{align*}
 We also recall that for each $k$, conditionally to the event $\{N^0(t_k) -N^0(t_{k-1} ) =q\}$, the $d$ sets  of occurrence times 
 $\mathcal{T}^j_k, \ j=1:d,$ are independent, that for each given $j$, 
 the set  $\mathcal{T}^j_k=\{T^{0,j}_l :   t_{k-1}\leq T^{0,j}_l<t_k \}$ of size $Q^j_k$ is a  sequence of i.i.d.  random variables  having probability density   
 $\lambda^0_j(u)/(\Lambda_j^0(t_k) -\Lambda_j^0(t_{k-1} ) )$ with  $ u \in [t_{k-1},t_k [$
 \cite{david2004order}. Consequently the set   $\{ \tilde{N}^{j,(l)} (\cdot- T^{0,j}_l) : t_{k-1}\leq T^{0,j}_l<t_k \}$ of size $Q^j_k$ also forms a sequence of i.i.d. Hawkes processes.

 So, we have

\begin{equation*}
\begin{aligned}
\label{eq_intermediaire}
    \tilde{L}^{0,(n)} (a,t)=& 
        \prod_{k=1}^n  \left(\sum_{q} p^0_k(q) \prod_{j=1}^d 
        e^{-\overline{a^{[k,n]}_j} q_j} \mathbf{E}\left[\prod_{l\in E^j_k } e^{ \left(-
      \sum_{m=k}^n  \tr a^m_\bu \tilde{N}^{j,(l)} (t_m- T^{0,j}_l) \right)}  |
    Q^j_k=q_j,  \mathcal{T}^j_k \right]\right) \\    
    =&\prod_{k=1}^n  \left(\sum_{q} p^0_k(q) \prod_{j=1}^d
    e^{-\overline{a^{[k,n]}_j} q_j} \mathbf{E} \left[\prod_{l \in E^j_k} 
     \tilde{L}^{j,(n+1-k)} (a^{[k,n]}_\bu, (t-T^{0,j}_l)_{[k,n]}) |
    Q^j_k=q_j  \right]\right) \\   
    =& \prod_{k=1}^n  \left(\sum_{q} p^0_k(q) \prod_{j=1}^d
        e^{- \overline{a^{[k,n]}_j} q_j}     
       \left( \int_{t_{k-1}}^{t_k} \tilde{L}^{j,(n+1-k)} (a^{[k,n]}_\bu, (t-T^{0,j}_l)_{[k,n]}) 
       \frac{\lambda^0_j(u) \mathrm du}{\Lambda^0(t_k) -\Lambda(t_{k-1})} 
        \right)^{q_j}
         \right) \\ 
  =& \prod_{k=1}^n   
  \exp \left(   \sum_{j=1}^d    
 \int_{t_{k-1}}^{t_k}  \left( e^{- \overline{a^{[k,n]}_j}} \tilde{L}^{j,(n+1-k)}  (a^{[k,n]}_\bu, (t-u)_{[k,n]}) -1\right)
  \lambda^0_j(u) \mathrm du 
  \right)  \\
  =&   \exp\left( -\sum_{j=1}^d  \int_0^{t_n} \lambda^0_j(u) \mathrm du \right)\ \prod_{k=1}^n   
  \exp \left(   \sum_{j=1}^d    
  \int_{t_{k-1}}^{t_k}   e^{- \overline{a^{[k,n]}_j}} \tilde{L}^{j,(n+1-k)}  (a^{[k,n]}_\bu, (t-u)_{[k,n]}) \ \ 
  \lambda^0_j(u) \mathrm du 
  \right).          
\end{aligned}
\end{equation*}

We  notice that the computation of any characteristic function  $\tilde{L}^{0,(n)}$  is  uniquely determined  by the knowledge of the Laplace transforms $\tilde{L}^{j',(k)}, \ j'=1:d, k=1:n $. Consequently, replacing $\lambda^0$ with $\phi^{j'}_\bu, \ j'=1:d$, yields a fundamental set of $d$ dependent integral equations:

\begin{equation}\label{laplacek}
	\tilde{L}^{j',(n)}\left(a,  t\right)= \tilde{p}^{j'}_0(t_n)
	\prod_{k = 1}^n 
	\exp\left( \sum_{j=1}^d
	\int_{t_{k - 1}}^{t_k} 
	 e^{- \overline{a^{[k,n]}_j}} \tilde{L}^{j,(n+1-k)},  (t-u)_{[k,n]})
	 \ \  \phi^{j'}_{j}(u)\mathrm du
	\right).
\end{equation}
\end{proof}

Equation \eqref{laplacek} actually yields a recursive formula of Laplace transforms.

\begin{corollary}
Enhancing the dependence of Laplace transforms with respect to $\lambda^0$, we get 
$$\tilde{L}^{0,(n)}_{\lambda^0} (a,t) = \tilde{L}^{0,(n)}_{\lambda^0\chi_{]0,t_n]}} (a,t) = \prod_{k=1}^n \tilde{L}^{0,(n+1-k)}_{\lambda^0\chi_{]t_{k-1},t_k]}}
(a^{[k,n]}_\bu,(t-t_{k-1})_{[k,n]}).$$   
\end{corollary}

The same recurrence holds for the equations satisfied by $\tilde{L}^{j',(n)}_{\lambda^0} (a,t), \ j'=1:d $.

%%%%%%%%%%%%%%%%%%%%%%%%%%%%%%%%%%%%%%%%%%%%%%%%%%%%%%%%%%%%%%%%
%%%%%%%%%%%%%%%%%%%%%%%%%%%%%%%%%%%%%%%%%%%%%%%%%%%%%%%%%%%%%%%%

\subsection{Functional Laplace Transforms}

Let $t\in\mathbb R$ and $\psi:\mathbb R_+ \longrightarrow \mathbb R_+^d $ be a bounded measurable vector function. We now state a somewhat more general theorem concerning the functional Laplace transform of multivariate Hawkes processes $\tilde{N}^{j'}$ driven by $(\phi^{j'}_\bu,\phi), \ j'=0:d$, where the functional Laplace transform is defined as:

\begin{equation*}
	\tilde{S}^{j'}(\psi, t):=\mathbf E_{(\phi^{j'}_\bu,\phi)}\left[\exp\left( -\int_0^{t}\tr\psi(s) \tilde{N}^{j'}(\mathrm ds) \right)\right].
\end{equation*}

 We have the following result, whose proof is very similar to that of Theorem \ref{thm:1}
\bigskip

\begin{theorem}
	\label{thrm:2}
	For $t\geq 0$ and any bounded measurable $\psi:\mathbb R_+ \longrightarrow \mathbb R_+^d$,  we have for $j'=0:d$,
	\small
	\begin{equation*}
		\tilde{S}^{j'}\left(\psi, t\right)= \tilde{p}^{j'}_0(t)\exp \left( -\sum_{j=1}^d\, \int_0^{t} 
		e^{\psi_j(s)} \tilde{S}^j(\psi_{s+},t-s) 
		\phi^{j'}_j(s)\,\mathrm ds
		\right).
	\end{equation*}
\end{theorem}

\begin{proof}

	We start dealing with the case $j'=0$ driven by $(\lambda^0,\phi)$. Once again, like in the proof of Theorem \ref{thm:1}, let us  denote for $j=1:d$ \  the set  of occurrence times 
	$\mathcal{T}^j=\{T^{0,j}_l:   T^{0,j}_l \leq t \}$ and  its associate  set of indices  $E^j=\{ l :  T^{0,j}_l \in \mathcal{T}^j \}$ both of random size $N^0_j(t)$.

    From the process representation \eqref{rep}, we have
	\begin{equation*}
		\begin{aligned}
			\int_{0}^{t}
			\tr \psi(s)\, \tilde{N}^0(\mathrm ds)
			&= \int_{0}^{t}
			\tr \psi(s)\,  N^{0}(s) 
			+\sum_{j=1}^d \sum_{l=1}^{N^0_j(t)}			\left(
			\int_0^{t-T^{0,j}_l}
			\tr \psi(T^{0,j}_l+s)\, \tilde{N}^{j,l}(\mathrm ds)
			\right)\\ 
			%%%%%%%%%%%%%%%%%%%%%%%%%%%
			&=\sum_{j=1}^d \sum_{l=1}^{N^0_j(t)} \left(
			 \psi_j(T^{0,j}_l)
			+
			\int_{0}^{t-T^{0,j}_l}
			\tr \psi(T^{0,j}_l + s)\, \tilde{N}^{j,l}(\mathrm ds)\right).
		\end{aligned}
	\end{equation*}
	
	\noindent Hence, considering the independence of the point processes $N^0_j, \ j=1:d,$ and  that conditionally to  $ N^0_j(t)= n^0_j$, the set $\mathcal{T}^j$ forms an i.i.d. sample of size $n^0_j$ of random variables with probability distribution function $\lambda^0_j (s)/\Lambda^0_j (t)$ concentrated on the interval $[0,t]$ for each $j=1,\ldots, d$,   entails that  
	the set of random variables  
	\[\{  \tilde{S}^j(\psi_{T^{0,j}_l+}, t-T^{0,j}_l ) =E_{(\phi^j_{\bu}, \phi)} \big(\exp \big(-\int_{0}^{t-T^{0,j}_l} 
		\tr \psi(T^{0,j}_l+s)\, \tilde{N}^{j,(l)}(\mathrm ds)\big) \big); \  l=1,\ldots, n^0_j   \}\]
	constitutes also a sample of i.i.d. random variables.
	
Therefore, we have 
	\begin{eqnarray*}
	   	\tilde{S}^0(\psi, t)
			&=&\mathbf E_{(\lambda^0, \phi)} \Bigg[
			\exp\left(-\sum_{j=1}^d \sum_{l=1}^{N^0_j(t)} \left(
			\psi_j(T^{0,j}_l)
			+
			\int_{0}^{t-T^{0,j}_l}
			\tr \psi(T^{0,j}_l + s)\, \tilde{N}^{j,l}(\mathrm ds)\right)\right) \Bigg]
			\\
			&=& \prod_{j=1}^d \mathbf E_{(\lambda^0, \phi)}\Bigg(
			\mathbf E \Big(\prod_{l=1}^{N^0_j(t)}
		e^{- \psi_j(T^{0,j}_l)} \mathbf E_{(\phi^j_{\bu},\phi)} \left( \exp \left(-
			\int_{0}^{t-T^{0,j}_l}\tr \psi(T^{0,j}_l + s)\, \tilde{N}^{j,l}(\mathrm ds) \right)\right) \Big|    N^0_j(t) \Big) \Bigg)	\\	
			&=& \prod_{j=1}^d \mathbf E_{(\lambda^0, \phi)}\Bigg(
			\Big[\mathbf E \Big(
			e^{-\psi_j(T^{0,j}_l)}\, \tilde{S}^j(\psi_{T^{0,j}_l+}, t-T^{0,j}_l )   \Big) 
			 \Big]^{ N^0_j(t)} \Bigg)	\\			
		&=&\prod_{j=1}^d \mathbf E_{(\lambda^0, \phi)}\Bigg( \Big[
		\int_0^t e^{-\psi_j(s)} \tilde{S}^j(\psi_{s+}, t-s)
		 \frac{\lambda^0_j(s)}{\int_0^t\lambda^0_j(v)\mathrm dv} \mathrm ds
		\Big]^{ N^0_j(t)} \Bigg).	\\	
	\end{eqnarray*}
	
Since  $N^0_j (t)$  are Poisson distributed with respective mean parameter 
$\Lambda^0_j(t)=\int_0^t \lambda^0_j(v)\mathrm dv$, we finally get 

\begin{eqnarray*}
	\tilde{S}^0(\psi, t)		
	&=&\prod_{j=1}^d \Big( \sum_{k=0}^\infty \frac{e^{-\Lambda^0_j(t) }}{k!}
	(\Lambda^0_j(t))^k \Big[
	\int_0^t e^{-\psi_j(s)} \tilde{S}^j(\psi_{s+}, t-s)
	\frac{\lambda^0_j(s) \mathrm ds}{\Lambda^0_j(t)}
	\Big]^k \Big)\\
	&=& e^{-\sum_{j=1}^d \Lambda^0_j(t)} \exp\Big(\sum_{j=1}^d \int_0^t e^{-\psi_j(s)} \tilde{S}^j(\psi_{s+}, t-s)
				\lambda^0_j(s) \mathrm ds\Big).
\end{eqnarray*}

	Similarly, replacing $\lambda^0$ with  $\phi^{j'}_{\bu}; \ j'=1:d$,  yields
	$$\tilde{S}^{j'}(\psi, t)= 	e^{-\sum_{j=1}^d \int_0^t \phi^{j'}_j(s)\mathrm ds} \exp\Big(
	\sum_{j=1}^d \int_0^t e^{-\psi_j(s)} \tilde{S}^j(\psi_{s+}, t-s) \phi^{j'}_j (s)\mathrm ds\Big).	$$
	
\end{proof}

\paragraph{Remarks}  
\begin{enumerate}
		
\item Actually Theorem \ref{thm:1} can   be deduced from Theorem \ref{thrm:2} by considering the following piece-wise constant function:
\begin{equation*}
	\psi(s)=\sum_{k=1}^n  \overline{a^{[k,n]}}\mathbf 1_{]t_{k-1}, t_k]}(s)
\end{equation*}
 and by observing that for  $s \in ]t_{k-1}, t_k]$ and $j=1:d$, we have 
$\psi_j(s) = \overline{a^{[k,n]}_j }$ and $\tilde{S}^j(\psi_{s+}, t-s) =
 \tilde{L}^{j,(n+1-k)} (a^{[k,n]}_\bu, (t-s)_{[k,n]}) $.
 
 \item  Moreover, emphasizing the dependence upon $\lambda^0$ via the notation $ \tilde{S}^{0}=\tilde{S}_{\lambda^0}$, we obtain a similar recurrence equation: for $0< t_1<\cdots< t_n <t$, 
 $$ \tilde{S}_{\lambda^0} \left(\psi, t\right)= \prod_{k=1}^n \tilde{S}_{\lambda^0\chi_{]t_{k-1},t_k]}}
 (\psi_{t_{k-1}^+},t-t_{k-1}).  $$
 
 \item Functional Laplace transforms may be of some importance in applications. For example, suppose that we are interested in the assessment of the medical or financial costs of an epidemic in $d$ regions, while the counting process is only observed over deterministic region-specific time intervals $]C_j^{2k}, C_j^{2k+1}], \ j=1:d, \ k \geq 0$. In this case, the relevant Laplace transform  would concern the functions
 $\psi_j(s)= a_j \theta_j(s) \sum_{k\geq 0}  \chi_{ ]C_j^{2k}, C_j^{2k+1}] }(s)$, where 
$\theta_j(s)$ corresponds to the  cost induced by the occurrence of an event in region $j$  at time $s$ and the vector $a=\tr(a_1,\ldots,a_d)$ holds for
variables of the Laplace transform $L^{j'}_C(a,t)$ of censored observations,  that is $\tilde{S}^{j'}(\psi,t)=L^{j'}_C(a,t).$

\end{enumerate}

%%%%%%%%%%%%%%%%%%%%%%%%%%%%%%%%%%%%%%%%%%%%%%%%%%%%%%%%%
\section{Infinite  Divisibility and Distribution Probabilities of Counts} 
\label{section_inf_div}

\subsection{Preliminaries}
We detail here the precise form of the infinitely divisible property for Hawkes processes and 
the integral equations related to the probability distribution of its counts.
First  we start with the following analytic result.
\begin{proposition} 
Let $f(z)= \sum_{l \in \mathbb{N}^d} \alpha_l z^l$  be a $d$-multivariate  analytic function on the complex polydisc $D^d$ with  $z=(z_1,\ldots,z_d), \, l=(l_1,\ldots, l_d)$ and $z^l=z_1^{l_1}\times \ldots\times z_d^{l_d}$. Then $g(z)= e^{f(z)}$ admits the representation
$	g(z) = e^{\alpha_0} \big( 1+ \sum_{k\ne 0} \beta_k z^k \big) $ where for $k=(k_1,\ldots,k_d)\neq 0$ and 
$|k|= k_1+\ldots+k_d$

\begin{equation} \label{coef1}
\beta_k = \sum_{r=1}^{|k|} \frac{ 1}{r! }\sum_{\substack{l^{(1)}\ne 0,\ldots, l^{(r)}\ne 0 \\ l^{(1)}+\ldots+ l^{(r)}= k}}
\alpha_{l^{(1)}}\times \ldots \times \alpha_{l^{(r)}}.
\end{equation}
Conversely, we have  for $l\neq 0$
\begin{equation}\label{coef2}
	\alpha_l = \sum_{r=1}^{|l|} \frac{ (-1)^{r-1}}{r}\sum_{\substack{k^{(1)}\ne 0,\ldots, k^{(r)}\neq 0 \\ k^{(1)}+\ldots+ k^{(r)}= l} }
\beta_{k^{(1)}}\times \ldots \times \beta_{k^{(r)}}.
\end{equation}
\end{proposition}

\begin{proof}
	Simply apply the series expansion of $\exp$ and $\log(1+\cdot)$  functions about point $0$.	
\end{proof}

\begin{corollary}\label{cor-coef}
	Let $X$ be an $\mathbb{N}^d$-valued random variable with probability generating function  $g(z)=\mathbf E(z^X)=\sum_{k \in \mathbb{N}^d} p_k z^k $ and $p_k=P(X=k)$.  If   $p_0>0$, then there exists a unique multivariate analytic function $f(z)= \sum_{l \in \mathbb{N}^d} \alpha_l z^l $ such that $g(z)=e^{f(z)}$,  with  $\alpha_0=\log(p_0)$ and for all $l\neq 0$, $\alpha_l$ satisfies Equation \eqref{coef2} with $\beta_k= p_k/p_0$.
\end{corollary}

\paragraph{Remark.} Notice that the summation in the previous formulas describes the number of distinct paths in the lattice $\mathbb{N}^d$ starting from point $(0,\ldots,0) $ and leading to point $k=(k_1, \ldots,k_d)$ in $r$ steps. 
Moreover the calculation of $\beta_k$ (or $\alpha_l$)  is recursive/causal since it only requires the knowledge of coefficients $\alpha_{l^{(p)}}$ with indices $l^{(p)} \leq  k$, that is to say $ 1\leq l^{(p)}_j \leq k_j, \, j=1,\ldots,d$.

\begin{proposition}\label{prop_vpos}
	Let $X$ be a multivariate random variable with probability distribution function $F(\mathrm dx)$  on $\R_+^d$, Laplace transform $L(a)=\mathbf E(e^{-\tr aX}), a \in \R_+^d$, and such that $0< p_0=F(\{0\})<1$. Let $\bar{F}^{(*)r}(\mathrm dx)$ denote the $r^{th}$ convolution of $\bar{F}= F(\mathrm dx)/(1-p_0)$ probability measure on $(\R_+^d)_*$. Then there exists a unique finite signed measure $\gamma(\mathrm dx)$ on $(\R^d_+)_*$ such that $\gamma((\R^d_+)_*)= -\log ( p_0)$ and
	\begin{equation} \label{id-rep}
		L(a)=\exp\left( \log(p_0)+\int_{(\R_+^d)_*} e^{-\tr a x}\gamma(\mathrm dx)  \right)
	\end{equation}
with
\begin{equation}\label{prop_vpos1}
	\gamma(\mathrm dx) =\sum_{r=1}^\infty \frac{(-1)^{r-1}}{r} \left(\frac{1-p_0}{p_0}\right)^r \bar{F}^{*r}(\mathrm dx). 	
\end{equation}
\end{proposition}

\begin{proof}
	Since  $\bar{F}(\mathrm dx)=F(\mathrm dx)/(1-p_0)$ is a probability distribution on $(\R^d_+)_*$ with Laplace transform $\bar{L}(a)$, the convolution rule applies and yields:  
	$$\bar{L}^r(a)=\left(\int_{(\R^d_+)_*} e^{-\tr ax}\bar{F}(\mathrm dx)\right)^r=\int_{(\R^d_+)_*} e^{-\tr ax}\bar{F}^{(*)r}(\mathrm dx),$$ such that	
\begin{equation}	
\log(	L(a))	= \log\left(p_0+ \int_{(\R^d_+)_*} e^{-\tr ax}F(\mathrm dx) \right)\\
=	 \log\left(1+   (1-p_0)\left(\bar{L}(a) -1\right)\right).
\end{equation}
Since $ \left|(1-p_0)\left(\bar{L}(a) -1\right) \right| <1$, the series expansion  yields 
\begin{eqnarray*}
\log(L(a))&=&
  \sum_{r=1}^\infty \frac{(-1)^{r-1}}{r} (1-p_0)^r  \left( C^r_0 (-1)^r +\sum_{j=1}^r C^r_j (-1)^{r-j}   \bar{L}^j(a) \right)  \\
&=& -\sum_{r=1}^\infty \frac{1}{r} (1-p_0)^r + 
\sum_{j=1}^\infty  
\left(\sum_{r=j}^\infty \frac{1}{r} (1-p_0)^r C^r_j \right)  (-1)^{j-1} \bar{L}^j(a). 
\end{eqnarray*}

\noindent Taking into account the equality $(1-x)^{-(j+1)}=\sum_{r=j}^\infty C^r_j x^{r-j}$, we get 
\begin{eqnarray*}
\log(L(a)) &= & \log(p_0) + \sum_{j=1}^\infty  
 \frac{(-1)^{j-1}}{j} \left(\frac{1-p_0}{p_0}\right)^j \bar{L}^j(a)  \\
&=& \log(p_0) +   \int_{(\R^d_+)_*} e^{-\tr ax} \left(\sum_{j=1}^\infty \frac{(-1)^{j-1}}{j} \left(\frac{1-p_0}{p_0}\right)^j    \bar{F}^{(*)j}(\mathrm dx) \right).
\end{eqnarray*}

\noindent Now, we observe that \[\gamma(\mathrm dx)\equiv \sum_{j=1}^\infty \frac{(-1)^{j-1}}{j}\left(\frac{1-p_0}{p_0}\right)^j \bar{F}^{(*)j}(\mathrm dx)\]  can be written as the difference of two  non-negative measures:
\begin{eqnarray*}
    \gamma_1(\mathrm dx)&=&
\sum_{j=1}^\infty \frac{1}{2j-1}\left(\frac{1-p_0}{p_0}\right)^{2j-1} \bar{F}^{(*)(2j-1)}(\mathrm dx)\\
\gamma_2(\mathrm dx)&=&\sum_{j=1}^\infty \frac{1}{2j}\left(\frac{1-p_0}{p_0}\right)^{2j} \bar{F}^{(*)2j}(\mathrm dx).
\end{eqnarray*}

\noindent Since $\bar{F}^{(*)j}$ is for each $j\geq 1$ a probability measure on $(\R^d_+)_*$, 
we  can deduce that $\gamma_1$ and $\gamma_2$ are non negative possibly infinite measures. 

It remains to show that $\gamma_1$ and $\gamma_2$ are finite.  In this aim, it is sufficient to  consider any Borelian interval product $B=\prod_{j=1}^d [0,b_j] $ with $b_j\geq 0$ and to observe that   $\bar{F}^{(*)r}(B)\leq \bar{F}^r(B)$.
Therefore, 
$\gamma_1(B)\leq \sum_{j=1}^\infty\frac{1}{2j-1}\left(\frac{1-p_0}{p_0} \bar{F} (B) \right)^{2j-1} $
 and so   $\gamma$ is a $\sigma-$additive signed measure on $(\R^d_+)_*$. Moreover $\gamma((\R^d_+)_*)=\sum_{r=1}^\infty \frac{(-1)^{r+1}}{r}\left(\frac{1-p_0}{p_0}\right)^r= \log(1+\frac{1-p_0}{p_0} )=-\log(p_0)$.

\end{proof}
 
%%%%%%%%%%%%%%%%%%%%%%%%%%%%%%%%%%%%%%%%%%%%%%%%

\subsection{Infinite Divisibility Property}

First, we recall the most achieved formulation of the L\'evy Kinchine representation theorem for infinitely divisible random vector $X$ with non-negative components, i.e. $P(X \in \R_+^d)=1$, see \cite{loeve1977elementary}. %\cite{Murr}.

\begin{theorem}\label{thm-id1}
Let $X$ be a multivariate random variable with values on $\R_+^d$ and Laplace transform $L(a)=\mathbf E(e^{-\tr aX}), a \in \R_+^d$. Then $X$  is infinitely divisible if and only if there exist a unique $\beta \in \R_+^d$ and a $\sigma$-finite measure $\gamma$ on $\R_+^d$ such that $\gamma(\{0\})=0$, $\int_{\R_+^d} (\|x\| \wedge 1)\gamma(\mathrm dx) < \infty $ and 
\begin{equation} \label{id-rep}
	L(a)=\exp\left(  -\tr a\beta +\int_{(\R_+^d)_*} (e^{-\tr a x} -1)\gamma(\mathrm dx)  \right).
	\end{equation}
\end{theorem}

Nevertheless, one can say a little more for non negative integer valued random vectors.

\begin{theorem} \label{thm-id2}
	Let $X$ be a multivariate random variable with values in $\N^d$ with $P(X=0)=p_0 >0$ and Laplace transform $L(a)=\mathbf E(e^{-\tr aX}), \ a \in \R_+^d$. Then $X$  is infinitely divisible if and only if there exists a (unique) finite measure $\gamma(\mathrm dx)= \sum_{l\in (\N^d)_* }\alpha_l \delta_l(\mathrm dx)$  on $(\N^d)_*$,   satisfying $ \sum_{l \in (\N^d)_*} \alpha_l  =-\mathrm{log}(p_0)$ and such that 
	$$L(a)=p_0\exp\left( \sum_{l\ne 0} e^{-\tr a l} \alpha_l  \right),$$
	where the coefficients $\alpha_l$, $l\neq 0$  satisfy Equation \eqref{coef2} with $\beta_k= p_k/p_0$.
\end{theorem}

\begin{proof}
	Since $L(a)$ can be written as
$g(z)= p_0 + \sum_{k \ne 0 } z^k p_k$ with 	$z=(e^{-a_1}, \ldots, e^{-a_d})$ and $p_0>0$,
then  according to Corollary \ref{cor-coef}, we have a unique function $f(z)= \sum_{l \in \mathbb{N}^d} \alpha_l z^l) $ satisfying $g(z)= e^{f(z)}$ with $\alpha_0=\log(p_0)$ and the $\alpha_l$ satisfying Equation \eqref{coef2} with $\beta_k= p_k/p_0$.

We now prove that in the representation \eqref{id-rep}, $\beta=0 $ and $\gamma$ is finite. 

Indeed, if we take $a=(u,\ldots,u) \in \R_+^d$, then 
for all  $k \in (\N^d)_*$, we have 
$0 \leq e^{- \tr a k}= e^{-u |k|} \leq e^{-u} $ and then 
$p_0 \leq L(a) \leq p_0 + e^{-u} (1-p_0) \underset{u \to +
	\infty}{\longrightarrow } p_0$.

On the other hand, for $x\neq 0$, we have $0\leq 1 -e^{-u|x|} \underset{u \to +
	\infty}{\uparrow} 1$; hence, by Beppo Levi's lemma we have, \eqref{id-rep} yields 
$L(a) \underset{u \to +
	\infty}{\longrightarrow } \exp\left(  -\infty|\beta|  - \gamma((\R_+^d)_*) \right)=p_0 >0$.
But this holds true only for $\beta=0$ and $-\log(p_0)=\gamma((\R_+^d)_*) <\infty$.
This also means that $\gamma$ is a finite measure.

Moreover, since $f$ can be written $f((e^{-a_1},\ldots, e^{-a_d} ))= \log(p_0)+ \int_{(\N^d)_*} e^{-\tr a x} \alpha(\mathrm dx)$, with  $ \alpha(\mathrm dx) =\sum_{l\neq 0} \alpha_l \delta_l (\mathrm dx)$ and that  
the measure  $\gamma(\mathrm dx)$ is unique, we simply have $\gamma(\mathrm dx)=\alpha(\mathrm dx)$, whose support is $(\N^d)_*$.
\end{proof}

\paragraph{Remark.}
	If  $X$ is indefinitely divisible and $\N^d$ valued, we always have 
	$\mathbf E(x) = \sum_{k\neq 0} k \alpha_k$ and $Cov(X)= \sum_k k\tr k \alpha_k$ for  both expressions, either finite or infinite.

%%%%%%%%%%%%%%%%%%%%%%%%%%%%%%%%%%

\subsection{ Distribution Probabilities of Hawkes Process Countings}
Without loss of generality, let us  assume for the indefinitely divisible Hawkes processes $\tilde{N}^{j'}(t),\  j'=0:d$, that $\lambda_0(s) $ and $\phi$ are  integrable over any interval $[0,t]$, and satisfy $\tilde{p}^{j'}_0(t) >0$ and $\tilde{p}^j_0(t) >0$ for  all $t>0$.

\begin{proposition}\label{counts-prob}
	Let  $\tilde{N}^{j'}, \ j'=0:d$, be multivariate Hawkes processes driven by integrable functions $(\phi^{j'}_\bu, \phi)$ and let $ \tilde{p}^{j'}_k(t)=P_{(\phi^{j'}_\bu , \phi)}(\tilde{N}^{j'}(t)=k),  k\in \N^d$, the probability distribution functions of $\tilde{N}^{j'}(t) , \ j'=0:d, t\in \R_+$.
	Then we have 
	\begin{equation}\label{prob-recur}
		\frac{\tilde{p}^{j'}_l}{ \tilde{p}^{j'}_0}(t) -\sum_{r=2}^{|l|} \frac{ (-1)^{r}}{r}\sum_{\substack{k^{(1)}\ne 0,\ldots, k^{(r)}\neq 0\\ k^{(1)}+\ldots+ k^{(r)}= l}}
	\frac{\tilde{p}^{j'}_{k^{(1)}}}{\tilde{p}^{j'}_0}(t)\times \ldots \times \frac{\tilde{p}^{j'}_{k^{(r)}}}{\tilde{p}^{j'}_0}(t)=
	\sum_{j=1}^d
	\int_{0}^{t}  
	\left(\tilde{p}^{j}_{l-e_j} (t-s) \right) \ 
	\phi^{j'}_{j}(s)\mathrm ds.
\end{equation}	
\end{proposition}

\begin{proof}
First, let us deal with $j'=0$ and then observe that equations of Theorem \ref{thm:1} reduce to the following forms for a single time $t>0$  and $j' \in \{1,\ldots d\} $: 
\begin{equation}
	\tilde{L}^{j'(1)}_{\phi^{j'}_{\bu}} \left(a,  t\right)=
	\tilde{p}^{j'}_0(t) \exp\left(\sum_{j=1}^d
	\int_{0}^{t} 
	 e^{- a_j} \tilde{L}^{j,(1)}(a,  (t-s)) \ 
	\phi^{j'}_{j}(s)\mathrm ds
	\right).
\end{equation}

Since   $ \tilde{N}^{j}, \ j=0:d,$ are infinitely divisible, then according to  Theorem \ref{thm-id2}, there exist finite measures $\gamma^{j'}(t,\mathrm dx))= \sum _{k \in (\N^d)_*} \alpha^{j'}_k(t)\delta_k (\mathrm dx)$, for all $t>0$ and  $j'=0:d$,  such that 
\begin{eqnarray*}
\sum_{l\ne 0} e^{-\tr a l}\alpha^{j'}_l(t)
	&=&\sum_{j=1}^d
	\int_{0}^{t} 
	e^{- a_j} \tilde{L}^{j,(1)}(a,  (t-s)) \ 
	\phi^{j'}_j(s)\mathrm ds \\
   &=& \sum_{j=1}^d
\int_{0}^{t} 
e^{- a_j} \left( \sum_{l \in \N^d} e^{-\tr a l} \tilde{p}^{j}_l (t-s) \right) \ 
	\phi^{j'}_j(s)\mathrm ds \\
&=&      \sum_{l \in \N^d} e^{-\tr a l} \sum_{j=1}^d
\int_{0}^{t}  
e^{- a_j} \left(\tilde{p}^{j}_l (t-s) \right) \ 
	\phi^{j'}_j(s)\mathrm ds\\
&=&   \sum_{l \ne 0} e^{-\tr a l} \left(\sum_{j=1}^d
\int_{0}^{t}  
 \left(\tilde{p}^{j}_{l-e_j} (t-s) \right) \ 
	\phi^{j'}_j(s)\mathrm ds\right)
\end{eqnarray*}
with $e_j$ being the $j^{th}$  unit vector of $\R^d$.
Therefore we have $\alpha^{j'}_l(t)=\sum_{j=1}^d
\int_{0}^{t}  \tilde{p}^{j}_{l-e_j} (t-s)  \ \phi^{j'}_j(s)\mathrm ds$.
 Then, according to Equation \eqref{coef2}, we obtain the following recursive formula for $l\ne 0$ and $j'=0:d$:
$$\frac{\tilde{p}^{j'}_l}{ \tilde{p}^{j'}_0}(t) -\sum_{r=2}^{|l|} \frac{ (-1)^{r}}{r}\sum_{\substack{k^{(1)}\ne 0,\ldots, k^{(r)}\neq 0\\ k^{(1)}+\ldots+ k^{(r)}= l}}
\frac{\tilde{p}^{j'}_{k^{(1)}}}{\tilde{p}^{j'}_0}(t)\times \ldots \times \frac{\tilde{p}^{j'}_{k^{(r)}}}{\tilde{p}^{j'}_0}(t)=
\sum_{j=1}^d
\int_{0}^{t}  
\left(\tilde{p}^{j}_{l-e_j} (t-u) \right) \ 
\phi^{j'}_{j}(u)\mathrm du.$$

\end{proof}

\paragraph{Example.}
	Let us consider the case $d=1$ with $\phi^0=\lambda^0$ and $\phi=\phi^1$. From the population representation \eqref{id-rep}, we  first  have for $l$ equals 0 and 1:	$\tilde{p}^{j'}_0(t)=e^{ -\Phi^{j'}(t)}$ and $\tilde{p}^{j'}_1(t)=e^{ -\Phi^{j'}(t)} \Phi^{j'}(t)$.
Then for $l=2$, we get:
\begin{equation*}
 \tilde{p}^{j'}_2(t)  e^{\Phi^{j'}(t)} =\frac{1}{2} \left(  \Phi^{j'}(t)\right)^2 + \int_0^t e^{ -\Phi^1(t-u)} \Phi^1(t-u) \phi^{j'}(u)\mathrm du ,
 \end{equation*}
 and so on...

\paragraph{Remark.}
Clearly, Proposition \ref{counts-prob} can be generalized to multi-dimensional multi-temporal counts  $\tr \tilde{\bold{N}}^{j'}=(\tr \tilde{N}^{j'}(t_1), \ldots,\tr \tilde{N}^{j'}(t_n)  ) $ of dimension $d\times n$. Indeed, the probability distribution of the process $\tilde N^{j'}$ is infinitely divisible and one can adapt Equation \eqref{prob-recur} to $(\N^d)^n$ (instead of $\N^d$), providing summation formulas over increasing paths passing through the different  points $\tilde{N}^{j'}(t_r); r=1:n$.

%%%%%%%%%%%%%%%%%%%%%%%%%%%%%%%%%%%%%%%%%%%%%%%%%%%%%%%%%%%%%%

\section{Integral Equations of the First Two Multi-temporal Moments} \label{section_int_equation_first_moments}

The Laplace transform  is an important tool that characterizes all properties of probability distributions of non negative random variables. This is also the case of the multivariate multi-time distributions of a Hawkes process 
with Laplace transform $\tilde{L}^{j',(n)}(a,t)$ satisfying Equation \eqref{Eq-tildeL}. Hence, if $a= (a^m_j : \  m=1:n, j=1:d)$ and $t=(t_1,\ldots,t_n)$, 
for any non negative integer valued set of triplets $Q=\{ (q_1,m_1,j_1), \cdots, (q_r,m_r,j_r)\}$ such that $q_k\geq 1, \ 1\leq m_k\leq n, 1\leq j_k\leq d ), \ k =1:r$ and $|q|= \sum_{k=1}^r q_k$, we have the following formula for the $Q^{th}$ moment of $\tilde{N}^{j'}, \ j'=0:d$: 
\begin{align*} M^{j',(Q)}(t)&= \E_{(\phi^{j'}_\bu, \phi)} \left( \prod_{k=1}^r [\tilde{N}^{j'}_{j_k} (t_{m_k})]^{q_k}  \right)\\
&=
   (-1)^{|q|} \left. \partial_Q \tilde{L}^{j',(n)}( a,t)    \right|_{a=0},
\end{align*}
   with 
   $$\partial_Q= \partial^{|q|}_{ (a^{m_1}_{j_1})^{q_1 } \ldots  (a^{m_r}_{j_r})^{q_r}}.$$   
Here, $Q$ represents the set $\mathcal{Q}$ of size $|q|$ containing $q_1$ copies of $a^{m_1}_{j_1} , \ \ldots,   \ q_r$ copies of $a^{m_r}_{j_r}$; hence, $\partial_Q$ can be replaced by $\partial_\mathcal{Q}$ without any risk of confusion.

 Therefore using the Fa\`a di Bruno formula for partial derivatives of 
composed functions, namely  $\tilde{L}^{j',(n)}(a,t)=\exp(\theta^{j'}(a,t))$ satisfying $\tilde{L}^{(j',n)}(0_{[d]},t)\equiv1$ for all $t, j', n$, we get

\begin{equation}
	\label{moments}
 M^{j',(Q)}(t)= (-1)^{|q|}   \left. 
 \sum _{\pi \in \mathcal{P} } \prod_{\tilde{Q}_l \in \tau }\ \partial_{\tilde{Q}_l} \theta^{j'}(a,t) \right|_{a=0},
 \end{equation}
 with   $\pi$ running  through the set $\mathcal{P}$ of all partitions $\{\tilde{\mathcal{Q}}_1,\cdots,\tilde{\mathcal{Q}}_p\}$ of the set $ \mathcal{Q}$. For example in the case of $\partial^3_{(a_1^1)^2 a^2_2}$ we have $\mathcal{Q}=\{a^1_1,a^1_1, a^2_2\}$. Moreover, as the  components $\tilde{\mathcal{Q}}_l$ of the partition $\pi$ may contain several copies of a same element $a^m_j$,  Equation \eqref{moments} simplifies a lot in the case where some of the $q_m $ are strictly greater than 1, such that the number of distinct terms then reduces to 
 $\frac{|q|!}{\prod_{k=1}^r q_k! (k!)^{q_k} }$. 
 
 In what follows, we  focus on the  derivation of formal and/or implicit equations related to the first and second order moments.

 %%%%%%%%%%%%%%%%%%%%%%%%%%%%%%%%%%%%%%%%%%%%%%%

\subsection{Equations of First Order Moments}

\begin{proposition}
	Consider the set of vector means  $\tilde{M}^{j',(1)}(t)=E_{(\phi^{j'}_\bu,\phi)}\left(\tilde{N}^{j'}(t)\right), \ j'=0:d,$ and the   $d\times d$ matrix of base vector means   $\tilde{M}^{(1)}(t)=\left[\tilde{M}^{1,(1)},\cdots,\tilde{M}^{d,(1)}\right]$. Let $\Phi(t)= \int_0^t \phi(s)\mathrm ds$ and $\Lambda^0(t)=\int_0^t \lambda^0(s)\mathrm ds$.
Then,  $\tilde{M}^{(1)}$ is solution of the following matrix integral equation
    \begin{equation}\label{implicit-mean}
        \tilde{M}^{(1)} (t)
        =  \Phi(t)+   \tilde{M}^{(1)}*\phi (t),  
        \end{equation}
 that corresponds to
 \begin{equation} 
 	\tilde{M}^{j',(1)}_l (t)
 	=  \Phi^{j'}_l(t)+   \sum_{j=1}^d \int_0^t\tilde{M}^{j,(1)}_l(t-s)\phi^{j'}_j(s) \mathrm ds,  
 \end{equation}
for  $1\le j',l\leq d$.    Moreover,    
  \begin{equation} \label{implicit-mean-0}     
  \tilde{M}^{0,(1)} (t)=\Lambda^0 (t)+ \tilde{M}^{(1)}*\lambda^0 (t).
    \end{equation}
   
\end{proposition}

\begin{proof}
Let $k=1, \ t\in \R_{+*}$ and $\tr a=(a_1,\ldots,a_d) \in (\R_+)^d$, then Equation  \eqref{laplacek} yields  for  $ j'=0:d$
\begin{equation}
	\tilde{L}^{j',(1)}\left(a,  t\right)=
	\tilde{p}^{j'}_0(t) \exp\left(\sum_{j=1}^d
	e^{- a_j} \tilde{L}^{j,(1)}(a, .) * \ 
	\phi^{j'}_{j}   (t) 
	\right) .
\end{equation}
Then, applying  Equation \eqref{moments},
we get  \[\tilde{M}^{j',(1)}(t)=-\partial^1_{a_l}\tilde{L}^{j',(1)} \left(0_{[d]},t\right).\]
Hence, for each $j'=1:d$, 
\begin{equation}
	\partial^1_{a_l}\tilde{L}^{j',(1)} \left(a,  t\right)= \tilde{L}^{j',(1)} \left(a,  t\right)\left( \sum_{j=1}^d \left( (-\mathbf 1_{l=j}) e^{- a_j} \tilde{L}^{j,(1)}(a, .) +
	e^{- a_j} \partial_{a_l}^1\tilde{L}^{j,(1)}(a, .)\right) * \phi^{j'}_{j}   
	\right)(t),
\end{equation}
and, consequently,
\begin{eqnarray*}
	\partial^1_{a_l}\tilde{L}^{j',(1)} \left(0_{[d]},  t\right)&=& \tilde{L}^{j',(1)} \left(0_{[d]},  t\right) \sum_{j=1}^d \left( (-\mathbf 1_{l=j})   \tilde{L}^{j,(1)}(0_{[d]}, .) +
	\mathbf 1\times \partial_{a_l}^1\tilde{L}^{(1)}(0_{[d]}, .)\right) * \phi^{j'}_{j}   (t)\\
	&=&  \sum_{j=1}^d \left( (-\mathbf 1_{l=j})  - \tilde{M}^{j,(1)}_l\right) * \phi^{j'}_{j} (t).
\end{eqnarray*}
since $\tilde{L}^{j',(1)} \left(0_{[d]},  t\right)\equiv 1$, for $j'=0:d$. Therefore, for all $j' = 1:d$ and $l = 1:d$ ,  we obtain
\begin{equation}
\tilde{M}^{j',(1)}_l (t)= \Phi^{j'}_l (t)+ \sum_{j=1}^d \left(\tilde{M}^{j,(1)}_l  * \phi^{j'}_{j}\right) (t).	
	\end{equation}
 
The latter formula can be rewritten under a matrix form and  replacing $\phi^{j'}_{\bu}$ with $\lambda^0$ yields the  second part of the proposition.   
\end{proof}

%%%%%%%%%%%%%%%%%%%%%%%%%%%%%%%%%%%%%

\subsubsection{Explicit Formulas for First Order Moments  }

Actually, the previous implicit/integral equations  for the first order moments can be developed  explicitly.

Remind that the $r^{th}$-convolution of  non-negative $d\times d$-matrix function $\phi(t)=(\phi^j_i)(t)$ on $\R_+$ is well defined recursively  as $\phi^{(*)r}=\phi^{(*)(r-1)} * \phi$, e.g.  $(\phi^{(*)2})^{j'}_j (t)=\sum_{r=1}^d \phi^r_j*\phi^{j'}_r(t)= \sum_{r=1}^d \phi^{j'}_r*\phi^r_j (t)$. We also convene that for $r=0$, we have 
$$\phi^{(*)0} (t)=\text{diag}(\delta_0(\mathrm dt),\ldots, \delta_0(\mathrm dt))=I_d \delta_0(\mathrm dt), $$ 
where  $\delta_0(\mathrm dt)$ is the Dirac distribution (or generalized function) supported by $0 \in \R_+$.

\begin{proposition}\label{mean}
Finite or infinite, the mean function $\tilde{M}^{(1)}(t)$ is written
\begin{equation}
\label{Mean}
\tilde{M}^{(1)}(t)=   \Phi*\left(\sum_{r=0}^{\infty}\phi^{(*)r}\right)(t))=   \left(\sum_{r=0}^{\infty}\phi^{(*)r}\right)*\Phi (t)
\end{equation}
with derivative 
\begin{equation}
	\tilde{m}^{(1)}(t)=   \sum_{r=1}^{\infty}\phi^{(*)r} (t).
\end{equation}\label{Mean-density}
Similarly, we have  
\begin{equation}\label{means}
\tilde{M}^{0,(1)}(t)=   \left(\sum_{r=0}^{\infty}\phi^{(*)r}\right) *\Lambda^0  (t),  \ 
\text{ with     derivative}  \  \ 
  \tilde{m}^{0,(1)}(t)=    \left(\sum_{r=0}^{\infty}\phi^{(*)r}\right)*\lambda^0 (t).
  \end{equation}

\noindent Moreover,   if  $\| \phi \|_{L^1}= \sup(\| \phi_i^j \|_{L^1}:  1\leq i,j\leq d ) \leq \alpha$ with $ d \alpha  <1$ , then $\tilde{M}^{(1)}(t) $ is finite  for  all $t\geq 0$.   
Additionally, if  $\lambda^0 \in L^1,$ then $\tilde{m}^{0,(1)} \in L^1$ with  primitive  $\tilde{M}^{0,(1)}(t)$  finite for all $t\geq 0$.

\end{proposition}

\begin{proof}
	Since all functions $\phi^{j'}_j$ are non negative, one can expand by iteration Equation \eqref{implicit-mean} and  get  the following   increasing converging series of matrix functions (whose components may be finite or infinite):
	$$ \tilde{M}^{(1)}(t)= \Phi* \left(\sum_{r=0}^\infty \phi^{(*)r} \right) (t).$$

Moreover, observe that using differential rules for convolution yields $  \Phi^{r}_l*\phi^l_j = \phi^{r}_l*\Phi^l_j 	$, which implies  that commutativity of  matrix convolution $\Phi*\phi=\phi*\Phi$  holds in this specific case.  Iterating the process proves the commuting property   $\Phi*\phi^{(*)r}=\phi^{(*)r}*\Phi$ for all $r\geq 0$.
	
Next, if 	 $\tilde{M}^{(1)}$ is finite on an interval $I$ of $\R_+$, then it is differentiable with derivative 
	$$ \tilde{m}^{(1)}(t)= \sum_{r=1}^\infty \phi^{(*)r}  (t).$$
	
and, due to Equation \eqref{implicit-mean-0},   if  $\tilde{M}^{0,(1)}$  is finite, then it is differentiable with derivative:	
$$ \tilde{m}^{0,(1)}(t)= \lambda^0(t) + \tilde{m}^{(1)}*\lambda^0 (t)= \left(\sum_{r=0}^\infty \phi^{(*)r} \right) *\lambda^0 (t).$$	
	
We now prove that the given conditions  are sufficient to ensure finite values of the means and their derivatives. For that purpose, let us recall the Young's inequalities for convolution in Lebesgue  spaces $L^p$: For any  $f \in L^p, \ g\in L^q$ and 
$1/p +1/q=1+1/r$, we have $f*g \in L^r $ and
 $ \|f*g\|_{L^r} \leq  \|f\|_{L^p}  \|f\|_{L^q} $.
 
 Thus, in the particular case where $r=p=q=1$ 
 and  $\|\phi^j_i\|_{L^1} \leq \alpha  , 1\leq i,j \leq d$, we get $\|\sum_{l=1}^d \phi^l_i * \phi^j_l \|_{L^1} \leq d \alpha^2 = (d \alpha)^2/d$.
 
 Iterating the procedure,  we get $\|\phi^{(*)r}\|_{L^1} \leq (d \alpha)^r/d $ for $r\geq 1 $.  For $d\alpha <1$ , we  finally get  
 $\| \sum_{r=1}^\infty \phi^{(*)r}\|_{L^1} \leq \sum_{r=1}^\infty (d \alpha )^r/d = \alpha /(1-d\alpha ) <\infty
$. 
 
Consequently the series $\sum_{r=1}^\infty \phi^{(*)r}$ converges in $L^1$ to  $\tilde{m}^{(1)}$ and, hence, all its components are almost everywhere finite (and integrable).

Clearly, $\Phi$ the primitive of $\phi$ is therefore bounded  and differentiable on any interval, the mean $\tilde{M}^{(1)}$  is also differentiable and   finite for all $t\geq 0$. 
Similar results hold true for  $M^{0,(1)}$   if $\lambda^0$ is in $L^1$.
\end{proof}

%%%%%%%%%%%%%%%%%%%%%%%%%%%%%%%%%%%%%%

\subsubsection{Examples and Remarks}\label{example1}

\begin{enumerate}

\item We start with the classical example of constant $\phi$, that is $\phi(t)= W \chi(t)$ where $W=(w^j_i)$ is a constant matrix and   $\chi(t)={\mathbf 1}_{[0,\infty[}(t)$. Hence $\phi \in L^\infty$ and condition of Proposition \ref{mean} holds true. Then, we observe that $\chi^{(*)r} (t)=  t^{r-1}/(r-1)!$ and, therefore, we get 
$\ \tilde{m}^{(1)}(t)= \sum_{r=1}^\infty \phi^{(*)r)} (t)= W ( \sum_{k=0}^\infty W^k t^k/k!)= We^{tW}$. 

Consequently, since $\Phi(t)=t W $, we also have $ \tilde{M}^{(1)}(t)= \Phi* \left(\sum_{r=0}^\infty \phi^{(*)r} \right) (t)=  tW + W^2 \left(\sum_{r=0}^\infty W^r t^{r+2}/(r+2)!\right)= e^{tW} - I_d$.

It is worth noticing that using the Cauchy formula related to $r-$repeated integrals namely $I^{(r)}[g](t)= \int_0^t \int_0^{u_{r}}.... \int_0^{u_1} g(u_1) du_1...du_{r}= \int_0^t g(u) (t-u)^{(r-1)}/(r-1)! du$, we get for an initial igniting intensity $\lambda^{0}$, the formulas  $\tilde{m}^{0,(1)}(t)=\sum_{r=0}^\infty W^r I^{(r)}[\lambda^0] (t)$ and 
 $\tilde{M}^{0,(1)}(t)=\sum_{r=0}^\infty W^r I^{(r+1)}[\lambda^0] (t)$. 

 The latter expression can be useful in the case of the existence of recursive formulas for the calculus of repeated integrals of function such as  exponential, sine, etc.

\item The condition of Proposition \ref{mean} holds for $\phi^j_i(t)= \alpha^j_i e^{-\beta^j_i t}$, if $ d \alpha^j_i  <\beta^j_i   , \ \ 1\leq i,j \leq d$  and 
$\lambda^0(t)$ is a constant vector.

\item One can also consider the interesting class of $\Gamma$ intensities admitting many different shapes. For example if 
$\phi^j_i(t)= w^j_i f(t; \kappa,\theta)= w^j_i \frac {t^{\kappa-1} e^{-{t/\theta}}}{\Gamma (\kappa)\theta^{\kappa}}$, then using the stability of the $\Gamma$ probability distribution family for convolution, we get for $W=\left(  w^j_i\right)_{i,j}$ the  expression
$$\phi^{(*)r)} (t)= W^r f(t; r\kappa,\theta).$$ 
Therefore,  we obtain    $\ \tilde{m}^{(1)}(t)= \sum_{r=1}^\infty \phi^{(*)r)} (t)= \sum_{r=1}^\infty \frac{W^r}{ \theta \Gamma (r \kappa ) } \left(\frac{t}{\theta}\right)^{r\kappa -1} e^{-{t/\theta }} $
and 

$ \tilde{M}^{(1)}(t)= 
\sum_{r=1}^\infty \frac{W^r}{ \Gamma (r\kappa ) } \gamma(  r\kappa, t/\theta)$,  
with  $\gamma$ being the incomplete gamma function.
 
 \begin{itemize}	
\item	In particular for $\kappa=1$ we have 

 $\tilde{m}^{(1)}(t)= 
	\frac{e^{-t/\theta }}{t}  \left( e^{ W  \frac{t}{\theta} }  - I_d\right)  $
	and 	
$	\tilde{M}^{(1)}(t)=   
	\sum_{r=1}^\infty \frac{W^r}{ r! } \gamma( r,  t/\theta)$
	
\item If $\kappa=2$ and $W$ has a square root say $W^{1/2}$ (e.g if  $W$ is positive definite), then 
  $\sum_{r=1}^\infty \phi^{(*)r)} (t)=
\frac{e^{-t/\theta }}{t}  \left( \cosh( W^{1/2}  \frac{t}{\theta})  - I_d \right)$.
 \end{itemize}
 
 In any case, one has to use the  Jordan form  of $ W $ and $ W^{1/2}$ to make easier the computation (or approximation)  of the series. 
 
 \item One can arguably consider that $\sum_{r=0}^\infty \phi^{(*)r)} (t)$ is simply the inverse (in the both senses of matrix product and convolution product) of the matrix valued generalized function $I_d\delta_{0} (\mathrm ds) -\phi(s)\mathrm ds$.

 \end{enumerate}

%%%%%%%%%%%%%%%%%%%%%%%%%%%%%%%%%%%%%%%%%%%%%%%%%%%%%%%%%

\subsection{Second Order Moments} 
 
Taking $k=2 , a=(a^1_\bu, a^2_\bu)$ and $t=(t_1,t_2)$ in Equation \eqref{laplacek} yields for $j'=0:d$
\begin{eqnarray}\label{laplace2}
\tilde{L}^{j',(2)} \left(a,  t\right) & =& \tilde{p}^{j'}_0(t_2)
	\exp\left( \sum_{j=1}^d
	\int_{0}^{t_1} 
	e^{- (a^1_j+a^2_j)} \tilde{L}^{j,(2)}(a,  (t_1-u,t_2-u)) 
	\ \  \phi^{j'}_{j}(u)\mathrm du \right)\\
	& \times & \exp\left( \sum_{j=1}^d
	\int_{t_1}^{t_2} 
	e^{- a^2_j} \tilde{L}^{j,(1)}(a^2_\bu,  (t_2-u)) 
	\ \  \phi^{j'}_{j}(u)\mathrm du
	\right).
\end{eqnarray}

\noindent Thus, according to moment formula \eqref{moments} with $Q=( (1,m_1,k), (1,m_2,l))$ and for $j'=0:d$, we obtain:

$$    \tilde{M}^{j',(2)}_{k,l}(t_{m_1}, t_{m_2})= \E_{(\phi^j_\bu,\phi)} \left(\tilde{N}^{j'}_k (t_{m_1}) \tilde{N}^{j'}_l (t_{m_2}) \right)= \left. \partial^2_{ a^{m_1}_{k} a^{m_2}_{l}} \tilde{L}^{j',(2)} (a,  t)   \right|_{a=0}.$$

We can also write expressions of the covariance function satisfying:
$$\tilde{C}^{j'}_{k,l}(t_{m_1},t_{m_2})=  \tilde{M}^{j',(2)}_{k,l}(t_{m_1}, t_{m_2})
- \tilde{M}^{j',(1)}_{k}(t_{m_1}) \tilde{M}^{j',(1)}_{l}(t_{m_2}).$$
Below, we consider two cases: the case corresponding to a single time where we are interested in the covariance function of the %, i.e.  $t_1=t_2$ ( i.e for the 
$d$-vector $\tilde{N}^{j'} (t_1)$, and the case of distinct times $t_1 < t_2$ where we consider the covariance function of the $2d$-vector
$\tr (\tr \tilde{N}^{j'} (t_1),\tr\tilde{N}^{j'} (t_2)  )$   with covariance matrix  partitioned as follows
\begin{equation*}
	\tilde{C}^{j'}(t_1,t_2)=  \begin{pmatrix}
		\tilde{C}^{j'}(t_1,t_1) &  \tilde{C}^{j'}(t_1,t_2)\\
		\tilde{C}^{j'}(t_2,t_1) & \tilde{C}^{j'}(t_2,t_2)
	\end{pmatrix}.	
\end{equation*}
In this partitioned matrix, the diagonal blocks correspond to single times $t_1$ and $t_2$. 

\paragraph{Remarks} 
\begin{itemize}
\item The case  $m_1=m_2$  means that we must only deal with a single time $t_1$ (resp. $t_2$) and hence to consider   $a^1_\bu$ (resp. $a^1_\bu +a^2_\bu)$ instead of $a=(a^1_\bu,a^2_\bu)$, and then we must simply deal with $\tilde{L}^{j',(1)} $ instead of $\tilde{L}^{j',(2)}$. This finally leads  
to twice differentiate $\tilde{L}^{j',(1)}(a, s)$ with $a=a^1_\bu$ and $s=t_1$ (resp.  $a=a^1_\bu+a^2_bu$ and $s=t_2$).

\item The case   $m_1\neq m_2$,  say $m_1=1, m_2=2$,  implies that $t_1\ne t_2$ and in this case  
$  \partial^2_{ a^{1}_{j_1} a^{2}_{j_2}  } \tilde{L}^{j',(2)}  (a, t )  $ simplifies a lot since the third term in Equation \eqref{laplace2} only depends on $a^2_\bu$.	 

\item Finally, note that 		$\tilde{C}^{j'}(t_2,t_1)= \tr \tilde{C}^{j'}(t_1,t_2)$.
\end{itemize}

%%%%%%%%%%%%%%%%%%%%%%%%%%%%%%%%%%%%%%%%%%%%%% 
 
\subsubsection{Case of a Single Time}
\begin{proposition} For a single time $t\geq 0$ and  $j'=0:d$, we have 
\begin{equation}
	\tilde{C}^{j'}(t)=\sum_{j=1}^d \left[   \left(e_j +
	\tilde{M}^{j,(1)}\right) \tr\left( e_j +\tilde{M}^{j,(1)} \right) + \tilde{C}^{j}  \right]*\phi^{j'}_j (t).  
\end{equation}
Moreover for each $j',k,l \in 1:d$,  the vector  of distinct covariance components 
 $ \tilde{C}_{k,l}^{\bu}(t)= \tr (\tilde{C}^{1}_{k,l},\ldots, \tilde{C}^{d}_{k,l})(t)$	satisfies the linear integral equation
 \begin{equation} 
 	\tilde{C}_{k,l}^{\bu}(t)= R_{k,l}^{\bu}(t)	 +  \tr\phi*\tilde{C}_{k,l}^{\bu} (t)
 \end{equation}
 whose solution is 
\begin{equation}\label{cov1-2}
	\tilde{C}_{k,l}^{\bu}(t)=\tr \left( \sum_{r=1} \phi^{(*)r}\right) *\tilde{R}_{k,l}^{\bu}   (t)
\end{equation}   
with components $\tilde{R}^{j}_{k,l}(u)=\tr e_k \left(e_j +
\tilde{M}^{j,(1)}(u)\right) \tr\left( e_j +\tilde{M}^{j,(1)}(u) \right)e_l$ and  $R_{k,l}^{\bu}=   \tr\phi*\tilde{R}^{\bu}_{k,l}$. 
\end{proposition}	

\begin{proof}
For  a single time  $t$ and a  vector $a$, we have for $j'=0:d$
$$\E_{(\phi^{j'},\phi)} \left( \tilde{N}^{j'}_k (t)  \tilde{N}^{j'}_l (t)\right)=  \tilde{M}^{j',(2)}_{k,l}(t)=  \left. \partial^2_{a_k,a_l}\tilde{L}^{j',(1)} (a,  t)  \right|_{a=0}.$$

Using  Formula \eqref{mean}, 
\begin{align*}
 \partial^1_{a_l}\tilde{L}^{j',(1)} \left(a,  t\right) &=	    \tilde{L}^{j',(1)} (a,t)\times  U_l^{j'}(a,t) \\
\partial^2_{a_k,a_l} \tilde{L}^{j',(1)} \left(a,  t\right) &= 	 \tilde{L}^{j',(1)} \left(a,  t\right) \left[ U_k^{j'}(a,t) U_l^{j'}(a,t) + \partial^1_{a_k} U_l^{j'}(a,t)\right],
\end{align*}
with
$ U^{j'}_l(a,t)=  \sum_{j=1}^d \left( (-1_{l=j}) e^{- a_j} \tilde{L}^{j,(1)}(a, .) +
e^{- a_j} \partial_{a_l}^1\tilde{L}^{j,(1)}(a, .)\right) * \phi^{j'}_{j}  (t)$.
So, we have 
\begin{eqnarray*}
\partial^1_{a_k} U_l^{j'}(a,t)&=&  \sum_{j=1}^d  \left[\partial^1_{a_k}\left( (-1_{l=j}) e^{- a_j} \tilde{L}^{j,(1)}(a, .) +
	e^{- a_j} \partial_{a_l}^1\tilde{L}^{j,(1)}(a, .)\right)\right] * \phi^{j'}_{j}   (t) \\
	&=&    \sum_{j=1}^d  -1_{l=j} \left( -(1_{k=j}) e^{- a_j} 
 \tilde{L}^{j,(1)}(a, .)  +
	             e^{- a_j} \partial^1_{a_k}\tilde{L}^{j,(1)}(a, .) \right)* \phi^{j'}_{j} (t)\\
	&+&    \sum_{j=1}^d \left( (-1_{k=j})	e^{- a_j} \partial_{a_l}^1\tilde{L}^{j,(1)}(a, .)
	+ 	e^{- a_j} \partial_{a_ka_l}^{2}\tilde{L}^{j,(1)}(a, .)\right) * \phi^{j'}_{j} (t). 	
\end{eqnarray*}
Setting $a=0$ and recalling that for  all $s\geq 0$ and $j'=0:d$, we have  
$$\tilde{L} ^{j',(1)}(0,s)\equiv 1, \  \partial^1_{a_k}\tilde{L}^{j',(1)}(0, s)=U^{j'}_k(0,s)= - \tilde{M}^{j',(1)}_k (s), \text{and}  \ \   \partial^2_{a_k,a_l}\tilde{L}^{j',(1)}(0, s)= \tilde{M}^{j',(2)}_{k,l}(s),$$  
we get
$$\partial^1_{a_k} U_l^{j'}(0,t)= \sum_{j=1}^d \left( 1_{l=j} 1_{k=j} +
1_{l=j}\tilde{M}^{j,(1)}_k (.)  + 1_{k=j}\tilde{M}^{j,(1)}_l (.) 
+ 	 \tilde{M}^{j,(2)}_{kl} (.) \right) * \phi^{j'}_{j} (t)$$
and we finally obtain  for $1\leq k,l \leq d$
$$\tilde{M}^{j',(2)}_{k,l}(t)= \left(\tilde{M}^{j',(1)}_k  \tilde{M}^{j',(1)}_l  + 1_{k=l}\Phi^{j'}_k + 
 \tilde{M}^{l,(1)}_k*\phi^{j'}_l  +\tilde{M}^{k,(1)}_l*\phi^{j'}_k + \sum_{j=1}^d\tilde{M}^{j,(2)}_{k,l}*\phi^{j'}_j \right)(t). $$

Next, let us consider the  $d\times d$ components of the matrix covariance function at a single time $t$ 
$$\tilde{C}_{k,l}^{j'}(t)= \tilde{M}^{j',(2)}_{k,l}(t)-\tilde{M}^{j',(1)}_k (t) \tilde{M}^{j',(1)}_l (t).$$

\noindent For $j'=0:d$, the latter expression can thus be written 
\begin{equation}\label{cov1-1}
 \tilde{C}^{j'}_{k,l}(t)= \sum_{j=1}^d \left[  \left(1_{j=k}+
 \tilde{M}^{j,(1)}_k\right) \left( 1_{j=l}  +\tilde{M}^{j,(1)}_l \right) + \tilde{C}^{j}_{k,l} \right]*\phi^{j'}_j (t) .  
\end{equation}
Next, if we restrain index $j'$ to $\{1,...,d\}$, we also  observe vector integral equations involving $d$ functions, that is for each pair $(k,l)$, the component $ \tilde{C}_{k,l}^{\bu}(t)= \tr (\tilde{C}^{1}_{k,l},\ldots, \tilde{C}^{d}_{k,l})(t)$,  satisfies an integral equation similar to that of the vector of means, that is 
\begin{equation} 
	\tilde{C}_{k,l}^{\bu}(t)= R_{k,l}^{\bu}(t)	 +  \tr\phi*\tilde{C}_{k,l}^{\bu} (t).
\end{equation}
Since $(\tr\phi)*(\tr\phi)=\tr (\phi*\phi)$ and all functions are non negative, the solution, finite or infinite, corresponds to the series:
\begin{equation}
	\tilde{C}_{k,l}^{\bu}(t)=\tr \left( \sum_{r=0} \phi^{(*)r}\right) *R_{k,l}^{\bu}  (t). 
\end{equation}  

Furthermore, we notice that Equation \eqref{cov1-1} can also be written 
\begin{equation}
\tilde{C}_{k,l}^{j'}(t) =\sum_{j=1}^d \left[  \tr e_k  \left(e_j +
 \tilde{M}^{j,(1)}\right) \tr\left( e_j +\tilde{M}^{j,(1)} \right)e_l +\tr e_k\tilde{C}^{j} e_l  \right]*\phi^{j'}_j (t).,   
\end{equation}
which takes the matrix  form,  
\begin{equation}
	\tilde{C}^{j'}(t)=\sum_{j=1}^d \left[   \left(e_j +
	\tilde{M}^{j,(1)}\right) \tr\left( e_j +\tilde{M}^{j,(1)} \right) + \tilde{C}^{j}  \right]*\phi^{j'}_j (t).  
\end{equation}
To end the proof, let us observe that  $R_{k,l}^{\bu}=   \tr\phi*\tilde{R}^{\bu}_{k,l}$ with 
 $\tilde{R}^{\bu}_{k,l} $ having components $\tilde{R}^{j}_{k,l}(u)=\tr e_k \left(e_j +
\tilde{M}^{j,(1)}(s)\right) \tr\left( e_j +\tilde{M}^{j,(1)}(s) \right)e_l$.

\end{proof}

%%%%%%%%%%%%%%%%%%%%%%%%%%%%%%%%%%%%%%%%%%%%%%%%

\subsubsection{Covariance for  Distinct Times}

We now deal with  the off diagonal covariance structures related to distinct times for indices $j'=0:d$
$$\tilde{C}^{j'}(t_1,t_2)= \E_{(\phi^{j'}_\bu,\phi)}\left(\tilde{N}^{j'} (t_1) \tr\tilde{N}^{j'}(t_2)\right)
- \tilde{M}^{j',(1))}(t_1) \tr \tilde{M}^{j',(1))}(t_2).$$

\begin{proposition} For  $t=(t_1,t_2) $ with $t_1< t_2$ and convening again that $u_{[2]}=(u,u)$ if $u\in \R_+$, then for $j'=0:d$,   the sub-covariance structure  $\tilde{C}^{j'}(t)$
satisfies the matrix integral equation: 
\begin{equation}\label{cov2-1}
	\tilde{C}^{j'}(t) = \sum_{j=1}^d \int_0^{t_1}  \left[ \left(e_j +\tilde{M}^{j,(1)}(t_1-u)\right)\tr
	\left( e_j+\tilde{M}^{j,(1)}(t_2-u)\right)   +   \tilde{C}^{j}(t-u_{[2]}) \right]\ \phi^{j'}_j (u) \mathrm du .   
\end{equation}
Moreover, if $\tilde{C}^{\bu}_{k,l}=\tr (\tilde{C}^{1}_{k,l},\ldots,\tilde{C}^{d}_{k,l})$ and 
$R^{\bu}_{k,l} =\tr (R^{1}_{k,l}, \ldots,R^{d}_{k,l})$ 
with 
$$ R^{j'}_{k,l} (t) = \sum_{j=1}^d \int_0^{t_1}  \tr e_k \left(e_j +\tilde{M}^{j,(1)}(t_1-u)\right)\tr
\left( e_j+\tilde{M}^{j,(1)}(t_2-u)\right) e_l \phi^{j'}_j (u) \mathrm du, $$ then each index pair $(k,l)$ gives rise to  independent vector integral equations  
\begin{equation}\label{eq-cov-2}
	\tilde{C}^{\bu}_{k,l}(t) = R^{\bu}_{k,l} (t) + \int_0^{t_1}  \tr \phi(u)   \tilde{C}^{\bu}_{k,l}(t-u_{[2]}) \mathrm du.    
\end{equation}

\end{proposition}
\begin{proof}

We now consider the case  $Q=\{ (1,1,k), (1,2,l) \}$, that is 
$$ \tilde{M}^{j',(2)}_{k,l}(t)= \left. \partial^2_{a^1_k a^2_l} \tilde{L}^{j',(2)} \left(a,  t\right)\right|_{a=0} $$  
  
\noindent with the Laplace transform $\tilde{L}^{j',(2)}\left(a,  t\right)$ given by Equation \eqref{laplace2}. 
 We have\!: 
 $$ \partial^1_{a^2_l}\partial^1_{a^1_k}\tilde{L}^{j',(2)}\left(a,  t\right) =	\tilde{L}^{j',(2)} (a,  t)   \left( V^{j'}_l(a,t)U^{j'}_k(a,t) + \partial^1_{a^2_l} U^{j'}_k(a,t)\right)$$
 where 
 $$
U^{j'}_k(a,t) = \sum_{j=1}^d  e^{- (a^1_j+a^2_j)}
\int_{0}^{t_1}  \left( -1_{k=j} +
\partial^1_{a^1_k} \right)\tilde{L}^{j,(2)}(a,  (t-u_{[2]}) \phi^{j'}_{j}(u)\mathrm du$$
satisfies when $a=0$
$$U^{j'}_k(0,t)= \sum_{j=1}^d \int_{0}^{t_1}  \left( -1_{k=j}  -\tilde{M}^{j,(1)}_k(t_1 -u) \right) \phi^{j'}_{j}(u) \mathrm du              
=   - \tilde{M}^{j',(1)}(t_1).$$
Moreover,
\begin{align*}
\partial^1_{a^2_l} U^{j'}_k(a,t)  = \sum_{j=1}^d e^{- (a^1_j+a^2_j)}
 \int_{0}^{t_1} & \left( 1_{l=j} 1_{k=j} -1_{k=j} \partial^1_{a^2_l}     -1_{l=j} \partial^1_{a^1_k} +    \partial^2_{a^2_l a^1_k} \right)\tilde{L}^{j,(2)} (a,  t-u_{[2]} ) \\[6pt]
 &\times \phi^{j'}_{j}(u)  \mathrm du
\end{align*}

 \noindent satisfies when $a=0$
\begin{align*}
\partial^1_{a^2_l} U^{j'}_k(0,t) = \,\sum_{j=1}^d 
 \int_{0}^{t_1} & \left( 1_{l=j} 1_{k=j} + 1_{k=j}\tilde{M}^{j,(1)}_l(t_2-u)  + 1_{l=j}\tilde{M}^{j,(1)}_k(t_1-u)     + \tilde{M}^{j,(2)}_{k,l} ( t-u_{[2]} )\, \right)\\&\times\phi^{j'}_{j}(u)  \mathrm du.  
\end{align*}

\noindent We also have  
\begin{eqnarray*}
V^{j'}_l(a,t) &=&\sum_{j=1}^d e^{- (a^1_j+a^2_j)}
	\int_{0}^{t_1} \left( -1_{l=j} +
	\partial^1_{a^2_l} \right)\tilde{L}^{j,(2)}(a,  (t - u_{[2]}) ) 
 \phi^{j'}_{j}(u)  \mathrm du  \\
    & & + \sum_{j=1}^d e^{- a^2_j}
 \int_{t_1}^{t_2} \left( -1_{l=j} + \partial^1_{a^2_l} \right)\tilde{L}^{j,(1)}(a^2_\bu,  t_2-u) 
 \phi^{j'}_{j}(u)  \mathrm du  
\end{eqnarray*}
 with
 \begin{eqnarray*}
V^{j'}_l(0,t)&=& \sum_{j=1}^d \left[ 
	\int_{0}^{t_1}   \left(-1_{l=j}  -\tilde{M}^{j,(1)}_l(t_2-u) \right)\phi^{j'}_{j}(u)  \mathrm du \right.\\
	& & + \left. \int_{t_1}^{t_2} \left( -1_{l=j} -\tilde{M}^{j,(1)}_l(t_2-u) \right)\phi^{j'}_{j}(u)  \mathrm du \right]\\
&=& -\tilde{M}^{j',(1)}_l(t_2). 
\end{eqnarray*}
All in all,  we finally get for $j'=0:d$
{\small
\begin{eqnarray*}
\tilde{M}^{j',(2)}_{k,l}(t)&= &\tilde{M}^{j',(1)}_{k}(t_1)\tilde{M}^{j',(1)}_{l}(t_2)
+\sum_{j=1}^d \int_0^{t_1}\Big( \mathbf 1_{l=j} \mathbf 1_{k=j} 
+ \mathbf 1_{k=j}\tilde{M}^{(1),j}_l(t_2\,-u)
\\ & & +
\mathbf 1_{l=j}\tilde{M}^{j,(1)}_k(t_1\,-u)     + \tilde{M}^{j,(2)}_{k,l} (t-u_{[2]})\, \Big)\phi^{j'}_{j}(u)  \,\mathrm du\\[7pt]
%%%%%%%%%%%%%%%%%%%%%%%%%%%%%%%
&= &\tilde{M}^{j',(1)}_{k}(t_1)\tilde{M}^{j',(1)}_{l}(t_2)
+
\sum_{j=1}^d \int_0^{t_1}\tr e_k\Big( e_j\tr e_j +e_j\tr\tilde{M}^{j,(1)}(t_2-u) +\tilde{M}^{j,(1)}(t_1-u)\tr e_j\\ 
&& + \tilde{M}^{j,(2)} (t-u_{[2]})\Big)e_l\ \phi^{j'}_{j}(u) \, \mathrm du\\
\end{eqnarray*}
}

\noindent which can be written under the matrix form
{\small
	\begin{eqnarray*}
\tilde{M}^{j',(2)}(t)&= &\tilde{M}^{j',(1)}(t_1)\tr\tilde{M}^{j',(1)}(t_2)\\
&&+
\sum_{j=1}^d \int_0^{t_1} \left( e_j\tr e_j +e_j\tr\tilde{M}^{j,(1),}(t_2-u) +\tilde{M}^{j,(1)}(t_1-u)\tr e_j +\tilde{M}^{j,(2)} (t-u_{[2]})\right) \phi^{j'}_{j}(u)  \mathrm du.
\end{eqnarray*}
}

\noindent Similarly, the covariance function can be written
\begin{equation}\label{cov2-1}
	\tilde{C}^{j'}(t) = \sum_{j=1}^d \int_0^{t_1}  \left[ \left(e_j +\tilde{M}^{j,(1)}(t_1-u)\right)\tr
	\left( e_j+\tilde{M}^{j,(1)}(t_2-u)\right)   +   \tilde{C}^{j}(t-u_{[2]}) \right] \phi^{j'}_j (u) \mathrm du    
\end{equation}
whose components satisfies for all $1\leq k,l\leq d$ and $j'=0:d$

\begin{equation}\label{eq-cov}
	\tilde{C}^{j'}_{k,l}(t) = R^{j'}_{k,l} (t) + \sum_{j=1}^d \int_0^{t_1}     \tilde{C}^{j}_{k,l}(t-u_{[2]})\ \phi^{j'}_j (u) \mathrm du    
\end{equation}
with 
\begin{equation}\label{R}
 R^{j'}_{k,l} (t) = \sum_{j=1}^d \int_0^{t_1}  \tr e_k \left(e_j +\tilde{M}^{j,(1)}(t_1-u)\right)\tr
\left( e_j+\tilde{M}^{j,(1)}(t_2-u)\right) e_l \phi^{j'}_j (u) \mathrm du.
\end{equation}

Once again, the restriction of index $j'$ to $1:d$, yields a set of  $d$ matrix integral equations \eqref{cov2-1} related to the set of $d$ base matrices.

For each pair of indices $k,l$, let
 $\tilde{C}^{\bu}_{k,l}=\tr (\tilde{C}^{1}_{k,l},\ldots,\tilde{C}^{d}_{k,l})$ and 
$R^{\bu}_{k,l} =\tr (R^{1}_{k,l}, \ldots,R^{d}_{k,l})$,
then Equation \eqref{eq-cov} can be summed up to the following vector integral equation
\begin{equation}\label{eq-cov-2}
	\tilde{C}^{\bu}_{k,l}(t) = R^{\bu}_{k,l} (t) + \int_0^{t_1}  \tr \phi(u)   \tilde{C}^{\bu}_{k,l}(t-u_{[2]}) \mathrm du.    
\end{equation}

\end{proof}

\subsubsection{Explicit Solution}

Here, we provide an explicit solution of Equations \eqref{eq-cov-2}. In this aim, we recourse again to generalized function calculus. Before formulating the proposition, we highlight that  a real valued locally integrable function $\psi$ considered as a measure $\psi(u)\mathrm du$ on $\R_+$, can be extended to a measure 
$\tilde{\psi}(\mathrm du,\mathrm dv)= \psi(u)\delta_u(\mathrm dv)\mathrm du$ supported by the diagonal $\Delta$ of  $\R_+\times \R_+$ and defined  for any measurable function $f$ on $\R_+^2$ as follows
$$\tilde{\psi} (f)= \int_0^\infty \int_0^\infty f(u,v)\psi(u)\delta_u(\mathrm dv)\mathrm du=\int_0^\infty f(u,u)
\psi(u)\mathrm du.$$
Moreover, by extension, we will consider the  $d\times d$ matrix of positive measures   $\tilde{\phi} (\mathrm du,\mathrm dv)=\left( \tilde{\phi}^{j'}_j (\mathrm du,\mathrm dv)\  \right)_{j',j \in 1:d}$
supported by the diagonal $\Delta$ of  $\R_+\times \R_+$.

\begin{proposition}\label{sol-cov-2}
For $t_1 <t_2$ and $k$ and $l$ in $1:d$, the vector  $\tilde{C}^{\bu}_{k,l}(t)=\tr(\tilde{C}^{1}_{k,l},\ldots,\tilde{C}^{d}_{k,l} )(t)$ satisfying Equation \eqref{eq-cov-2}
	has for solution  
\begin{equation}\label{cov-2-2}
\tilde{C}^{\bu}_{k,l}(t_1,t_2)= \tr \widetilde{\left(\sum_{r=1}^\infty \phi^{(*)r} \right)}* \tilde{R}^{\bu}_{k,l}
\ (t_1,t_2)
\end{equation}
where, for $1 \leq j,k,l\leq d$, the vector bivariate  function $\tilde{R}^{\bu}_{k,l}$ has components 
$$ \tilde{R}^{j}_{k,l}(u,v)=  \tr e_k \left(e_j +\tilde{M}^{j,(1)}(u)\right)\tr
\left( e_j+\tilde{M}^{j,(1)}(v)\right) e_l. $$	
Moreover,  for $j'=0$, that is to say $\phi^0=\lambda^0$, we have
\begin{equation}\label{cov-2-2-2}
\tilde{C}^{0}_{k,l}(t_1,t_2)= \tr \left( \widetilde{\left(\sum_{r=0}^\infty \phi^{(*)r} \right)}* \tilde{\lambda}^0 \right)* \tilde{R}^{\bu}_{k,l} (t_1,t_2).
\end{equation}
\end{proposition}

\begin{proof}

With the previous convention related to convolution, we set 
\begin{equation*}
   \tilde{\phi}^{(*)0}(\mathrm du,\mathrm dv)=\mathrm{diag}(\delta_{0_{[2]}},\ldots,\delta_{0_{[2]}} )=I_d\delta_{0_{[2]}}(\mathrm du, \mathrm dv), 
\end{equation*}

\noindent which is of dimension $d\times d$, and $\tilde{\phi}^{(*)1}=\tilde{\phi}$. For  higher orders, we obtain according to the definition of convolution of generalized functions:
	
	\begin{eqnarray*}
		\tilde{\phi^{j'}_j}* \tilde{\phi^{i'}_i}(f)&=&\int_{\R_+^2}	\int_{\R_+^2} f(u_1+u_2, v_1+v_2)\,\phi^{j'}_j(u_1)\,\delta_{u_1}(\mathrm dv_1)  \phi^{i'}_i(u_2)\,\delta_{u_2}(\mathrm dv_2) \,\mathrm du_1 \mathrm du_2\\
		&= &\int_{\R_+^2} f(u_1+u_2, u_1+u_2)\,\phi^{j'}_j(u_1)\,  \phi^{i'}_i(u_2)\, \mathrm du_1 \mathrm du_2\\
		&=& \int_0^\infty  f(w, w)\left(\int_0^\infty \phi^{j'}_j(u_1)\,  \phi^{i'}_i(w-u_1)\, \mathrm du_1\right) \mathrm dw \\
		&=& \int_0^\infty  f(w, w)\left(\int_0^w \phi^{j'}_j(u_1)\,  \phi^{i'}_i(w-u_1)\, \mathrm du_1\right) \mathrm dw \\
		&=&\widetilde{\phi^{j'}_j* \phi^{i'}_i}(f).
	\end{eqnarray*}
	Similarly, combining matrix product properties and iterating the latter result, we get      $(\tilde{\phi})^{(*)r}= \widetilde{\left(\phi^{(*)r}\right)}$  for all $r\geq 2.$

Then, let us observe that    

\begin{eqnarray*}
\int_0^{t_1}     \tilde{C}^{j}_{k,l}(t-u_{[2]})\ \phi^{j'}_j (u) \mathrm du&=& 
\int_0^{t_1} \int_0^{t_2}     \tilde{C}^{j}_{k,l}(t-(u,v)) \phi^{j'}_j (u)
\delta_{u}(\mathrm dv) \mathrm du\\
&=& \tilde{C}^{j}_{k,l} *\tilde{\phi}^{j'}_j (t_1,t_2).
\end{eqnarray*}
 Similarly, using tensor product $\otimes$ for functions, we get   from Equation \eqref{R}  
 \begin{eqnarray*}
 	R^{j'}_{k,l} (t_1,t_2) &=& \sum_{j=1}^d \int_0^{t_1}\int_0^{t_2} \tr e_k \left(e_j + \tilde{M}^{j,(1)}(t_1-u)\right) \tr \left(  e_j +\tilde{M}^{j,(1)}(t_2-v)\right)e_l \phi^{j'}_j(u) \delta_u(\mathrm dv)\mathrm du \\
 	&=& \sum_{j=1}^d  \tilde{R}^j_{k,l} *\tilde{\phi}^{j'}_j  (t_1,t_2)	
 \end{eqnarray*}
with,  
$$\tilde{R}^{j}_{k,l} (u,v) =   \tr e_k  \left(e_j + \tilde{M}^{j,(1)}(u)\right)  \otimes\tr \left(  e_j +\tilde{M}^{j,(1)}(v)\right)e_l .$$

\noindent To sum up, the vector $\tilde{C}^{\bu}_{k,l}= \left( \tilde{C}^{1}_{k,l}\ldots, \tilde{C}^{d}_{k,l}\right)$ satisfies   

$$\tilde{C}^{\bu}_{k,l}(t_1,t_2)= R^{\bu}_{k,l} (t_1,t_2) +  \tr\tilde{\phi} * \tilde{C}^{\bu}_{k,l} (t_1,t_2)$$
and has for solution: 
$$	\tilde{C}^{\bu}_{k,l}=  \tr \widetilde{\left(\sum_{r=0}^\infty \phi^{(*)r} \right)}*R^{\bu}_{k,l} =  \tr \widetilde{\left(\sum_{r=1}^\infty \phi^{(*)r} \right)}*\tilde{R}^{\bu}_{k,l}, $$
with $\tilde{R}^{\bu}_{k,l}=\tr( \tilde{R}^{1}_{k,l},\ldots, \tilde{R}^{d}_{k,l})$ and 
$R^{\bu}_{k,l}=\tr \tilde{\phi}* \tilde{R}^{\bu}_{k,l}.$

Finally, for $j'=0$, that is letting   $\phi^0=\lambda^0$, we get from Equation \eqref{cov2-1} 
\begin{eqnarray*}
\tilde{C}^0_{k,l} &=& \tr\tilde{\lambda}^0*\tilde{R}^{\bu}_{k,l} + 
\tr\tilde{\lambda}^0* \tr \widetilde{\left(\sum_{r=1}^\infty \phi^{(*)r} \right)}*\tilde{R}^{\bu}_{k,l}\\
&=& \tr\tilde{\lambda}^0*\tr \widetilde{\left(\sum_{r=0}^\infty \phi^{(*)r} \right)}*\tilde{R}^{\bu}_{k,l}.
\end{eqnarray*}
 
\end{proof}

\paragraph{Remarks}
	
\begin{itemize}
\item Each component $\tilde{C}^{j'}_{k,l}$ or $\tilde{M}^{(2)j'}_{k,l}$ satisfies an integral equation only including components of indices $k$ and $l$, such that this equation can be solved independently from the other equations corresponding to other pairs of indices.
\item We can consider that $\widetilde{\left(\sum_{r=0}^\infty \phi^{(*)r} \right)}$ is the inverse (with respect to both  matrix product and  convolution product  of distribution functions) of the signed matrix valued measure supported by the diagonal $\Delta$ of $\R_+^2$, namely $ I_d \delta_{0_{[2]}} - \tilde{\phi}$.
\item The covariance measures $C^{j'}_{k,l}$, $j'=0:d$ and $ 1\leq k, l\leq d$, are all positive measures and, hence, components $N^{j'}_k(t_1)$ and $N^{j'}_l(t_2)$ are always non negatively correlated.
\item Note finally that for  $t_1=t_2$, Solution \eqref{cov-2-2} reduces  to Solution  \eqref{cov1-2} given  for a single time.
\end{itemize}

\subsubsection{Example}\label{example2}

As in the example related to the mean equations and described in Section \ref{example1}, let us take the $\Gamma$ intensities with $\phi(t) =W f(t; \kappa,\theta) $ whose primitive function is $\Phi(t)=W\gamma\left(\kappa,t/\theta \right) /\Gamma(\kappa)$.  
For $\kappa=1$, we readily get 
	$$\tr \widetilde{\left(\sum_{r=1}^\infty \phi^{(*)r} \right)}(\mathrm du,\mathrm dv)=\frac{ e^{-u/\theta}}{u}\left(e^{ W u/\theta} -I_d\right)\delta_u (\mathrm dv)\mathrm du.$$
On the other hand, with 
 $\tilde{R}^{j}_{k,l}(u,v)=     \tr e_k  \left(I_d +  \tilde{M}^{(1)}(u)\right)e_j  \otimes\tr e_j \tr\left(  I_d +\tilde{M}^{(1)}(v)\right)e_l $
 and  $	\tilde{M}^{(1)}(u)=   
 \sum_{r=1}^\infty \frac{W^r}{ r! } \gamma( r,  u/\theta)$,  we get after transposition operations
 $$	\tilde{C}^{j'}_{k,l}(t_1,t_2)= \sum_{j=1}^d  \int_0^{t_1} \frac{ e^{-u/\theta}}{u}\tr e_{j'}\left(e^{W u/\theta} -I_d\right)e_{j} \tilde{R}^{j}_{k,l}(t_1-u,t_2-u) \mathrm du$$
which has no known analytic expression but which can be numerically approximated.

\subsection{Lebesgue Decomposition of the Covariance Measure}

The covariance functions $C$ of a  point process $N$ on a measurable space $(D, \mathcal{D})$  defined by $C(A\times B)= \mathbf E(N(A)N(B))- \mathbf E(N(A))\mathbf E(N(B)), \ A,B \in \mathcal{D} $, is advantageously considered from a statistical point of view as a  signed moment measure of order $2$, that is $C(\mathrm du,\mathrm dv)$ on $D^2$ or simply  $C(t_1,t_2)=C([0,t_1]\times [0,t_2])$. It is well known that  it has at least a singular part $^SC$ whose support is the diagonal $\Delta=\{ u_{[2]}=(u,u), u\in D\}$ of $D^2$. The singular and continuous parts   may take different forms depending on the probability distribution of $N$.  We develop this issue for our multivariate Hawkes process and prove that the singular part is also supported by $\Delta$ but has an additional part to the usual first order moment. The absolutely  continuous part is also determined.  For that purpose we recourse once again to a measure and generalized function context.

\begin{proposition} Let the matrix valued function be continuous and non negative. Then the covariance measure $\tilde{C}^{j'}_{k,l}$ is locally finite, it has 
for all  $j'= 0:d$ and $1\leq  k,l \leq d$, an absolute continuous part, namely  
 
 \begin{align*}     
 ^{AC}\tilde{C}^{j'}_{k,l}(\mathrm du,\mathrm dv) = 
 %\left(
 \Big( &
 \tilde{m}^{j',(1)}_k(u) \,\tilde{m}^{k,(1)}_l(v-u) \\
 & +  \newline \sum_{j=1}^d \int_0^{u\wedge v}\tilde{m}^{j',(1)}_j(w)\tilde{m}^{j,(1)}_k(u-w) \tilde{m}^{j,(1)}_l(v-w)  \mathrm dw 
 \Big)
 %\right) 
    \mathrm du \,\mathrm dv.
\end{align*}

\noindent The singular part is supported by the diagonal $\Delta$ of $\R_+^2$ and is written: 
$$^S\tilde{C}^{j'}_{k,l}(\mathrm du,\mathrm dv) =  \left(
 \delta_{k=l}\tilde{m}^{j',(1)}_k+ \tilde{m}^{j',(1)}_l*\tilde{m}^{l,(1)}_k\right)(u)\,\delta_u(\mathrm dv) \,\mathrm du.$$
\end{proposition}

\begin{proof} 
We  first recall some results on  differential calculus as well as some notations related to   generalized functions with support on $\R_+^2$.
	
The Heaviside function $H(u,v)=\mathbf{1}_{\R_+^2}(u,v)$ considered as a distribution (precisely the restricted Lebesgue measure) satisfies the equation  
	 $\partial_u\partial_v H (\mathrm du,\mathrm dv)= \delta_0(\mathrm du)\delta_0(\mathrm dv)=\delta_{0_{[2]}}(\mathrm du,\mathrm dv)$.	
\bigskip

Let us assume firstly that $j' \ne 0$. The convergent series   $h(u)=  \left(\sum_{r=1}^\infty \phi^{(*)r}\right)(u)= \tilde{m}^{(1)}(u)$ if finite,  is necessarily  continuous since it is written 
as the convolution of a generalized function with a continuous one, that is to say $h=\left(\sum_{r=0}^\infty \phi^{(*)r}\right) *\phi$, inheriting  therefore the continuity property of $\phi$.

Then, emphasizing that the constant $1$ is actually a function $1(u)=1_{[0,\infty[}(u)$,  we can develop  $ \tilde{R}^{j}_{k,l}$ as follows
$$ \tilde{R}^{j}_{k,l}(u,v)=  \delta_{k=l=j} 1(u)1(v) + \delta_{j=l} \tilde{M}^{l,(1)}_k(u)1(v)
+\delta_{j=k}1(u)\tilde{M}^{k,(1)}_l(v) + \tilde{M}^{j,(1)}_k(u)\tilde{M}^{j,(1)}_l(v).
  $$
Moreover, according to Proposition \ref{sol-cov-2} with $\tilde{h}(\mathrm du,\mathrm dv)= h(u)\delta_u(\mathrm dv)\mathrm du$, we have

$$\tilde{C}^{j'}_{k,l}(t_1,t_2) = \left(\sum_{j=1}^d \tilde{h}^{j'}_j * \tilde{R}^{j}_{k,l}\right)(t_1,t_2). $$

Now, we separately compute each term of $\tilde{h}^{j'}_j * \tilde{R}^{j}_{k,l}$ using the expression of $\tilde{R}^{j}_{k,l}$ provided above. We recall that $\tilde{m}^{(1)} $  has support $\R_+$, and we first get 
\begin{eqnarray*}
\left(\tilde{h}^{j',(1)}_j * [1(.)1(.)]\right)(t_1,t_2) &=& \int_0^{t_1} \tilde{m}^{j',(1)}_j(u)\mathrm du\\
\left(\tilde{h}^{j'}_j * [\tilde{M}^{l,(1)}_k(.)1(.)]\right)(t_1,t_2)
 &=& \int_0^{t_1} h^{j'}_j(u) \tilde{M}^{l,(1)}_k(t_1-u) \mathrm du  \\
&=& \int_0^{t_1}  \int_0^{t_1}\tilde{m}^{j',(1)}_j(u) \tilde{m}^{l,(1)}_k(w-u) \mathrm du \mathrm dw\\
&=&  \int_0^{t_1} \left(\int_0^{w} \tilde{m}^{j',(1)}_j(u) \tilde{m}^{l,(1)}_k(w-u) \mathrm du \right) \mathrm dw\\
&=& \int_0^{t_1} \tilde{m}^{j',(1)}_j*\tilde{m}^{l,(1)}_k (w)  \mathrm dw.
\end{eqnarray*} 
Let us observe that both components depend only on the first variable $t_1$ meaning that both measures have the diagonal $\Delta$ as support.

For the third term, we obtain
\begin{eqnarray*}
\left(\tilde{h}^{j'}_j * [1(.)\tilde{M}^{k,(1)}_l(.)]\right)(t_1,t_2) &=& \int_0^{t_1} h^{j'}_j(u) \tilde{M}^{k,(1)}_l(t_2-u) \mathrm du   \\
&=& \int_0^{t_1} \int_0^{t_2} \tilde{m}^{j',(1)}_j(u) \tilde{m}^{k,(1)}_l(v-u) \mathrm du\mathrm dv. 
\end{eqnarray*}

\noindent This means that this measure component is absolutely continuous with respect to the Lebesgue measure on $\R_+^2$ with density  $\tilde{m}^{j',(1)}_j(u) \,\tilde{m}^{k,(1)}_l(v-u)$. 

For the last term,  we get 
\begin{eqnarray*}
\left(\tilde{h}^{j'}_j * [\tilde{M}^{j,(1)}_k(\cdot)\tilde{M}^{j,(1)}_l(\cdot)]\right)(t_1,t_2) &=& 
\int_0^{t_1} h^{j'}_j(u) \,\tilde{M}^{j,(1)}_k(t_1-u) \,\tilde{M}^{j,(1)}_l(t_2-u) \,\mathrm du   \\[7pt]
&=& \int_0^{t_1}  \int_0^{t_1} \int_0^{t_2} \tilde{m}^{j',(1)}_j(u)\,\tilde{m}^{j,(1)}_k(w-u) \\&& \qquad \times \tilde{m}^{j,(1)}_l(v-u)  \mathrm du\,\mathrm dw\,\mathrm dv\\[7pt]
&=&\int_0^{t_1}  \int_0^{t_2}\Big( \int_0^{v\wedge w}\tilde{m}^{j',(1)}_j(u)\,\tilde{m}^{j,(1)}_k(w-u) \\&& \qquad \times  \tilde{m}^{j,(1)}_l(v-u)  \mathrm du\Big) \mathrm dw\,\mathrm dv.
\end{eqnarray*} 
Therefore this last component is absolutely continuous with density 
\[ \int_0^{v\wedge w}\tilde{m}^{j',(1)}_j(u)\tilde{m}^{j,(1)}_k(w-u) \tilde{m}^{j,(1)}_l(v-u)  \mathrm du.\]

To sum up, $\tilde{C}^{j'}(\mathrm du,\mathrm dv)$ has a singular component supported by the diagonal $\Delta$, namely

$$^S\!\tilde{C}^{j'}_{k,l}(\mathrm du,\mathrm dv) =  \left(
\delta_{k=l}\tilde{m}^{j'}_k+ \tilde{m}^{j'}_l*\tilde{m}^{l,(1)}_k\right)(u)\,\delta_u(\mathrm dv) \,\mathrm du ,$$
	
\noindent and the absolute continuous component is written

\begin{align*}
  ^{AC}\!\tilde{C}^{j'}_{k,l}(\mathrm du,\mathrm dv) =  \Big(&\tilde{m}^{j',(1)}_k(u)\, \tilde{m}^{k,(1)}_l(v-u) \\
&+  \sum_{j=1}^d \int_0^{u\wedge v}\tilde{m}^{j',(1)}_j(w)\,\tilde{m}^{j,(1)}_k(u-w)\, \tilde{m}^{j,(1)}_l(v-w)\,  \mathrm dw \Big) \mathrm du\, \mathrm dv.
\end{align*}

If $j'=0$, we get the same formulas by defining $h=(\sum_{r=0}^{+\infty}\phi^{(*)r})*\lambda^0$ which is still continuous as soon as $\lambda^0$ is continuous.

\end{proof}

\paragraph{Remarks}
\begin{itemize}
\item It is worth  noticing that the singular and continuous parts satisfy the same integral equation when considering  the associated components of $\tilde{R}^j_{k,l} $. 
\item  Fundamentally, the first and second moments roughly satisfy the same type of integral equation, save that the space dimension increases. It is sensibly expected that the moments of higher order involving several times, would satisfy similar integral equations with explicit solutions having more complex singular parts supported by one dimensional diagonal and semi diagonal hyperplanes, etc.
\end{itemize}

\section{Schemes for Numerical Approximation} \label{section-approximation}

In this section, we describe some algorithms for computing the numerical solutions of the equations of the first two moments and the Laplace transforms developed in the previous sections. 
Since these functions are time-dependent, we use the following discretization scheme: For $T>0$, we consider the  regular partition of the time interval $[0,T]$ consisting of successive sub-intervals with length $\tau=1/M >0$, $M\in\mathbb N_*$, 
that is $0 < \tau < 2\tau< ... < M\tau=T$, and we call it the grid denoted by $\mathcal{G}= \{0,\ldots,M\}.$

\subsection{Time Approximation of Laplace Transforms} 	

Since the solution of  the Laplace transform  of any Hawkes process driven by $(\lambda^0,\phi)$, is simply expressible as  the solution of the basic system of Volterra-like equations of Theorem \ref{thm:1} (see the corresponding remark), we propose the algorithm described below to approach this basic system on any time interval $[0,T]$. We further assume that the multi-time Laplace transforms $\tilde{L}^{j',(k)}(a,t)$ are continuous with respect to the time variable $t$, such that the limits of their discrete time approximation remain meaningful. We also assume that the matrix parameter $a$ is fixed and, hence, we often omit 
it in the sequel.
 
\subsubsection{Single-time Laplace Transform Approximation}

For any fixed vector parameter $ \tr a=(a^1_1,\ldots,a^1_d)$, we propose to approximate the values of $\tilde{L}^{j',(1)} (a^1_\bu,.)$, $j'=1:d$, on the grid $\mathcal{G}$, that is $\bar{L}^{j',(1)}_a(m) \approx \tilde{L}^{j',(1) }(a,m\tau)$. From  Equation \eqref{Eq-tildeL} in Theorem \ref{thm:1} restricted to the case $k=1$, and after a time inversion, we get, for $j'=0:d$, 
\begin{equation*}
	\tilde{L}^{j',(1)}(a,t)=
	\exp\left( \sum_{j=1}^d
	\int_{0}^{t}\left(
	e^{-\, a_j} 
	L^{j,(1)}(a,u) -1
	\right)
	\phi_j^{j'}(t-u)\,\mathrm du
	\right).
\end{equation*}

\noindent Hence, we trivially set, for all $j'=1:d$,
\begin{align*}
\bar{L}^{j',(1)}_a(0)&=1  \\
\bar{L}^{j',(1)}_a(1)&=\exp \left( \tau\sum_{j=1}^d \left( e^{-a_j} \bar{L}^{j,(1)}_a(0) -1\right) \phi^{j'}_j(\tau)\right)\\
\bar{L}^{j',(1)}_a(2)&=\exp \left( \tau \sum_{j=1}^d \left( e^{-a_j} \bar{L}^{j,(1)}_a(0) -1\right) \phi^{j'}_j(2\tau)  +
\left( e^{-a_j}\bar{L}^{j,(1)}_a(1) -1 \right) \phi^{j'}_j(\tau) \right),
\end{align*}
and  for $m=3:M$,
\begin{equation}\label{approx1}
\bar{L}^{j',(1)}_a(m)=\exp \left( \tau \sum_{j=1}^d \sum_{r=0}^{m-1} \left( e^{a_j} \bar{L}^{j,(1)}_a(r) -1\right)  \phi^{j'}_j((m-r)\tau) \right).
\end{equation}

\subsubsection{Two-times Laplace Transform Approximation}

In this case, the general formula in Theorem \ref{thm:1} 
with the $d\times 2$-matrix $a =(a^1_\bu,a^2_\bu)$ and the time vector $t=(t_1,t_2)$ becomes, for $j'=0:d$,
\begin{align*}
	\tilde{L}^{j',(2)} (a,t) =\exp \Bigg(
	\sum_{j=1}^d \Bigg( & \int_0^{t_1} 
	 \left[ e^{-(a^1_j +a^2_j)} \tilde{L}^{j,(2)} (a,(t_1-u,t_2-u)) -1 \right] \phi^{j'}_{j} (u)\mathrm du\\
  &+ \int_{t_1}^{t_2} 
	 \left[e^{-a^2_j} \tilde{L}^{j,(1)} (a^2_\bu,  (t_2-u)) -1 \right] \phi^{j'}_{j} (u)\mathrm du \Bigg) \Bigg).
\end{align*}
In the same way as in the previous subsection, for any fixed $a$, we propose to approximate $\tilde{L}^{j',(2)}(a,.)$  on the nodes $(m_1\tau,m_2\tau)$, $0\leq m_1 <m_2<M$, of the bi-dimensional grid $\mathcal{G}^2$ by 
$\bar{L}^{j',(2)}_{a}( m_1,m_2)$.

Considering the previous equation,  we have essentially to approximate $\bar{L}^{j',(2)}_{a}( m_1,m_2)$ recursively only along secondary diagonals $(m_1-r,m_2-r)$, $r=0:m_1$. Clearly, we  have for all $j'=1:d$ and all $m=0:M$,

$$\bar{L}^{j',(2)}_{a}(0,m)= \bar{L}^{j',(1)}_{a^{2}_\bu}(m).$$ 

Hence, after the time inversion $v=t_1-u$, we propose the following recursive formulas:
\begin{align*}
\bar{L}^{j',(2)}_{a}( 1,m+1)= \exp\Bigg( \tau \sum_{j=1}^d\Bigg( & \left[e^{-(a^1_j +a^2_j)}\bar{L}^{j,(2)}_{a}(0,m) -1\right]\phi^{j'}_j(\tau) \\
& + \sum_{r'=1}^{m} \left[e^{-a^2_j}\bar{L}^{j,(1)}_{a^2_\bu}(r') -1\right] \phi^{j'}_j((m+1-r')\tau) \Bigg)\Bigg),
\end{align*}
and for $1<r\leq m$
\begin{equation} \label{approx2}
\begin{split}
\bar{L}^{j',(2)}_{a}( r,m+r)= \exp\Bigg( \tau \sum_{j=1}^d  \Bigg( &\sum_{r'=0}^{r-1} \left[  e^{-(a^1_j +a^2_j)}\bar{L}^{j,(2)}_{a}(r',m+r') -1\right] \phi^{j'}_j((r-r')\tau) \\
& + \sum_{r'=r}^{m+r} \left[e^{-a^2_j}\bar{L}^{j,(1)}_{a^2_\bu}(r') -1\right] \phi^{j'}_j((m+r-r')\tau) \Bigg)\Bigg).
\end{split}
\end{equation}

Thus, using the latter formula, for any $0<m_1 <m_2 \leq M$ and $j'=1:d$, we finally get:
$$\bar{L}^{j',(2)}_{a}( m_1,m_2)= \bar{L}^{j',(2)}_{a}( m_1, (m_2-m_1)+m_1).$$

\subsubsection{Remarks}
 \begin{enumerate}
     \item One may  consider that the  previous approximations $\bar{L}^{j',(k)}_{a}$, $k=1:2$,  simply constitute the first step to initialize  an iterative process using the implicit integral equation of Theorem \ref{thm:1} to yield successive approximations, using the well known fixed-point theorem.   
     \item Using the same fixed-point theorem, one may as well initialize the iteration processes with the well known analytic form of the Laplace transforms of the $d$ multi-type Poisson processes each driven by density $\phi^{j'}_\bu$, $j'=1:d $, and compare the two approaches.
     \item For a general Hawkes process driven by $(\lambda^0,\phi)$, we have only to substitute $\phi^{j'}_j$ for $\lambda^0_j$  in Equations \eqref{approx1} and \eqref{approx2}  
     \item Even tedious, similar procedures of numerical approximations can be applied to higher-order multi-time Laplace transforms 
    $\tilde{L}^{j',(k)} (a, (t_1,\ldots,t_n))$, $2 < k \leq n$.   
 \end{enumerate}

 \subsection{Approximation of the First Two Moment Functions}
 
With regard to the calculation of the  mean and covariance functions of  Hawkes processes,   we have to essentially deal with  the expression  of the fundamental matrix function $\sum_{r=0}^{\infty}\phi^{(*)r}$. 
For specific models, the  elements of the  matrix function $\phi=(\phi^{j'}_j)$ can be easily  convoluted analytically with each other at  any order, and even yield an explicit expression for the series; see Examples \ref{example1} and \ref{example2}.  Otherwise, the different mean and covariance formulas  in Equations \eqref{means}, \eqref{cov-2-2} and \eqref{cov-2-2-2} can be simply approximated on any given grid $\mathcal{G}$ by  limiting the fundamental matrix series to 
 $\sum_{r=0}^K \phi^{(*)r}$ with a large enough integer $K$.

Indeed,  plain numerical methods can  be implemented for convolution product of functions. For example, if  $f_1$ and $f_2$ defined on $\R^+$ have values $f_1(m_1\tau)$ and $f_2(m_2\tau) $ on the grid $\mathcal{G}$ then the values of $f_3=f_1*f_2$ can be approximated by 
 $\bar{f}_3 (m\tau)= \tau\sum_{r=1}^m f_1(r\tau) f_2( (m+1-r)\tau)$
 on the same grid. 
 
 This algorithm extends easily to convolution products of matrix functions.
 Note also that the expression  $\widetilde{\left(\sum_{r=0}^{\infty}\phi^{(*)r}\right)} (t_1,t_2)$ related to the  covariance solution in Proposition \ref{sol-cov-2} mainly consists in restraining integration over the diagonal of $\R_+^2$.
 
 Even if all calculation procedures presented above deal with generic multidimensional Hawkes process, what follows illustrate these procedures for a specific model of a 2D-Hawkes process, for which we give some numerical results about its first two moments.  

\subsubsection{Specification of a 2D-Hawkes process}

Thus, let us consider  the following 2D-Hawkes process driven by the functional parameter  $(\lambda^0,\phi)$, where the baseline  intensity $\lambda^0$ (that may also be called the ignition or immigration intensity) is periodic, and each component  of the basic matrix function $\phi$ is a beta-like distribution, more precisely: 
  \begin{equation} \label{model}\begin{split}
 \lambda^0_i (t) &=  a_i +b_i \sin(c_i t), \\
 \phi^{j'}_j (t) &=  \alpha^{j'}_j t^{\beta^{j'}_j} \ [\gamma^{j'}_j - t]^{\rho^{j'}_j}  \ \text{if} \  t \leq \gamma^{j'}_j \ \text{and 0 otherwise},
 \end{split}\end{equation}
where all parameters are real numbers, $i=1:2$, $a_i >0$,  $a_i >b_i$ and $1\leq j',j\leq 2$.
 Observe that  $(\lambda^0,\phi)$ depends on a finite number of scalar parameters and thus is considered as  a classical parametric model for the sequel. 
 For illustration purpose, parameters $a, b, c, \alpha, \beta,\gamma, \rho$ were randomly drawn, yielding:
\[ 
\begin{array}{c|ccccccccccc}
i \ \text{or } j & a    &   b   &   c   & \alpha^1_\bu&\alpha^2_\bu&\beta^1_\bu&\beta^2_\bu&\gamma^1_\bu&\gamma^2_\bu&\rho^1_\bu & \rho^2_\bu \\
\hline
1 & 1.057 & 0.031 & 0.845 &  0.073      &  0.046     &   0.060   & 1.254    &     1.576  &   1.831    &    0.598  &   0.897     \\
2 & 1.061 & 0.093 & 0.817 & 0.050       &  0.096     &   1.897   &  1.923   &    0.369   &   0.182    &    0.789  &     0.713
\end{array}
\]
The whole time interval is defined with  $T=10$, the mesh length is $\tau=0.005$ and the infinite sum is stopped at $r=500$. 
 
 Figure \ref{fig1} shows the shapes of the functional parameters constituting $\lambda^0$ and $\phi$.  The periodic shape of $\lambda^0$ was chosen to picture  a potential
 seasonality  effect in the ignition of epidemics. The different beta-like shapes of $\phi$ components presented here may represent differences in the characteristics of the epidemic dynamics in and between different regions (e.g., heterogeneity in latency, virulence, infection duration, spatial connectivity, properties of pathogen strains, social and care conditions, etc.). Figure  \ref{fig2} shows the behavior of the fundamental series, and illustrates in this case the rapid extinction  of the epidemic if no ignition / immigration occurs.
 
 \begin{figure}[H] 
\begin{center}
%\footnotesize
%\rotatebox{90}{\hspace*{1.0cm}Ignition intensity} 
\hspace*{-0.15cm}
\includegraphics[width=0.5\textwidth]{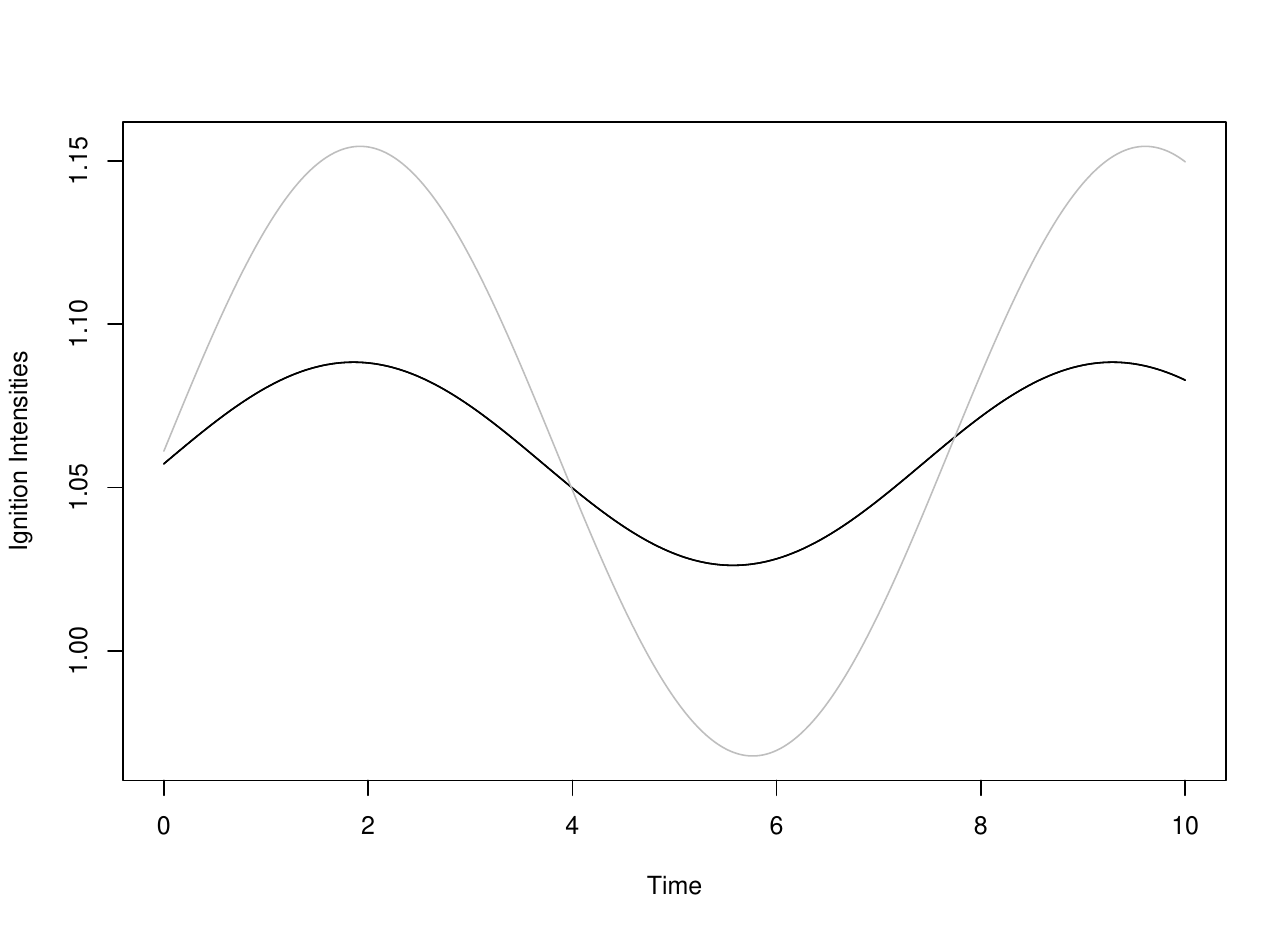}
\hspace*{-0.25cm}
%\rotatebox{90}{\hspace*{1.2cm}Basic functions} 
\includegraphics[width=0.5\textwidth]{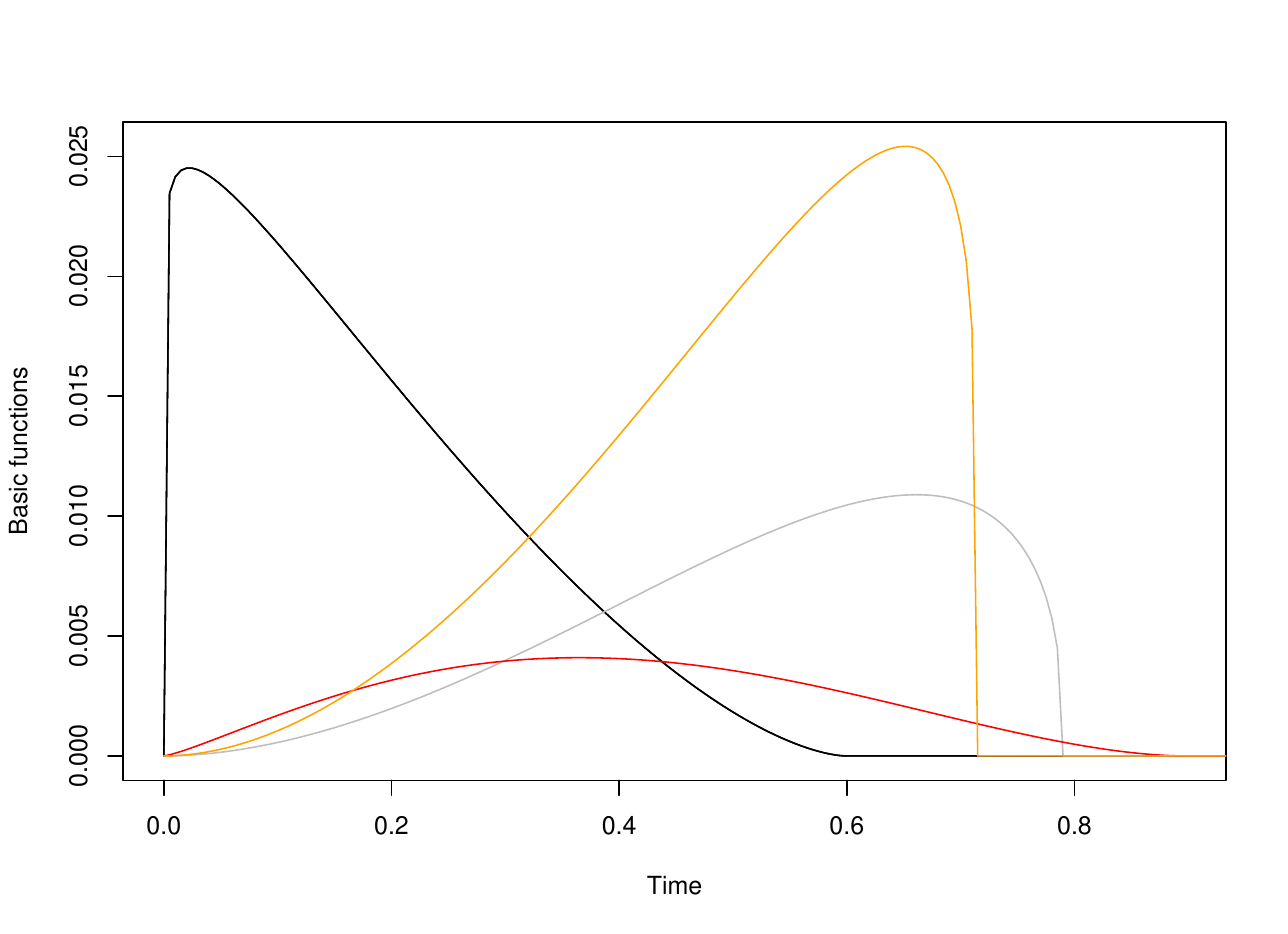}
%\vspace*{-0.3cm}\\ \hspace*{4cm} Time \hfill Time \hspace*{4cm}
\caption{Functional parameters constituting $\lambda^0$ and $\phi$ specified in Equation \eqref{model}, and drawn over the time interval $[0,10]$.
Left panel: Ignition intensities $\lambda^0_j (t) \equiv \phi^0_j (t)$, $j=1,2$. Right panel: Basic 2$\times$2-matrix functions $\phi^{j'}_j(t)$, $1\leq j',j \leq 2$. Black and grey colors correspond to column $j'=1$, red and orange to column $j'=2$.  }   \label{fig1}
 \end{center}
\end{figure}

\begin{figure}[H] 
\begin{center}
%\footnotesize
%\rotatebox{90}{\hspace*{1.5cm}Fundamental series} 
\includegraphics[width=0.55\textwidth]{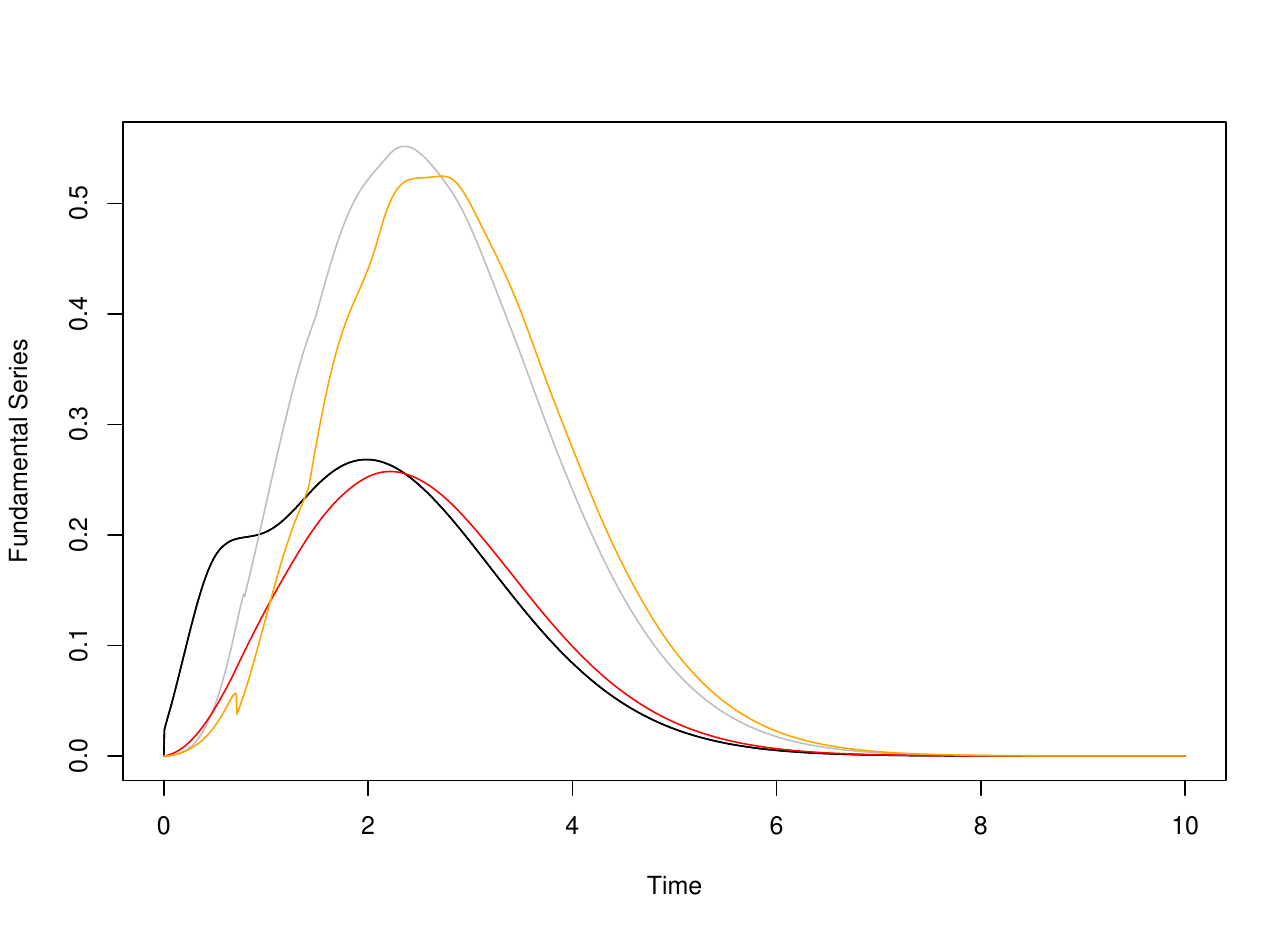 }
%\vspace*{-0.3cm}\\Time
\caption{ Fundamental series $\sum_{r=1}^\infty \phi^{(*)r}(t)$, for $ t \in [0,10] $, approximated up to order $r=500$.
 Black and grey colors correspond to column $j'=1$, red and orange to column $j'=2$.  }   \label{fig2}
 \end{center}
\end{figure}

\subsubsection{Calculation of the Mean functions }

The left panel of Figure \ref{fig3} gives the basic intensity ratios for $1\leq j,j'\leq 2$, that is to say: $\tilde{m}^{j',(1)}_j (t)/ \left(\tilde{m}^{j',(1)}_1(t)+ 
\tilde{m}^{j',(1)}_2(t)\right)$. We observe the fast convergence of the component ratios towards two limits. This phenomenon is also observed for all models we have tried out and is probably related to Perron theorem on positive matrix.  The right panel shows the basic mean functions $\tilde{M}^{j',(1)}_j (t)$. We observe in this example finite asymptotic limits in the absence of an ignition / immigration process.

Figure \ref{fig4} shows, on the left, a subtle effect of the periodicity of the baseline intensity on the behavior of the intensities $\tilde{m}^{0,(1)}_j$,  $j=1:2$, and, on the right, the increasing mean components $\tilde{M}^{0,(1)}_j$, $j=1:2$, of the original Hawkes process driven by  $(\lambda^0, \phi)$.
  
\begin{figure}[H] 
\begin{center}
%\footnotesize
%\rotatebox{90}{\hspace*{0.5cm}Ratio of basic intensities} 
\hspace*{-0.15cm}
\includegraphics[width=0.5\textwidth]{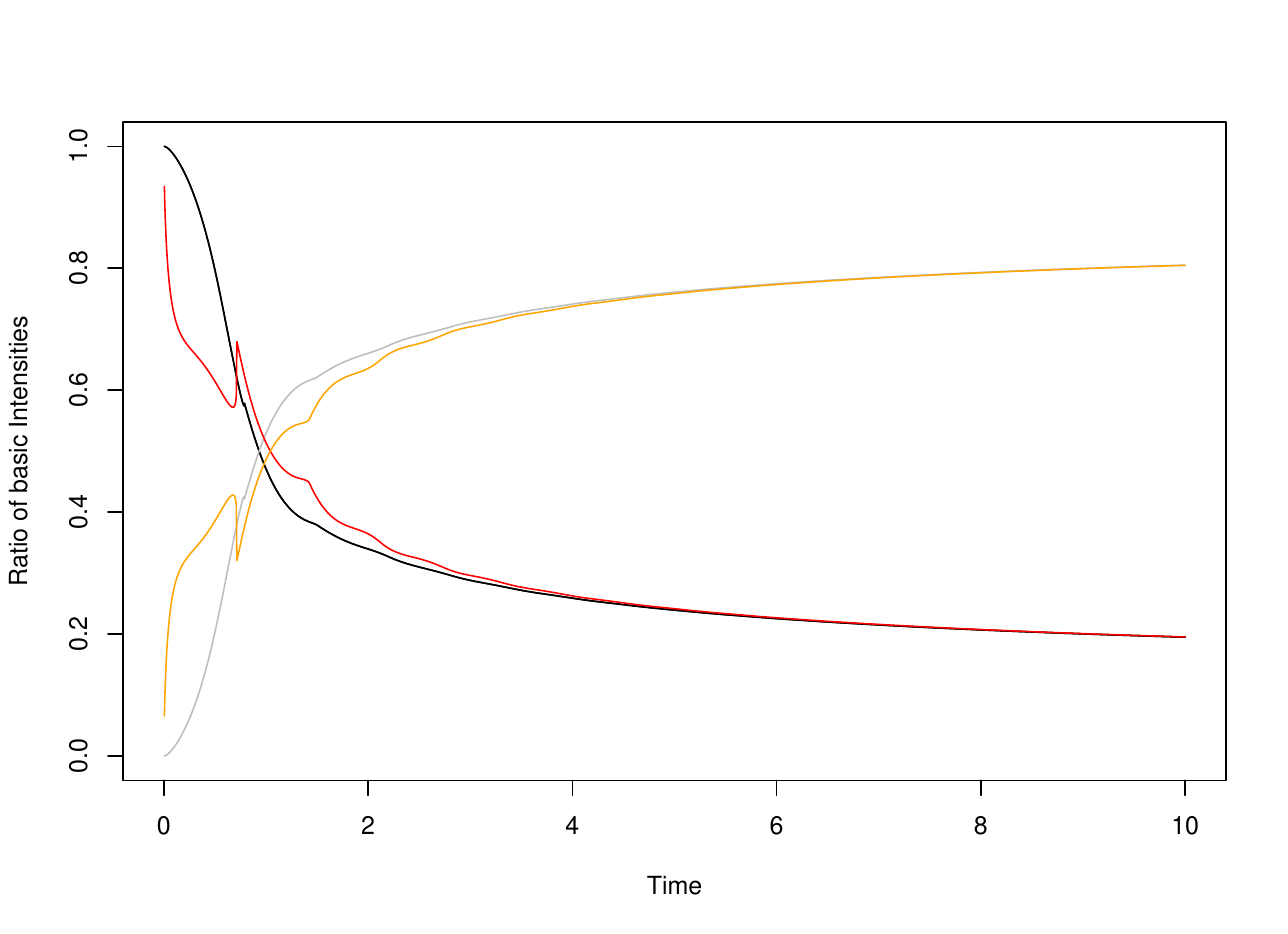}
%\rotatebox{90}{\hspace*{0.3cm}Corresponding cumulative} \rotatebox{90}{\hspace*{1.3cm}intensity} 
\hspace*{-0.25cm}
\includegraphics[width=0.5\textwidth]{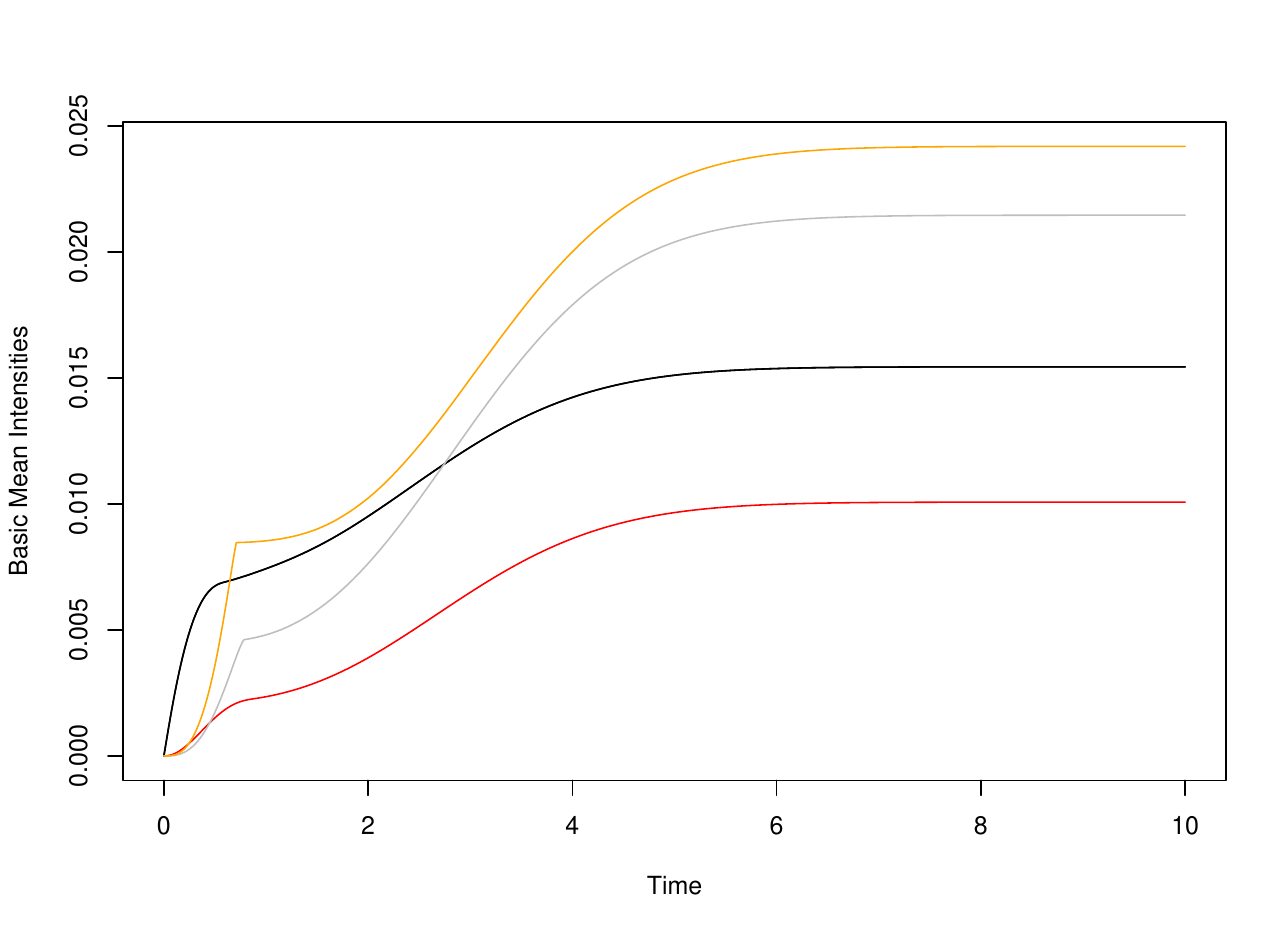}
%\vspace*{-0.3cm}\\ \hspace*{4cm} Time \hfill Time \hspace*{4cm}
\caption{ 
Ratios of basic intensities (left panel) and corresponding cumulative intensities, i.e., the mean functions (right panel), over the time interval $[0,10]$.  Black and grey colors hold for $j'=1$ while red and orange hold for $j'=2$.}  \label{fig3}
 \end{center}
\end{figure}

\begin{figure}[H]
\begin{center}
%\footnotesize
%\rotatebox{90}{\hspace*{1.0cm}Process intensities} 
\hspace*{-0.15cm}
\includegraphics[width=0.5\textwidth]{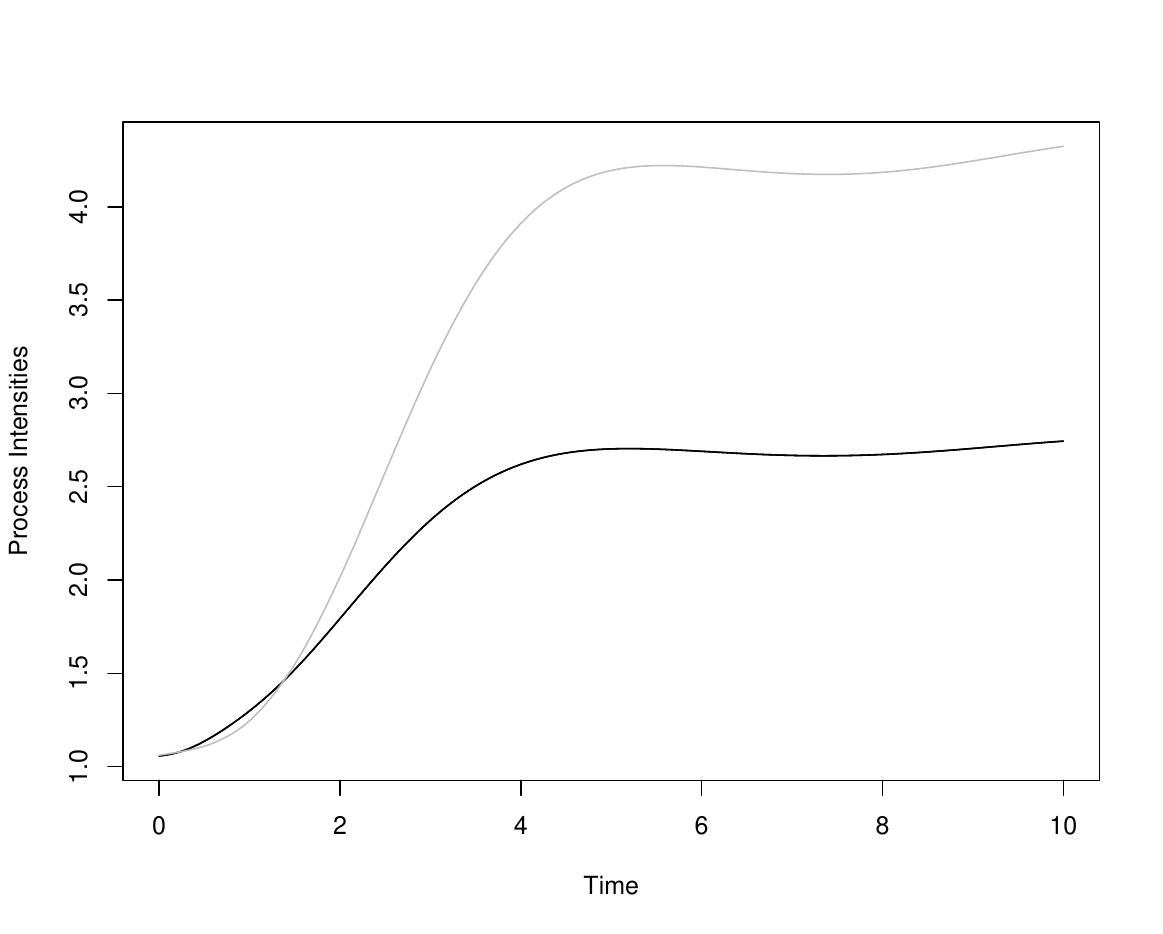}
%\rotatebox{90}{\hspace*{0.8cm}Process cumulative} \rotatebox{90}{\hspace*{1.5cm}intensity} 
\hspace*{-0.25cm}
\includegraphics[width=0.5\textwidth]{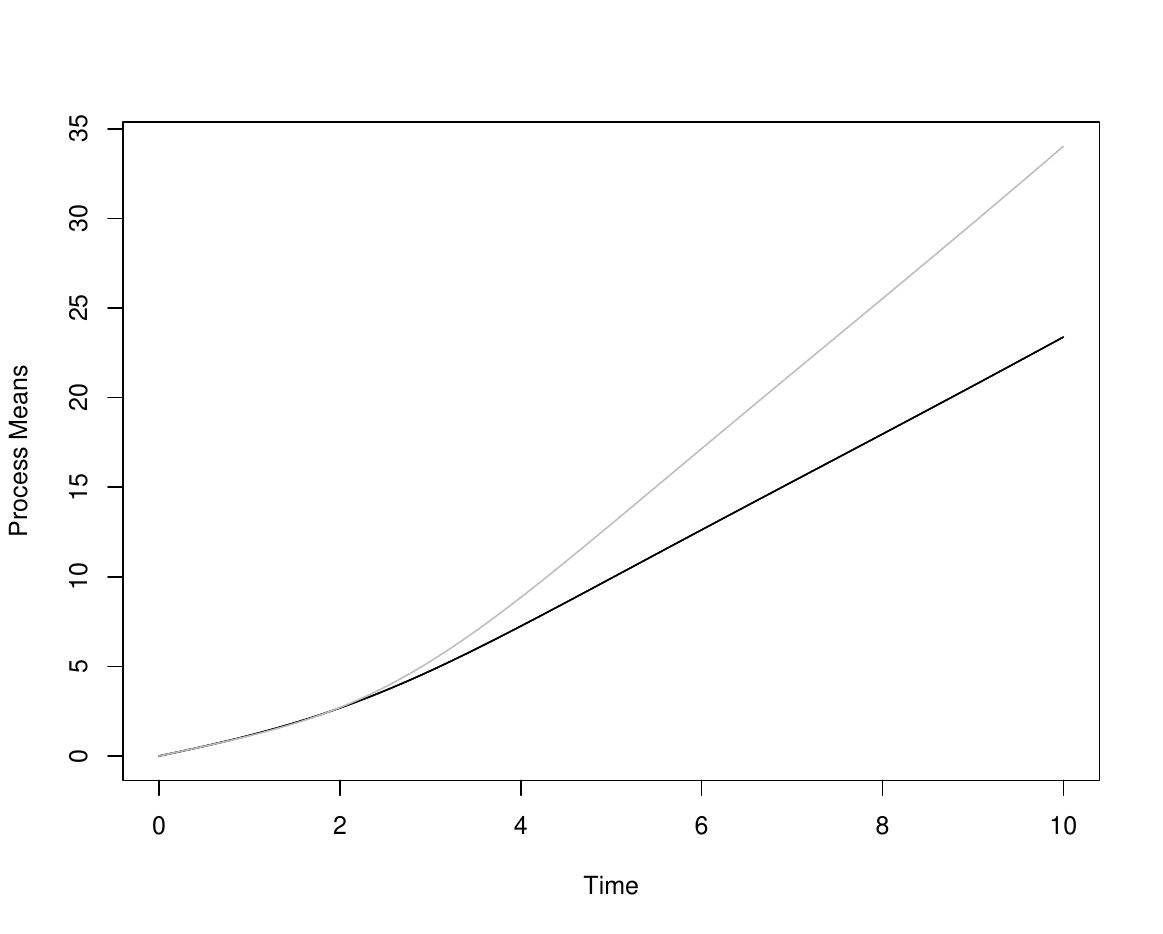}
%\vspace*{-0.3cm}\\ \hspace*{4cm} Time \hfill Time \hspace*{4cm}
\caption{ First order moments of the 2D-Hawkes process driven by the functional parameter $(\lambda^0, \phi)$, over the time interval $[0,10]$.
Left panel: The two components of the intensities $\tilde{m}^{0,(1)}_j$, $j=1:2$. 
Right panel:  The correspondent cumulative intensities, that is the mean functions $\tilde{M}^{0,(1)}_j$, $j=1:2$.  } \label{fig4}
 \end{center}
\end{figure}
  
\subsubsection{Calculation of the Covariance Structure}

From the algorithmic point of  view, the computation of the covariance structure is  very similar to that of the  mean. It is however much longer since the covariance function depends on two variates. Figure \ref{fig5} shows two functional elements of  the fundamental covariance structure, namely  the inner-region variance $\tilde{C}^{j'}_{1,1} (t,t)$ (left) and  the inter-regions covariance $\tilde{C}^{j'}_{1,2} (t,t)$ (right). Once again, due to the beta like form of $\phi$, we observe  asymptotic finite values pointing out the extinction in the absence of immigration.  Figure \ref{fig6} represents images with level contours of two components of the $2\times2$ matrix function of correlations, which depends on 2 variables.

 \begin{figure}[H] 
\begin{center}
%\footnotesize
%\rotatebox{90}{\hspace*{1.5cm}Variance} 
\hspace*{-0.15cm}
\includegraphics[width=0.5\textwidth]{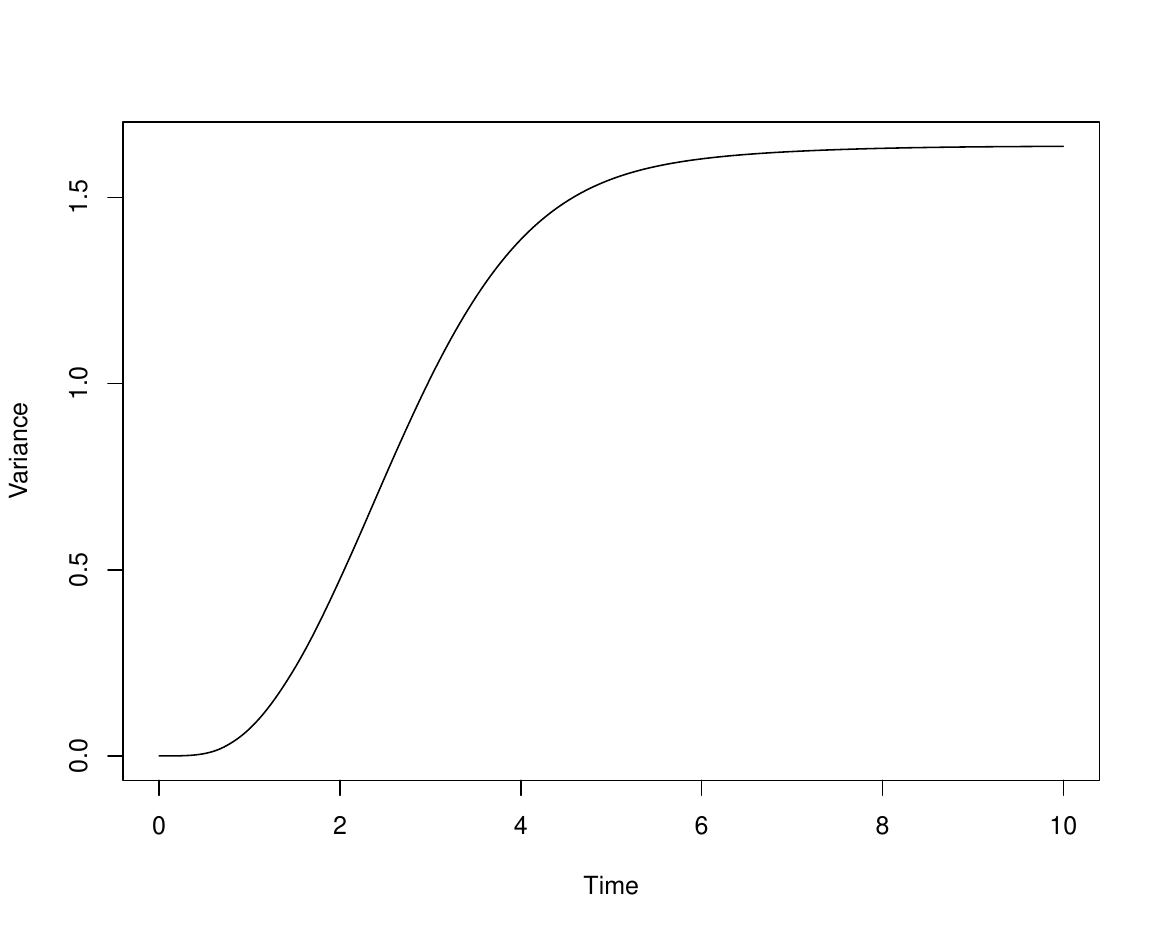}
%\rotatebox{90}{\hspace*{1.4cm}Covariance} 
\hspace*{-0.25cm}
\includegraphics[width=0.5\textwidth]{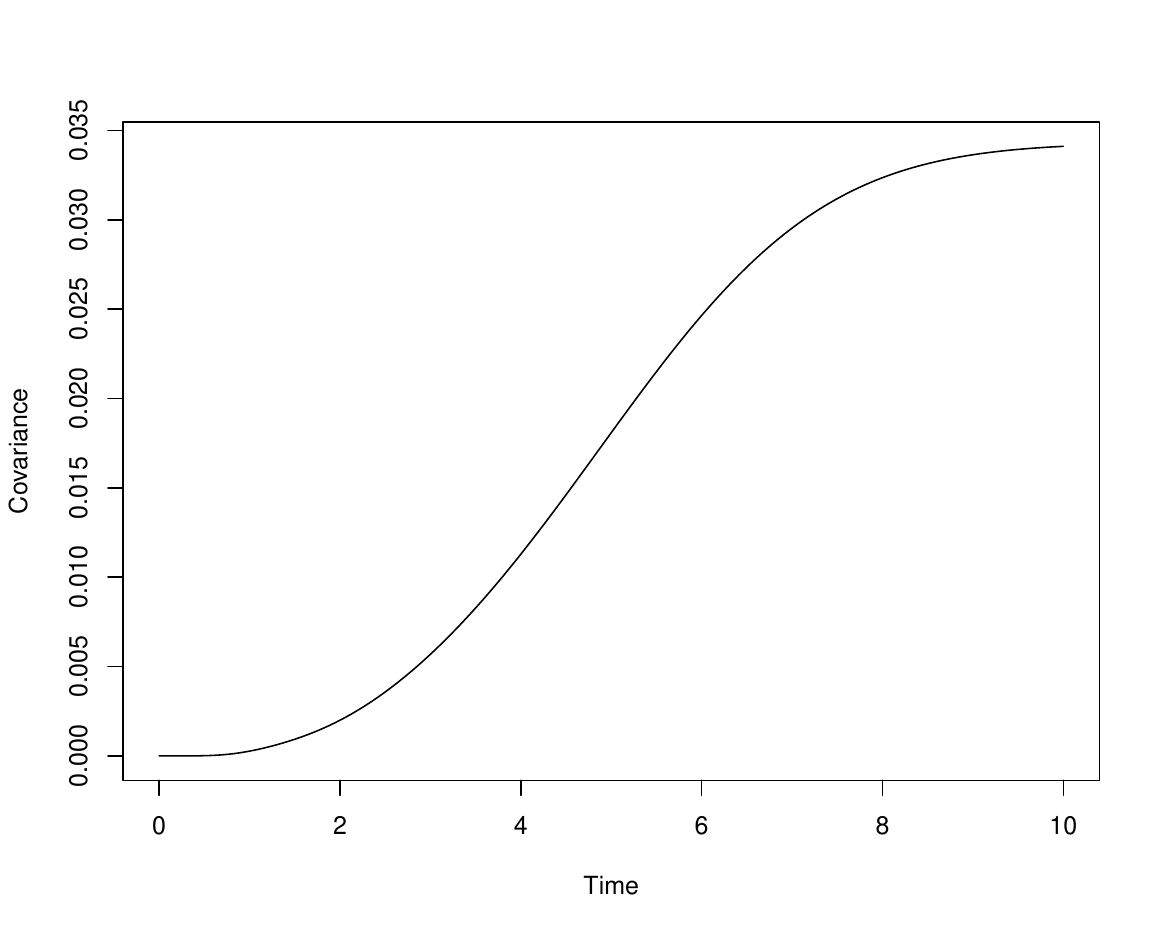}
%\vspace*{-0.3cm}\\ \hspace*{4cm} Time \hfill Time \hspace*{4cm}
\caption{ Example of basic covariance structure over the time interval $[0,10]$, for $j'=1$.
Left panel: Variance function  $t\mapsto \tilde{C}^{j'}_{1,1} (t,t)$. 
Right panel: Covariance function $t\mapsto \tilde{C}^{j'}_{1,2} (t,t)$.
}  \label{fig5}
 \end{center}
\end{figure}

\begin{figure}[H] 
\begin{center}
%\footnotesize
%\rotatebox{90}{\hspace*{2.2cm}$t_2$}
\hspace*{-0.15cm}
\includegraphics[width=0.5\textwidth]{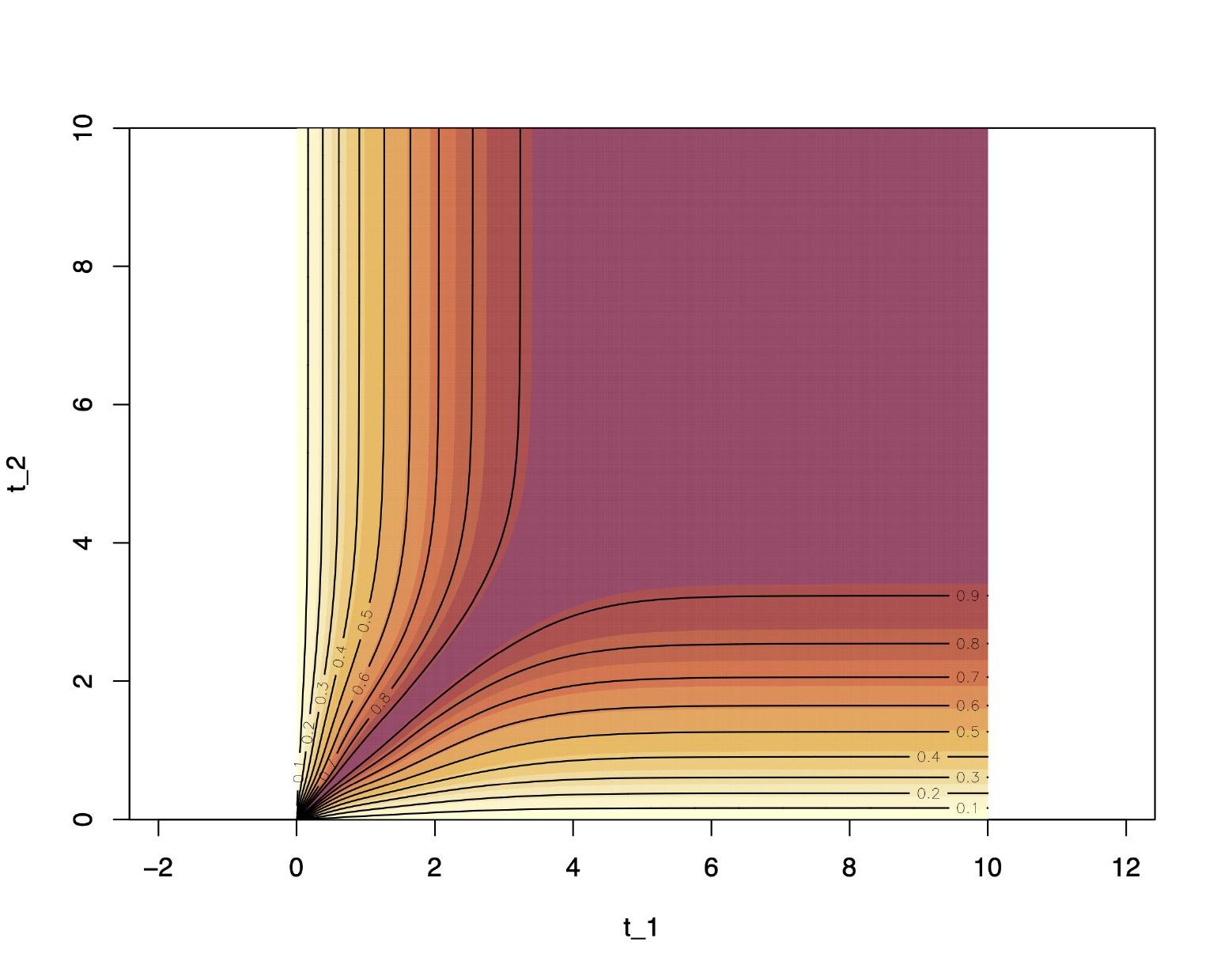}
%\rotatebox{90}{\hspace*{2.2cm}$t_2$}
\hspace*{-0.25cm}
\includegraphics[width=0.5\textwidth]{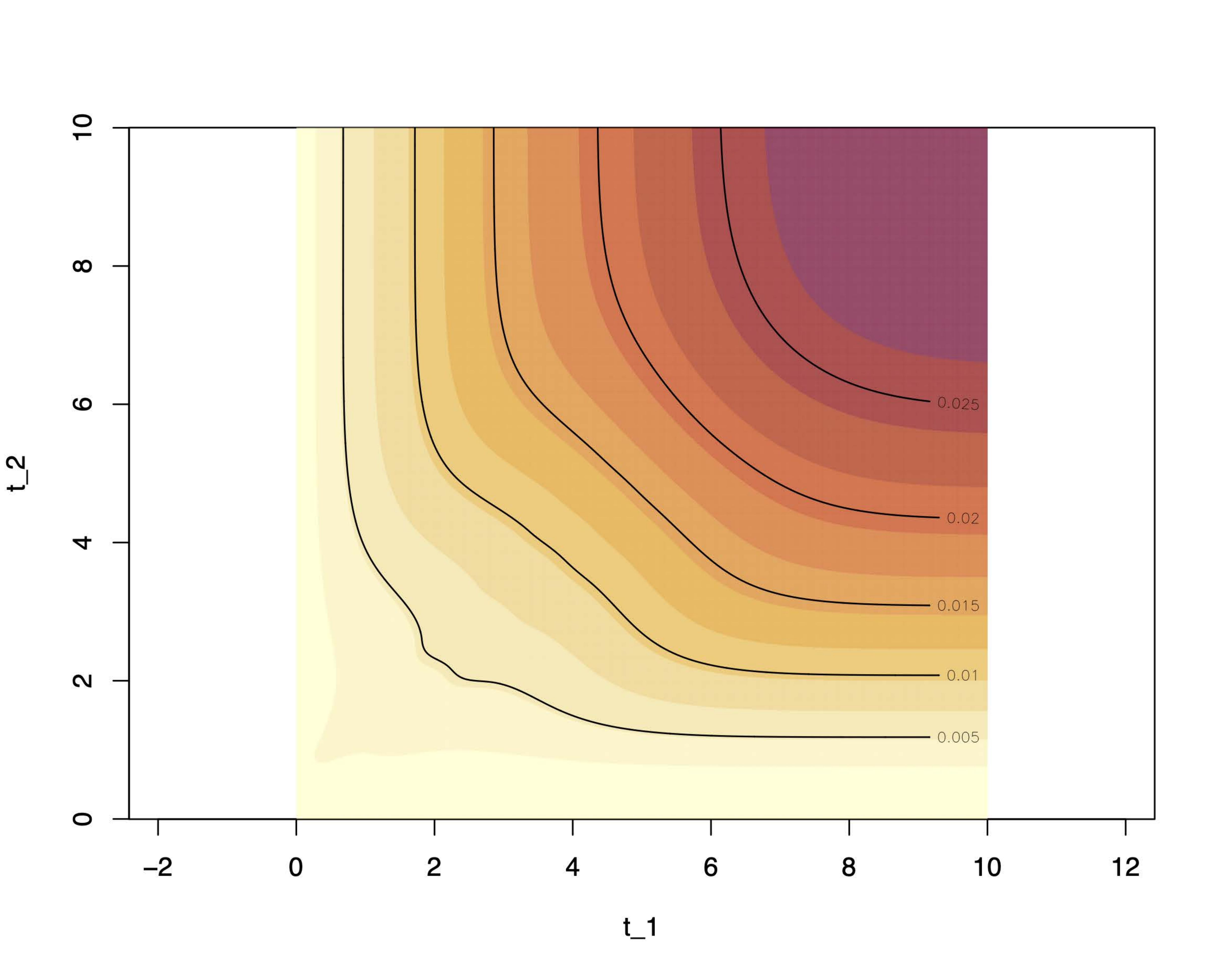}
%\vspace*{-0.3cm}\\ \hspace*{4cm} $t_1$ \hfill $t_1$ \hspace*{4cm}
\caption{ Continued example. Basic correlation structure of a point process driven by $(\phi^1_\bu, \phi)$.
Left panel:  Correlation  within the region $j'=1$ obtained by standardization of $\tilde{C}^{1}_{1,1} (t_1,t_2)$, $0\leq t_1,t_2 \leq 10$.
Right panel: Correlation  between  regions $j'=1$ and $j=2$ obtained by  standardization of $\tilde{C}^{1}_{1,2} (t_1,t_2)$, $0\leq t_1,t_2 \leq 10$.
}  \label{fig6}
 \end{center}
\end{figure}

 Figure \ref{fig7} shows similar graphs related to the original Hawkes process (i.e., when the ignition/immigration process is accounted for). The effect of the periodic trend of 
 $\lambda^0$ can be slightly observed through the inner-region  correlation structure (left; see the subtle fluctuations of the isolines). The inter-regions covariance structure is displayed in the right panel.

 \begin{figure}[H]
\begin{center}
%\footnotesize
%\rotatebox{90}{\hspace*{2.2cm}$t_2$}
\includegraphics[width=0.45\textwidth]{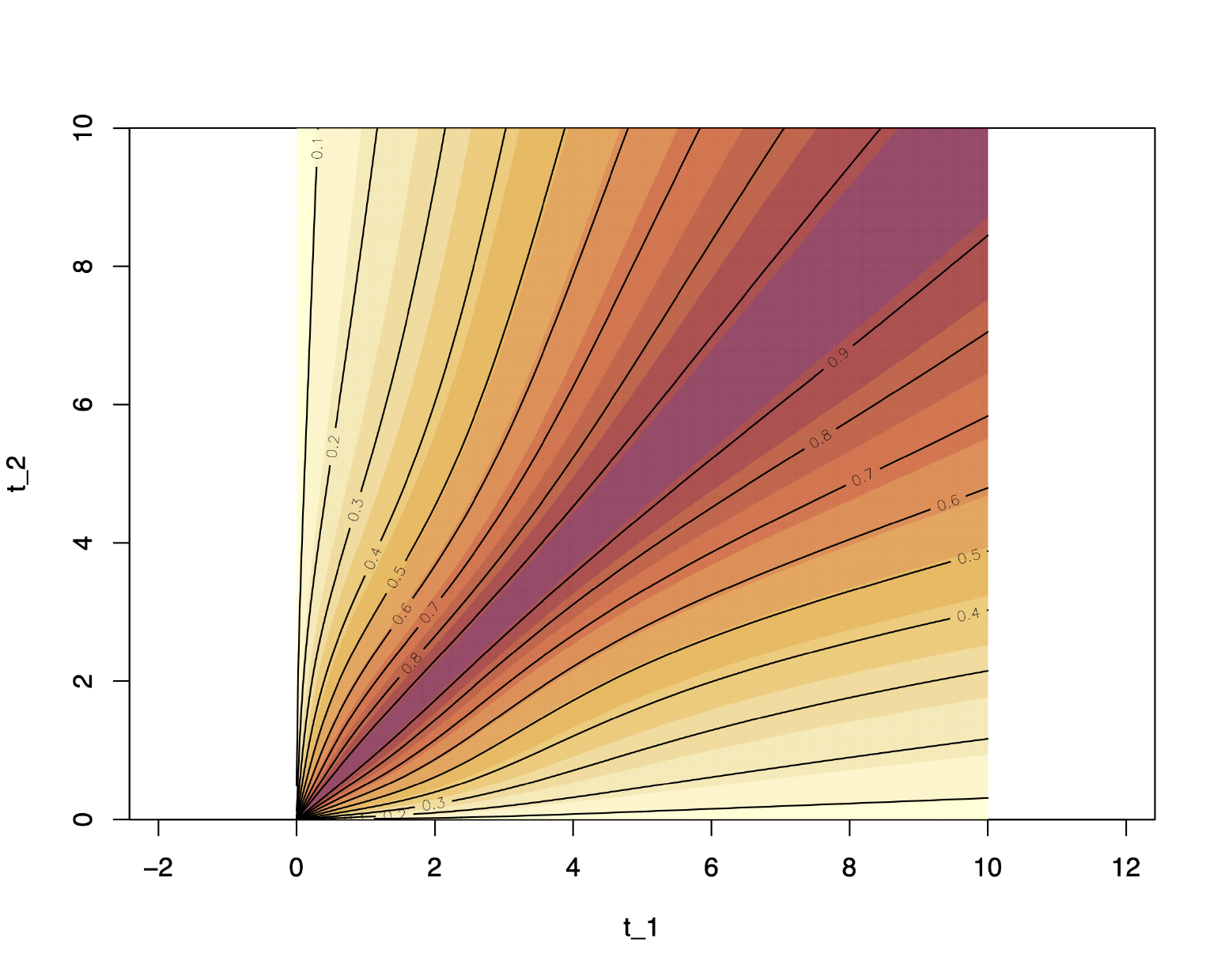}
%\rotatebox{90}{\hspace*{2.2cm}$t_2$}
\includegraphics[width=0.45\textwidth]{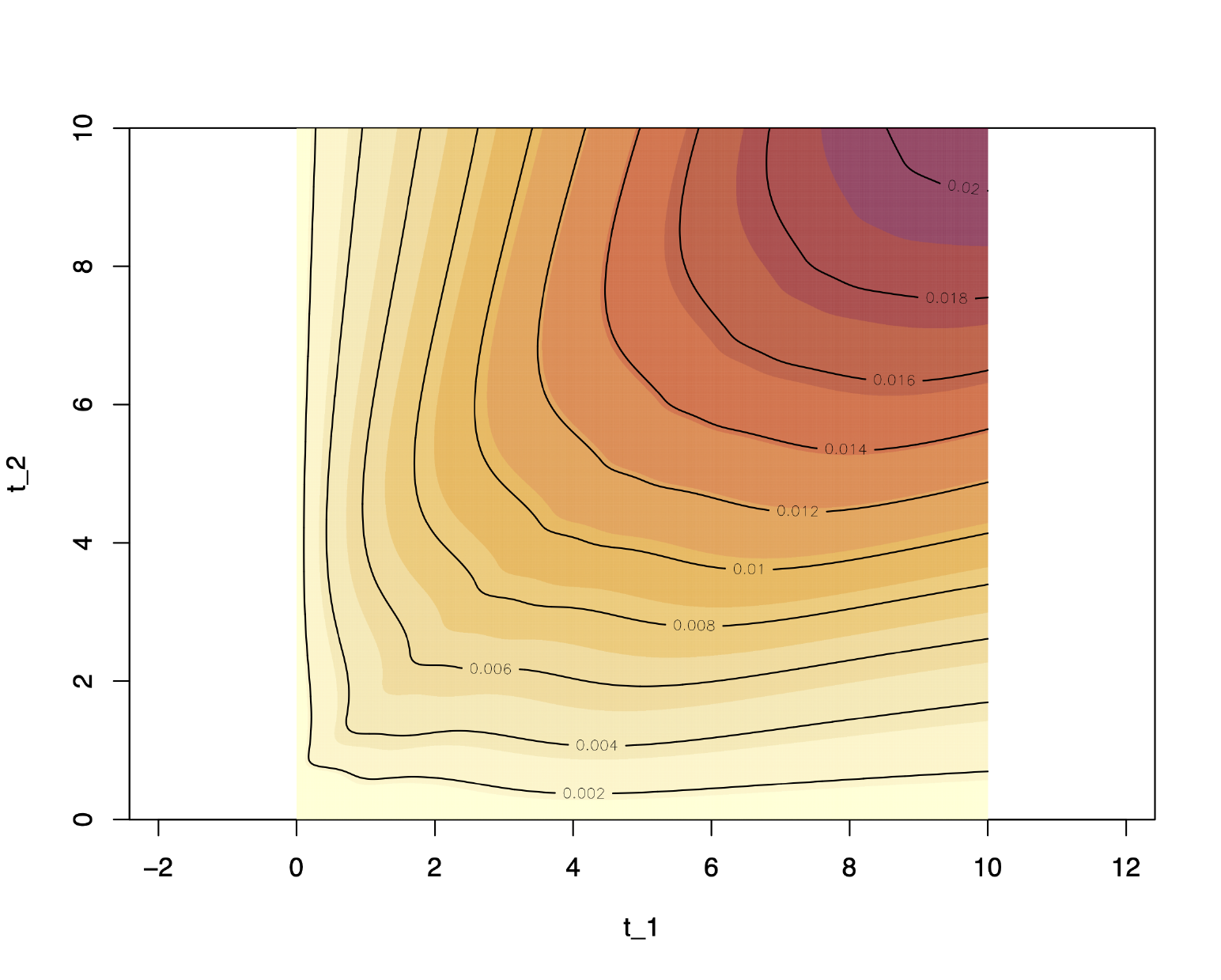}
%\vspace*{-0.3cm}\\ \hspace*{4cm} $t_1$ \hfill $t_1$ \hspace*{4cm}
\caption{ Correlation structures  of the Hawkes process driven by $(\lambda^0, \phi)$, with periodic $\lambda^0$.
Left panel:  Correlation  within the region $j'=2$, obtained from 
$\tilde{C}^{0}_{2,2} (t_1,t_2)$, $0\leq t_1, t_2 \leq 10$.
Right panel: Correlation  between  regions $j'=1$ and $j=2$ obtained from 
$\tilde{C}^{0}_{2,1} (t_1,t_2)$ , $0\leq t_1, t_2 \leq 10$.
}   \label{fig7}
 \end{center}
\end{figure}

\section{Trajectories of Multivariate Hawkes Processes}\label{sec:simul}

In this section, we present simulated trajectories for different multivariate Hawkes processes in order to illustrate the various behaviors that the Hawkes Process can embrace. 

First, we look at the effect of the excitation function by considering exponential, gamma, constant and beta excitation functions with constant baseline intensity.

\paragraph{Case 1: Exponential Excitation Function with Constant Baseline Intensity.}

In this case, the excitation function is exponential, and the baseline intensity is constant. The intensity function $\lambda(t)$ at time $t$ is given by:
\[
\lambda(t) = \mu + \sum_{i=1}^n \alpha_i e^{-\beta (t - t_i)^+}
\]
where $\mu$ is the constant baseline intensity, $\alpha_i$ is the excitation parameter for the $i$-th process, $\beta$ is the decay rate of the exponential excitation, and $(t - t_i)^+$ denotes the positive part of $(t - t_i)$, which is zero for $t < t_i$ and $t - t_i$ for $t \geq t_i$.
Figure \ref{fig:expo} displays two contrasted realizations of trajectories for this setup. The $3$-dimensional counting process $N$ is displayed as well as the evolution of the $3$-dimensional intensity function.

\begin{figure}[h!]
\small
\texttt{Realization 1}\vspace*{-0.3cm}
    \begin{center}
        \includegraphics[width=0.45\textwidth]{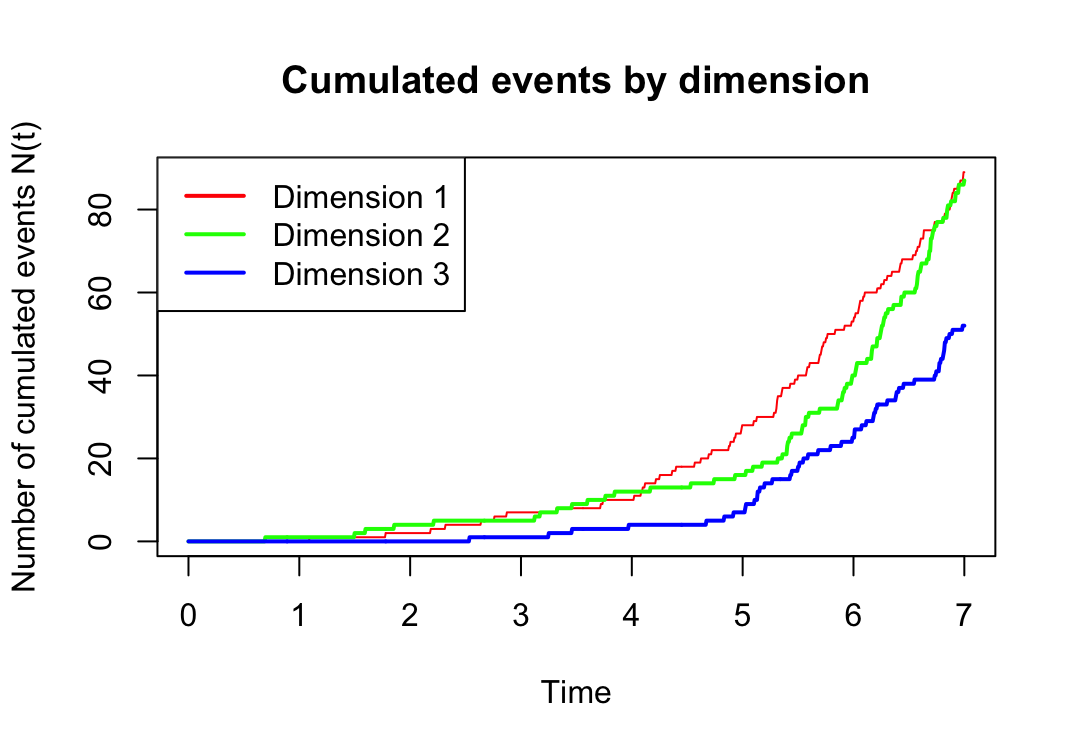}
    \includegraphics[width=0.45\textwidth]{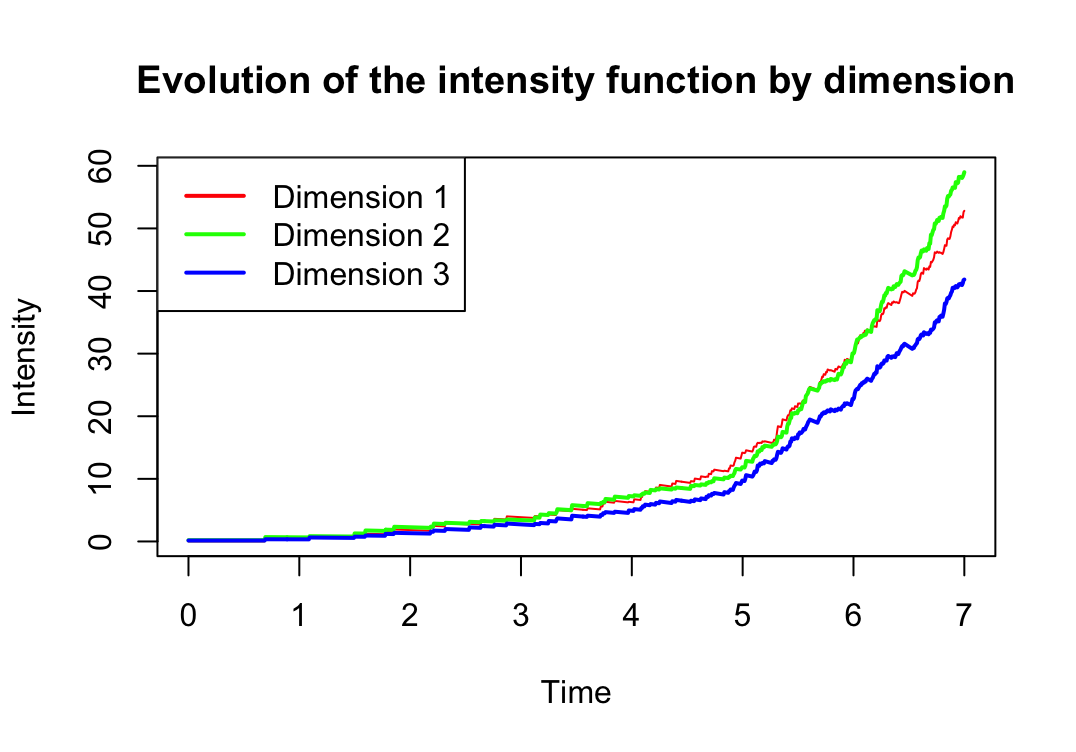}
\end{center}    \vspace*{-0.3cm}
\texttt{Realization 2}\vspace*{-0.3cm}
    \begin{center}
        \includegraphics[width=0.45\textwidth]{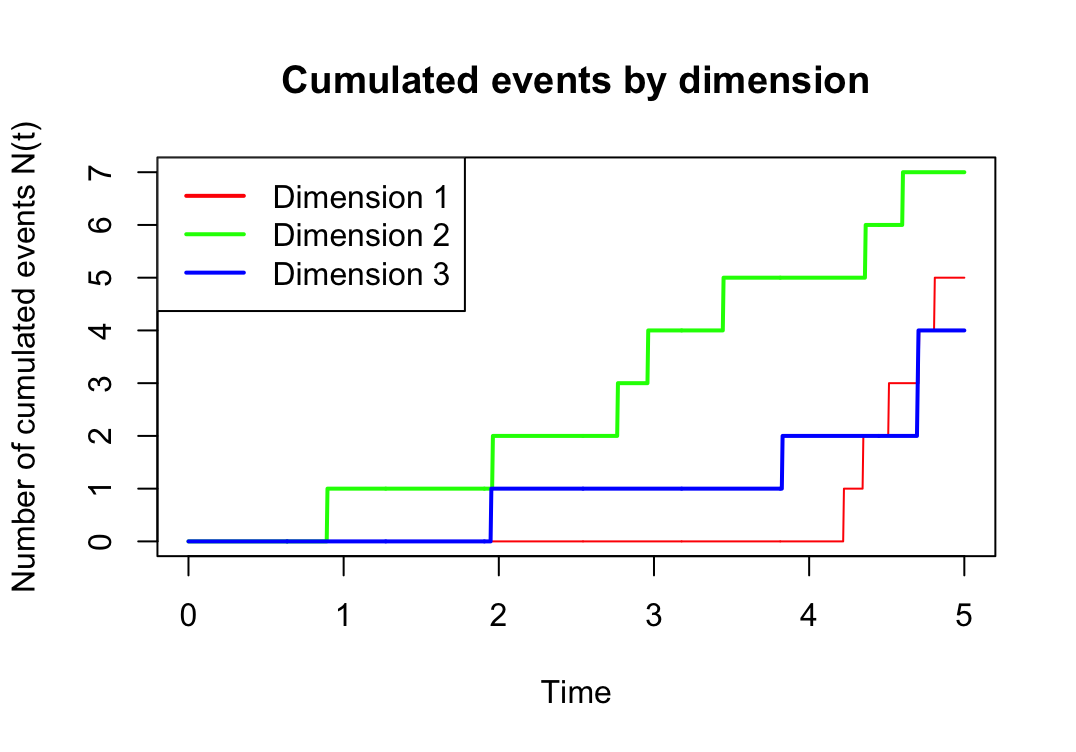}
    \includegraphics[width=0.45\textwidth]{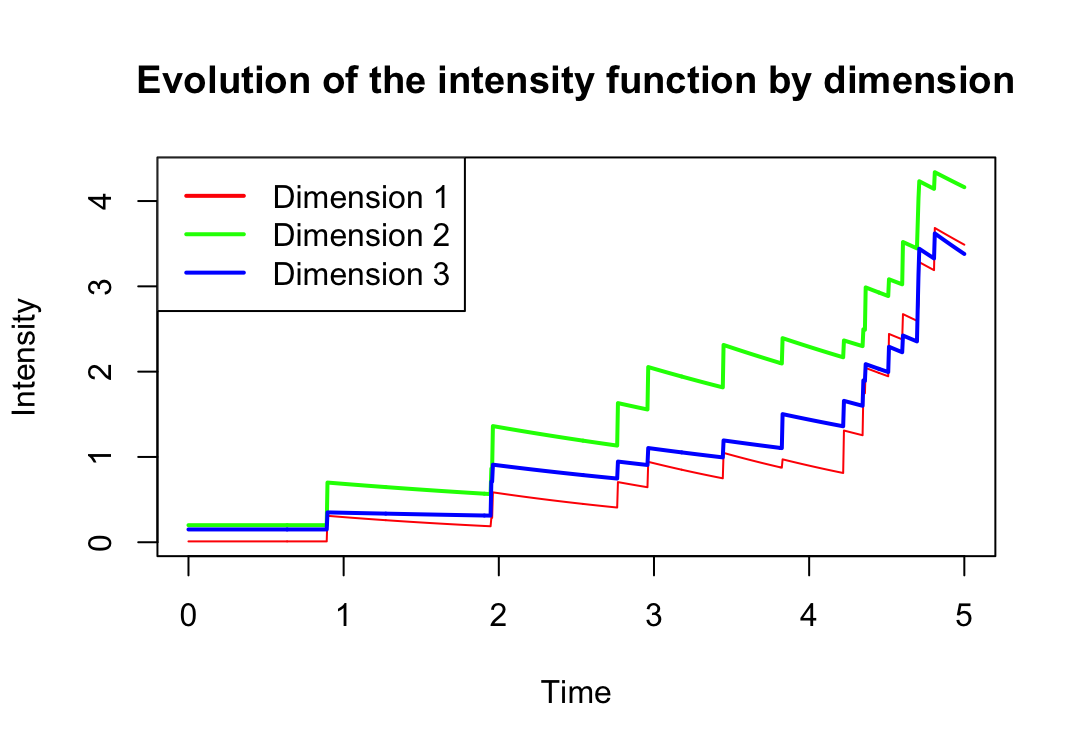}
\end{center}    \vspace*{-0.3cm}
    \caption{Trajectory of the counting process and evolution of the intensity function for the exponential excitation function with a constant baseline intensity (Case 1).}
    \label{fig:expo}
\end{figure}

\paragraph{Case 2: Gamma Excitation Function with Constant Baseline Intensity.}

In this case, the excitation function corresponds to a gamma distribution, and the baseline intensity remains constant. The intensity function $\lambda(t)$ is given by:
\[
\lambda(t) = \mu + \sum_{i=1}^n \alpha_i \left(\frac{(t - t_i)^{\gamma - 1}}{\Gamma(\gamma)}\right) e^{-\delta (t - t_i)^+}
\]
where $\mu$ is the baseline intensity, $\alpha_i$ is the excitation parameter for the $i$-th process, $\gamma$ is the shape parameter of the gamma distribution, $\delta$ is the rate parameter, and $(t - t_i)^+$ is the positive part function as described previously.
An example of trajectories is shown in Figure \ref{fig:gamma}.

\begin{figure}[h!]
    \centering
        \includegraphics[width=0.45\textwidth]{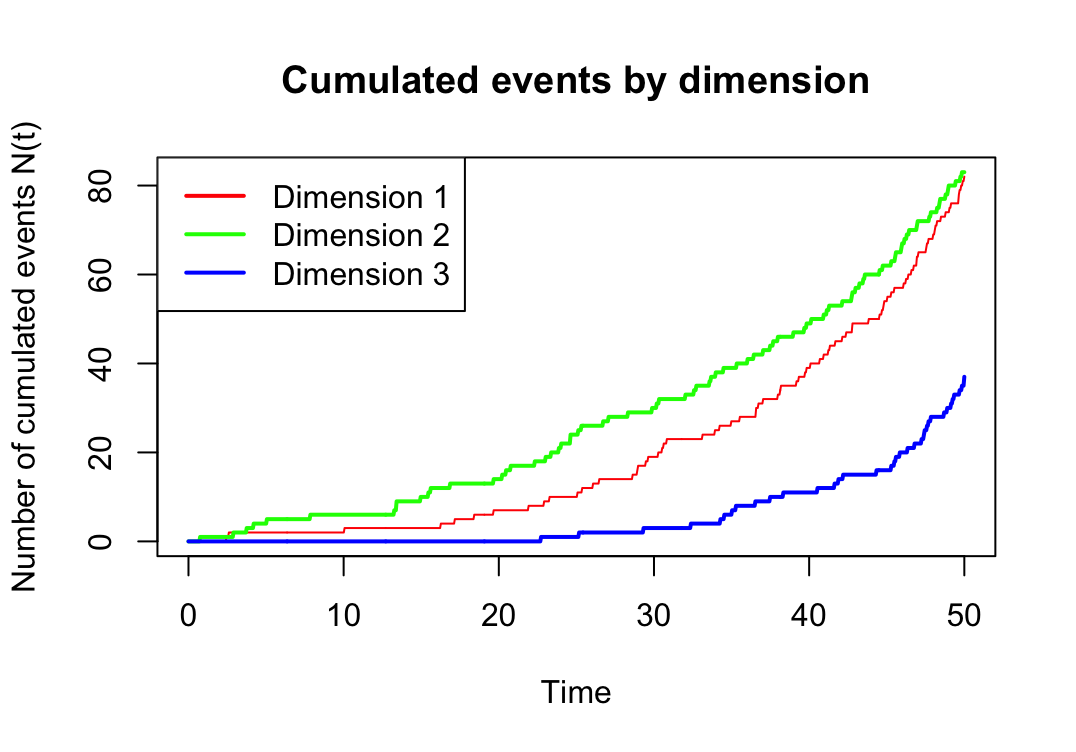}
        \includegraphics[width=0.45\textwidth]{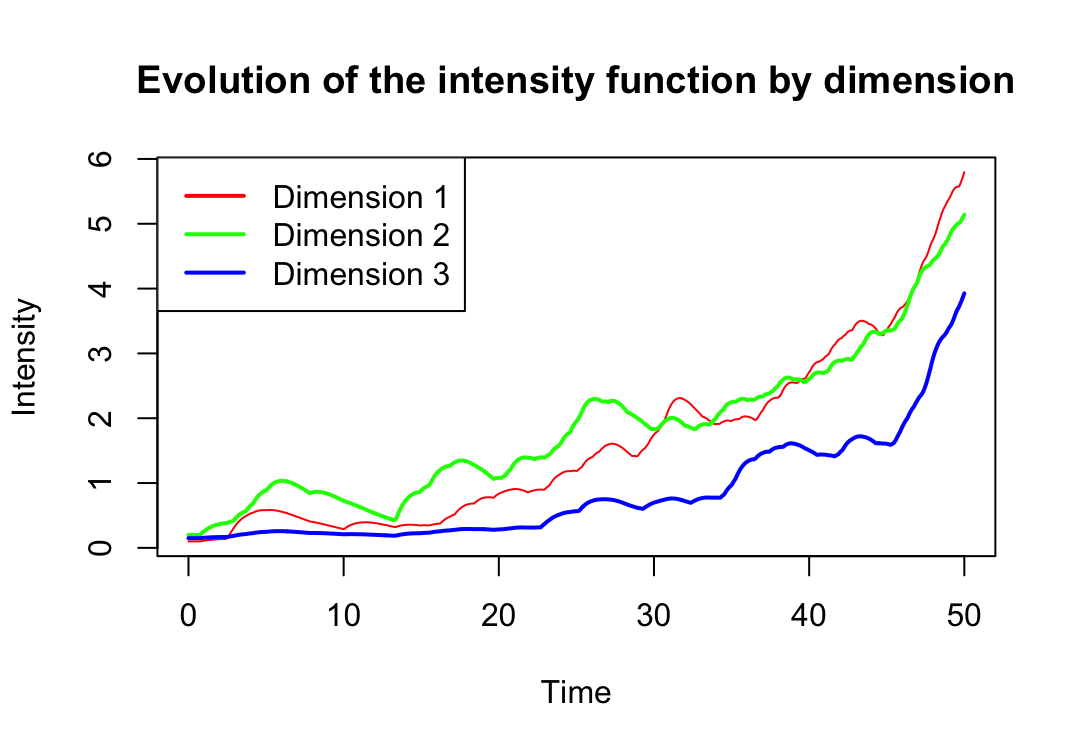}
    \caption{Trajectory of the counting process and evolution of the intensity function for the gamma excitation function with a constant baseline intensity (Case 2).}
    \label{fig:gamma}
\end{figure}

\paragraph{Case 3: Constant Excitation Function with Constant Baseline Intensity.}

In this case, the excitation function is constant, and the baseline intensity is also constant. The intensity function $\lambda(t)$ is given by:
\[
\lambda(t) = \mu + \sum_{i=1}^n \alpha_i
\]
where $\mu$ is the constant baseline intensity, and $\alpha_i$ is the constant excitation parameter for the $i$-th process.
Figure \ref{fig:const} gives a realization of trajectories in this case.

\begin{figure}[h!]
    \centering
        \includegraphics[width=0.45\textwidth]{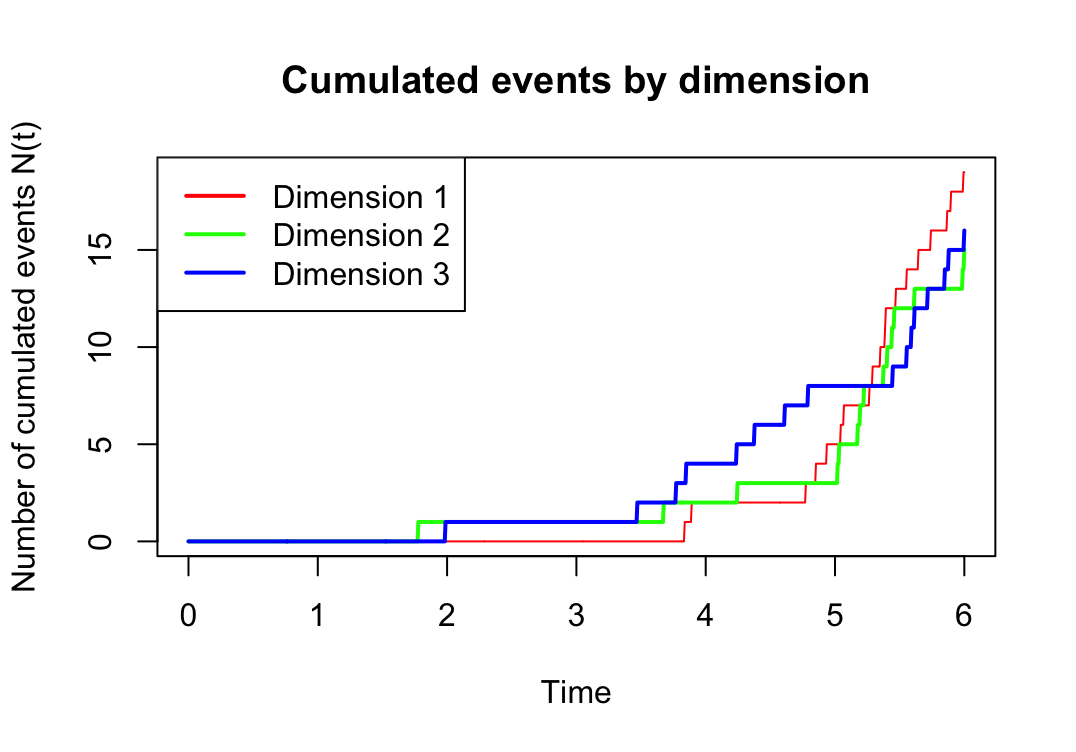}
        \includegraphics[width=0.45\textwidth]{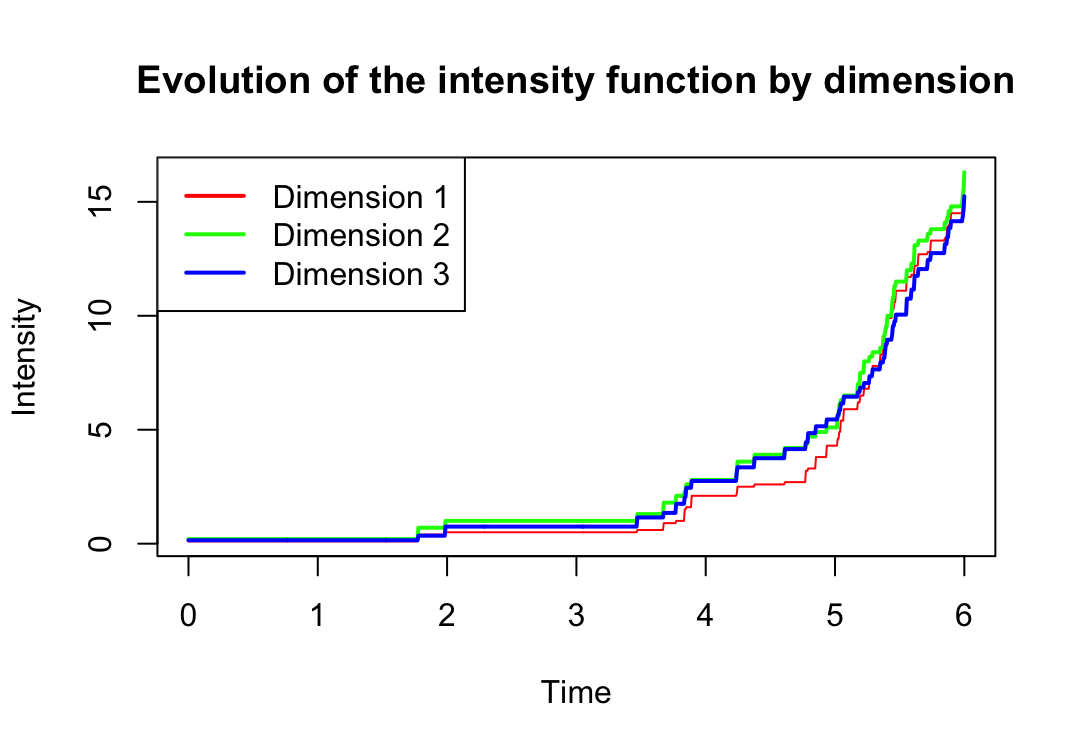}
    \caption{Trajectory of the counting process and evolution of the intensity function for the constant excitation function with a constant baseline intensity (Case 3).}
    \label{fig:const}
\end{figure}

\paragraph{Case 4: Beta Excitation Function with Constant Baseline Intensity.}

In this case, the excitation function corresponds to a Beta distribution defined on the interval $[0, 1]$ with two shape parameters labeled $\alpha$ and $\beta$ such that the probability density function of the Beta distribution is given by:
\[
f(x; \alpha, \beta) = \frac{x^{\alpha-1} (1 - x)^{\beta-1}}{B(\alpha, \beta)},
\]
where $B(\alpha, \beta)$ is the Beta function, defined as:
\[
B(\alpha, \beta) = \int_0^1 x^{\alpha-1} (1-x)^{\beta-1} dx .
\]
The intensity function $\lambda(t)$ in this case is given by:
\[
\lambda(t) = \mu + \sum_{i=1}^n \alpha_i \left(\frac{(t - t_i)}{(t - t_i + \delta)}\right)^{\gamma - 1} \left(1 - \frac{(t - t_i)}{(t - t_i + \delta)}\right)^{\beta - 1} ,
\]
where $\mu$ is the baseline intensity, $\alpha_i$ is the excitation parameter for the $i$-th process, and the term involving $(t - t_i)$ represents the Beta-shaped excitation function.
Figure \ref{fig:beta} shows a realization of the trajectory of the multivariate Hawkes Process in this case.

\begin{figure}[h!]
    \centering
        \includegraphics[width=0.45\textwidth]{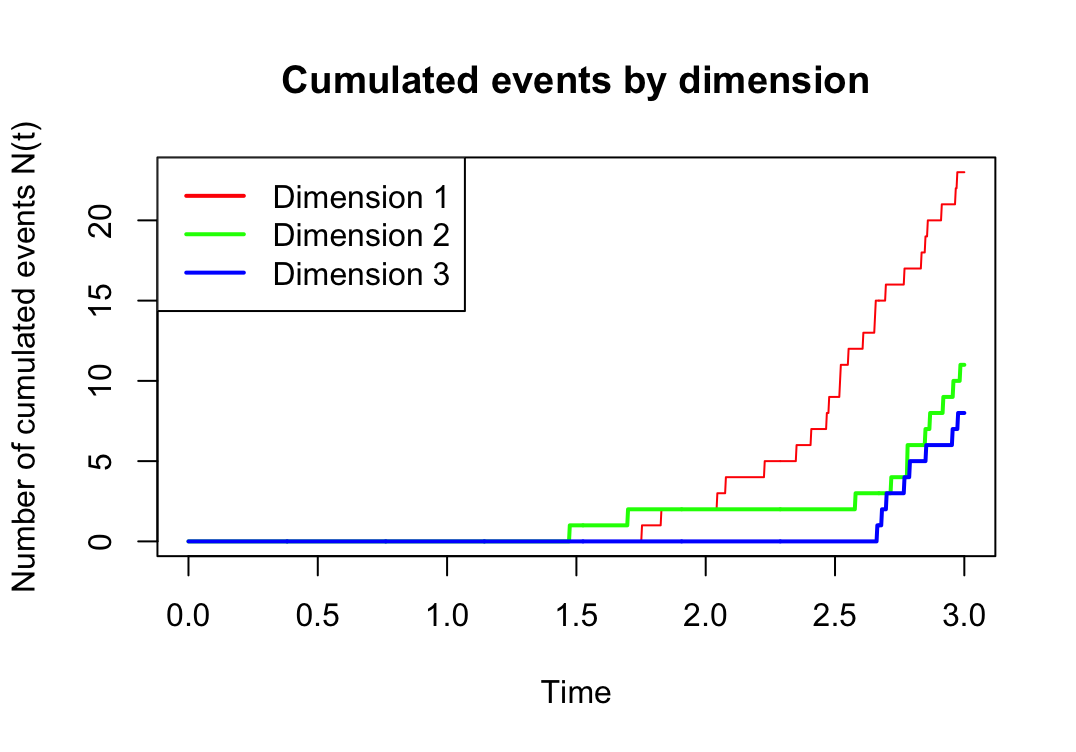}
        \includegraphics[width=0.45\textwidth]{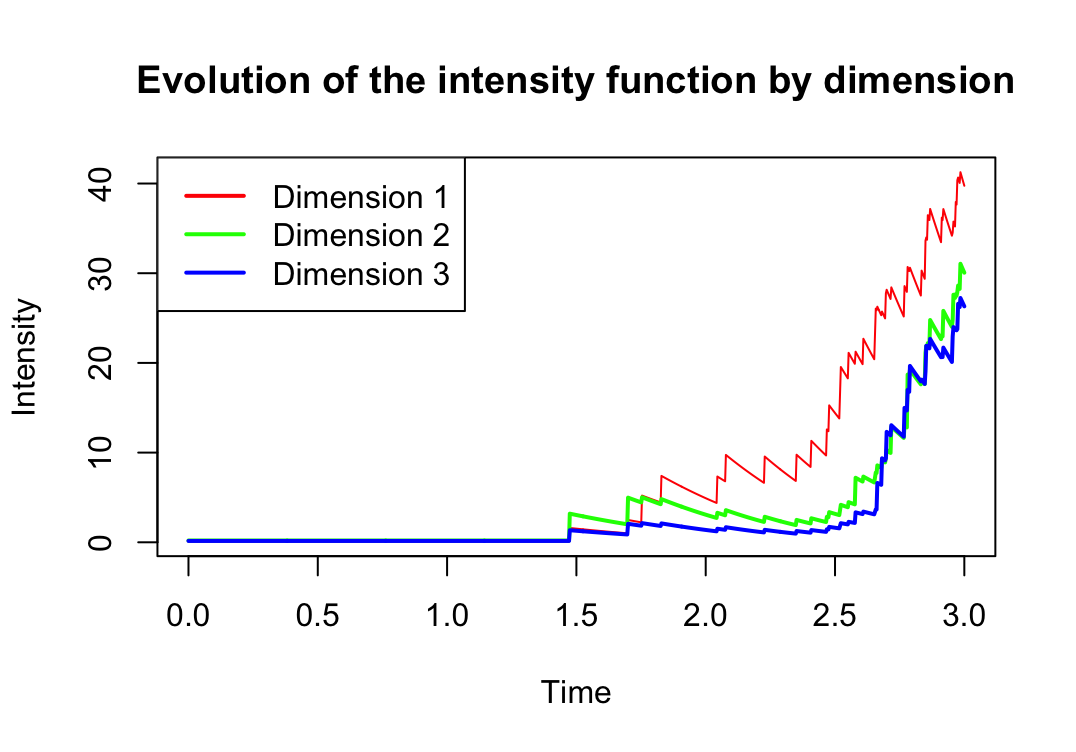}
    \caption{Trajectory of the counting process and evolution of the intensity function for the Beta excitation function with a constant baseline intensity (Case 4).}
    \label{fig:beta}
\end{figure}

Second, we simulated 4-dimensional realizations of the Hawkes process for varying interaction matrices between four entities in order to illustrate the effect of the interactions on the trajectory of the multi-dimensional counting process. We considered an intensity function with the form
$\lambda(t) = \mu + W \mathrm{diag}(\phi)$ with exponential $(\phi_i^i)_{i=1,\ldots,d}$. 
Figure \ref{fig:interaction} shows different behaviors obtained for various $W$ matrices and various ordering in parameter values. We clearly observe contrasted curve shapes consistent with input parameters.

\begin{figure}[h!]
    \begin{align*}
 W&=\left(
    \begin{array}{cccc}
    {\textcolor{red}{a}} & {\textcolor{red}{a}} & 0 & 0 \\
    {\textcolor{olive}{a}} & {\textcolor{olive}{a}} & 0 & 0 \\
    0 & 0 & {\textcolor{cyan}{b}} & 0 \\
    \tiny{0} & 0 & 0 & {\textcolor{violet}{b}} \\
    \end{array}\right) \hspace*{7cm}\\
\mu&=(c,c,c,c)\\
b&<a
\end{align*}
\vspace*{-4cm} \\ \hspace*{7cm}        \includegraphics[width=0.45\linewidth]{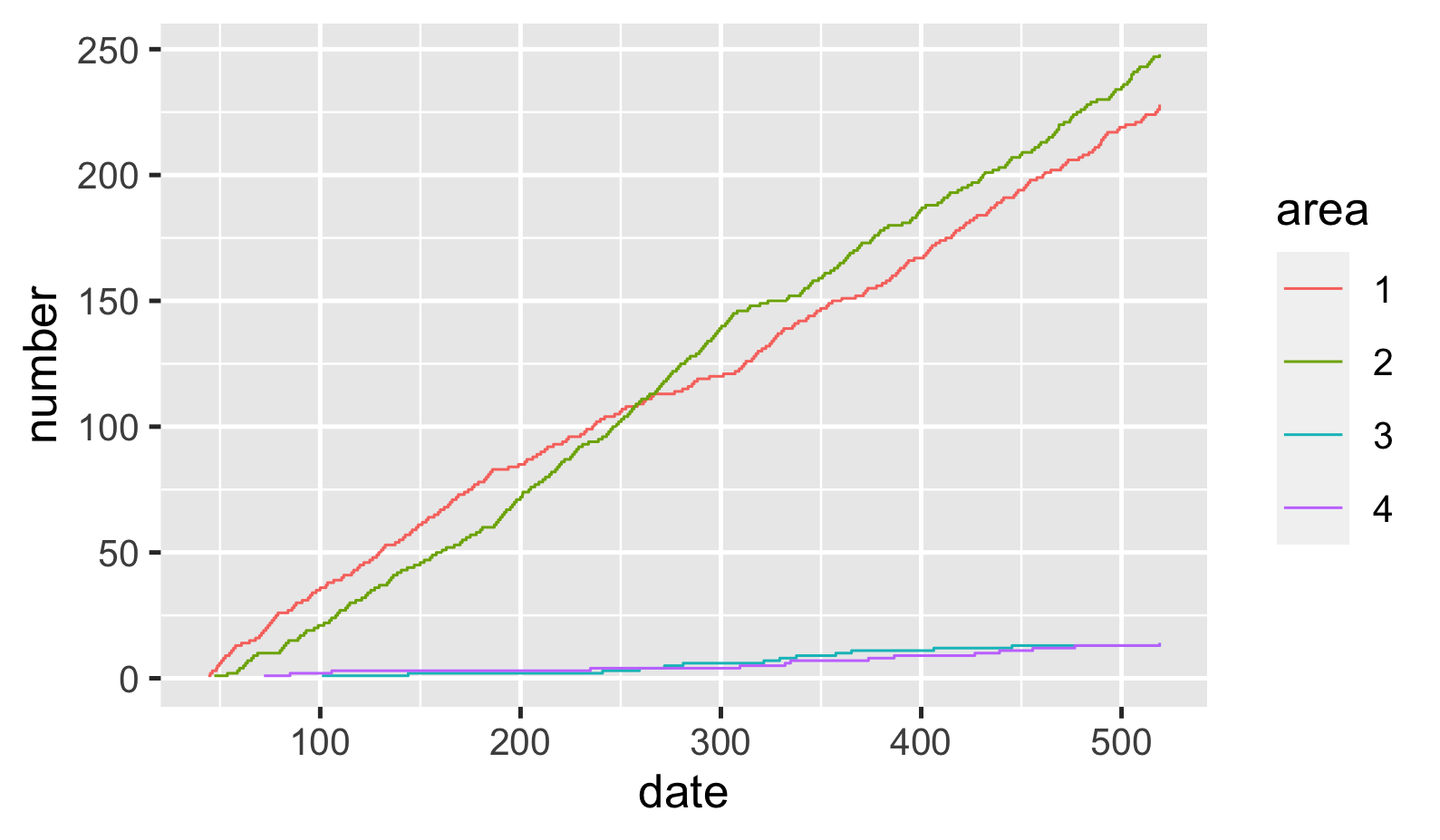}\\ \ \\
\hspace*{1cm}  \includegraphics[width=0.45\linewidth]{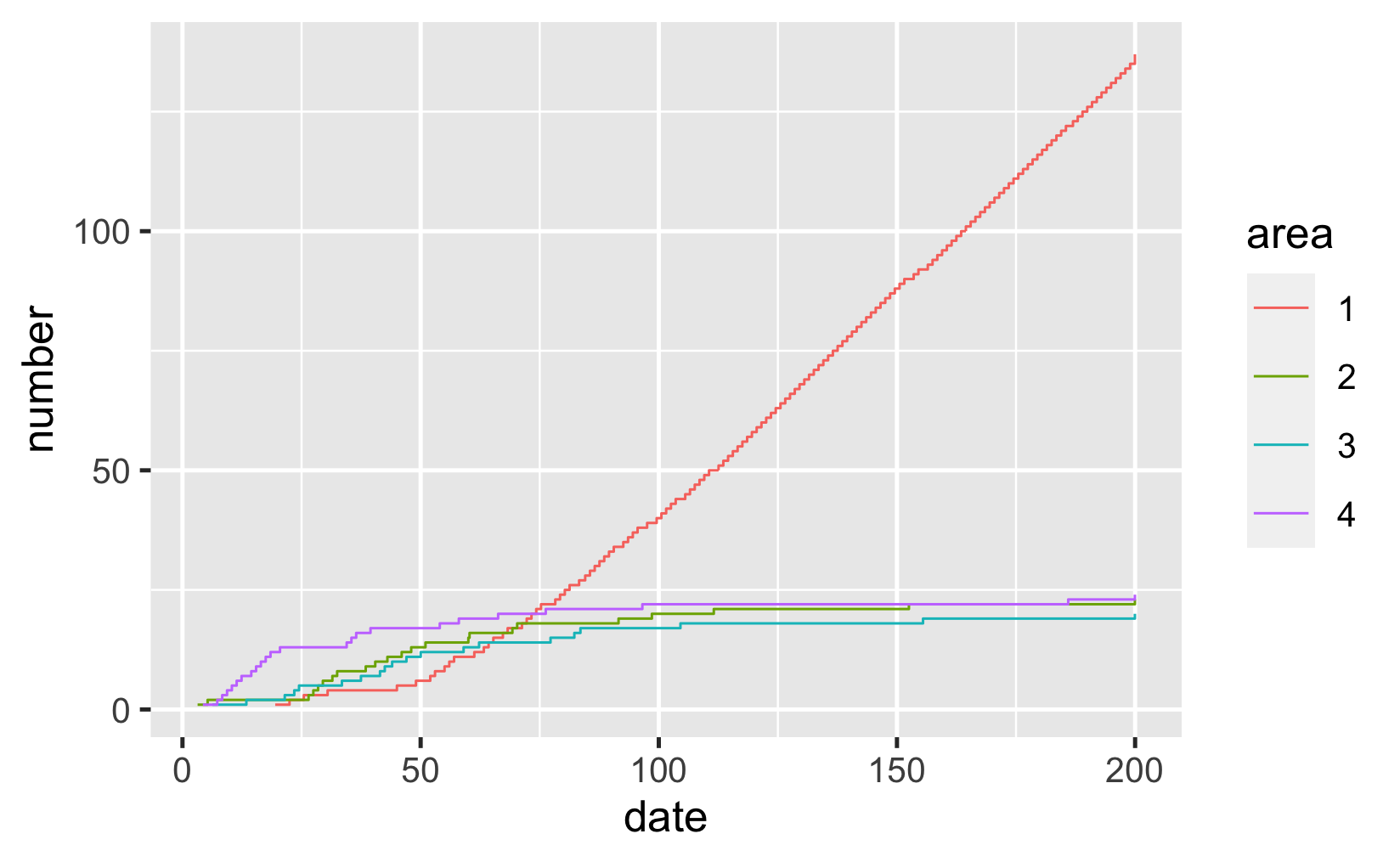}\\
\vspace*{-5cm} \\ 
 \begin{align*}
 \hspace*{7cm}   W&=\left(
    \begin{array}{cccc}
    {\textcolor{red}{a}} & {\textcolor{red}{a}} & {\textcolor{red}{a}} & {\textcolor{red}{a}} \\
    0 & {\textcolor{olive}{b}} & 0 & 0 \\
    0 & 0 & {\textcolor{cyan}{b}} & 0 \\
    0 & 0 & 0 & {{\textcolor{violet}{b}}} \\
    \end{array}\right)\\
\mu&=(c,C,C,C)\\
b&<a\;\; ; \;\;c<C
\end{align*}\\
\begin{align*}
W&=\left(
    \begin{array}{cccc}
    {\textcolor{red}{a}} & {\textcolor{red}{a}} & {\textcolor{red}{a}} & {\textcolor{red}{a}} \\
    0 & {\textcolor{olive}{b}} & 0 & 0 \\
    0 & 0 & {\textcolor{cyan}{b}} & 0 \\
    0 & 0 & {{\textcolor{violet}{b}}} & {{\textcolor{violet}{b}}} \\
    \end{array}\right)\hspace*{7cm}\\
\mu&=(c,C,C,C)\\
b&<a  \;\; ; \;\; c<C
\end{align*}
\vspace*{-4cm} \\ \hspace*{7cm}   
\includegraphics[width=0.45\linewidth]{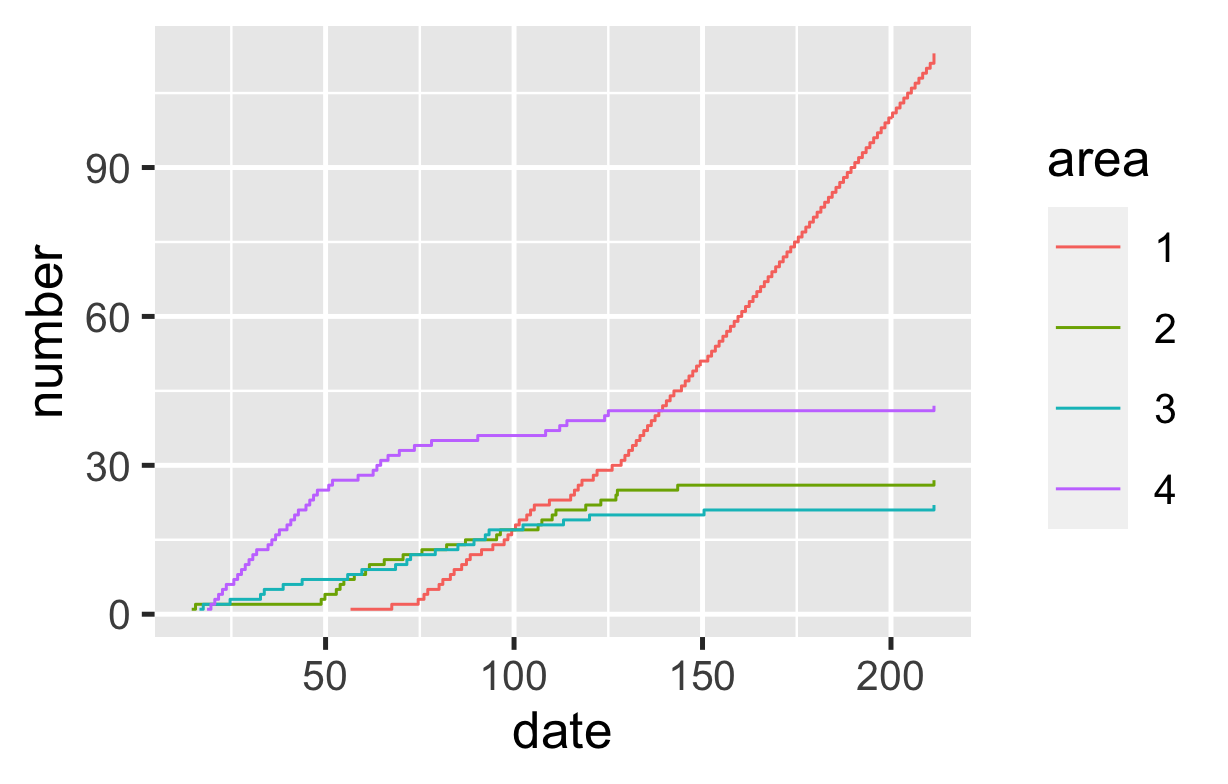} \\ \ \\
\hspace*{1cm}  \includegraphics[width=0.45\linewidth]{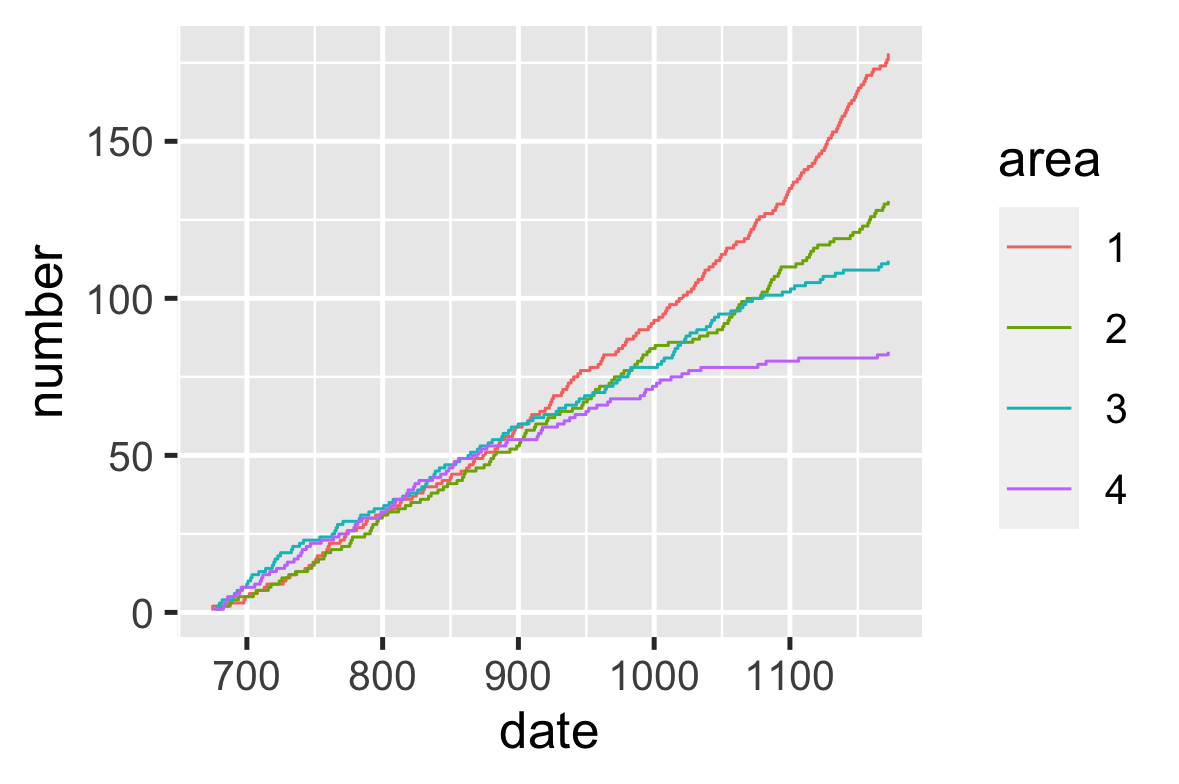}
 \\
\vspace*{-5cm} \\ 
 \begin{align*}
 \hspace*{7cm} 
W&=\left(
    \begin{array}{cccc}
    {\textcolor{red}{a}} & {\textcolor{red}{a}} & {\textcolor{red}{a}} & {\textcolor{red}{a}} \\
    0 & {\textcolor{olive}{b}} & {\textcolor{olive}{b}} & 0 \\
    0 & {\textcolor{cyan}{b}} & {\textcolor{cyan}{b}} & 0\\
    0 & 0 & 0 & {{\textcolor{violet}{b}}} \\
    \end{array}\right)\\
\mu&=(c,c,c,c)\\
b&<a 
\end{align*}
    \caption{Trajectory of the counting process for different interaction matrices and parameter ordering.}
    \label{fig:interaction}
\end{figure}

\section{Conclusion}\label{conclusion}

In this work, we proposed and characterized a self-exciting point process taking the form of a Hawkes process defined in a spatial context. The space may correspond to distinct geographic regions or any other entities, whose connections drive how much the process that is going on in a given entity generates offspring points in any other entity. Hence, this process may be applied to represent self- and inter-exciting dynamics at the nodes of any weighted-oriented network like those referred in \cite{choufany2021,richard2023}. 

In this article, we detailed the precise form of the infinitely divisible property associated with such a Hawkes process. We computed its multi-dimensional and multi-temporal characteristic function, which allows a thorough comprehension of the process dynamics across both time and space. The generalized Laplace Transform of the process was also described as well as its first two moment functions (mean and covariance structures). We were able to obtain closed-form formula for the characteristic and moment functions, and we proposed a numerical scheme for solving the equation that we obtained. 

As pointed out in the introduction, the main results in our article are the formulas concerning the covariance structure of the process and more generally the formulas that concern multiple times for multivariate, non-stationary Hawkes processes. These formulas and all the accompanying material provided in this article will allow, in further studies, the development of estimators for the parameters of our process and the description of its behavior. 

In terms of estimation, our spatial Hawkes process could be fitted to data collected during epidemics such as those caused by the phytopathogenic bacterium {\it Xylella fastidiosa}. The spatial Hawkes process that we proposed could allow us to draw inferences not simply based on a deterministic propagation model with a unique introduction of the disease (such assumptions were made in \cite{abboud2019,abboud2023}), but based on a stochastic model with multiple introductions driven by the baseline intensity function $\lambda^0$, as suggested by \cite{abboud2018}. The most trivial ways to make estimation for our spatial Hawkes process may consist in using either the moment functions at multiple times or the probability distribution of counts in the framework of the minimum contrast method \cite{dacunha1982,soubeyrand2009}. 

In terms of description of the model behavior, several objectives may be considered. Suppose that the model is used to describe epidemics in multiple regions, then properties linking the baseline intensity $\lambda^0$  with the overall epidemic size (typically the cumulative number of points at a given time) may be derived to select the $d'<d$ regions in which the baseline intensity should be reduced to minimize the overall epidemic size, given $d'$ and the reduction factor. Similarly, properties linking inter-region migration terms in $\phi$ with the overall epidemic size may be derived to select the $d'<d(d-1)$ migration terms should be reduced to minimize the overall epidemic size, given $d'$ and the reduction factor. Beyond the control of the epidemic size, another interesting issue concerns the surveillance of the epidemics, with two typical objectives: early detection \cite{martinetti2019} and disease delimitation \cite{charras2013}. For example, further studies may focus on deriving the properties of the model allowing us to determine how the sampling strategies for early detection or disease delimitation should be adapted to the baseline intensity function $\lambda^0$ or the excitation function $\phi$ (including the migration terms).

Finally, it would be interesting to explore how the approach and tools proposed here could be mobilized to further characterize multidimensional self-exciting processes with dependencies (MSPD), recently introduced by \citep{hillairet2025}. MSPDs generalize Hawkes processes by allowing the excitation function to depend on a mark associated with each point event. In an epidemiological context, such dependencies could reflect the influence of pathogen variants with differing transmission profiles or of infected hosts exhibiting heterogeneous contact patterns ---including the effect of super-spreaders.

\

\paragraph{Acknowledgements.}

This work was supported by an `INRAE MathNum -- R\'egion PACA' PhD grant, the BEYOND Project funded by  ANR (grant ANR-20-PCPA-0002) and the BeXyl Project funded by the HORIZON.2.6 programme (grant 101060593).

\bibliography{biblio.bib}
\bibliographystyle{amsplain}

\end{document}